\documentclass[11pt]{amsart}
\usepackage[margin=1.5in]{geometry}
\usepackage{graphicx}
\usepackage{amssymb, amsthm, amsfonts}
\usepackage[all]{xy}
\usepackage[dvipsnames]{xcolor}
\usepackage{comment}
\usepackage{pinlabel}
\usepackage[inline]{enumitem}
\usepackage{todonotes}
\usepackage{thm-restate}
\usepackage{verbatim}
\usepackage{subfigure}
\usepackage{tikz}
\usepackage{soul}
\usepackage{latexsym}
\usepackage{textgreek}
\usepackage[unicode]{hyperref}
\usepackage[nameinlink]{cleveref}
\usepackage{cite}

\usepackage[abs]{overpic}	
\usepackage{xcolor}
\definecolor{lred}{RGB}{226, 106, 106}
\definecolor{lblue}{RGB}{52, 152, 219}
\definecolor{lyellow}{RGB}{232, 197, 91}
\definecolor{lgrey}{RGB}{199, 195, 189}	
\definecolor{lgreen}{RGB}{154, 188, 135}	
\definecolor{lpurple}{RGB}{176, 156, 255}	
\definecolor{lorange}{RGB}{255, 170, 0}

\hypersetup{colorlinks=true,linkcolor=teal,citecolor=purple}

\newtheorem{theorem}{Theorem}[section]
\newtheorem{lemma}[theorem]{Lemma}
\newtheorem{proposition}[theorem]{Proposition}

\theoremstyle{definition}
\newtheorem{definition}[theorem]{Definition}
\newtheorem{remark}[theorem]{Remark}
\newtheorem{example}[theorem]{Example}

\newtheorem{construction}[theorem]{Construction}

\newcommand{\R}{{\ensuremath{\mathbb{R}}}}
\newcommand{\N}{{\ensuremath{\mathbb{N}}}}
\newcommand{\Z}{{\ensuremath{\mathbb{Z}}}}
\newcommand{\Q}{{\ensuremath{\mathbb{Q}}}}
\newcommand{\F}{{\ensuremath{\mathbb{F}}}}
\newcommand{\mfs}{\mathfrak{s}}

\newcommand{\mft}{\mathfrak{t}}
\newcommand{\mfu}{\mathfrak{u}}
\newcommand{\lk}{\operatorname{lk}}
\newcommand{\cptwo}{\mathbb{C} P^2}
\newcommand{\cptwobar}{\overline{\cptwo}}
\newcommand{\adam}{\alpha}
\newcommand{\adamhat}{\widehat{\alpha}}

\newcommand{\spincs}{\mathfrak{s}}
\newcommand{\spinct}{\mathfrak{t}}
\newcommand{\spincu}{\mathfrak{u}}

\def\co{\colon\thinspace}
\def\minus{\smallsetminus}

\def\CFK{\mathit{CFK}}

\def\CFKm{\mathit{CFK}^-}
\def\HF{\mathit{HF}}
\def\HFh{\widehat{\mathit{HF}}}
\def\HFp{\mathit{HF}^+}
\def\HFi{\mathit{HF}^\infty}
\def\HFm{\mathit{HF}^-}
\def\HFmc{\mathbf{HF}^-}
\def\HFmr{\mathit{HF_{red}^-}}
\def\HFred{\mathit{HF_{red}}}
\def\CFD{\widehat{\mathit{CFD}}}
\def\CFA{\widehat{\mathit{CFA}}}
\def\CFDD{\widehat{\mathit{CFDD}}}
\def\HFKh{\widehat{\mathit{HFK}}}

\def\X{\mathcal{X}}
\def\E{\mathcal{E}}
\def\W{\mathcal{W}}
\def\RR{\mathcal{R}}

\def\S{\mathcal{S}}

\def\AA{\mathcal{A}}

\def\conn{\mathbin{\#}}

\newcommand{\abs}[1] {\left\lvert #1 \right\rvert}
\newcommand{\gen}[1] {\left\langle #1 \right\rangle}
\newcommand{\mat}[1] {\begin{bmatrix} #1 \end{bmatrix}}

  \DeclareMathOperator{\nbd}{nbd}
\DeclareMathOperator{\PD}{PD}  \DeclareMathOperator{\im}{im}
\DeclareMathOperator{\gr}{gr}  \DeclareMathOperator{\Span}{Span}
\DeclareMathOperator{\Spin}{Spin} \DeclareMathOperator{\pt}{pt}

\author[Adam Simon Levine]{Adam Simon Levine}
\thanks{ASL was supported in part by NSF grants DMS-1806437 and DMS-2203860.}
\address{Department of Mathematics, Duke University, Durham, NC 27705}
\email{alevine@math.duke.edu}

\author[Tye Lidman]{Tye Lidman}
\thanks{TL was supported in part by NSF grants DMS-1709702, DMS-2105469, and a Sloan Fellowship.}
\address{Department of Mathematics, North Carolina State University, Raleigh, NC 27607}
\email{tlid@math.ncsu.edu}

\author[Lisa Piccirillo]{Lisa Piccirillo}
\thanks{LP\ was supported in part by a Sloan Fellowship, a Clay Fellowship, and the Simons collaboration ``New structures in low-dimensional topology.'' LP\ thanks the NCCR SwissMAP of the Swiss National Science Foundation for their hospitality during a portion of this project.}
\address{Department of Mathematics, Massachusetts Institute of Technology, Cambridge, MA 02139}
\email{piccirli@mit.edu}

\title{New constructions and invariants of closed exotic 4-manifolds}

\numberwithin{equation}{section}

\begin{document}
\begin{abstract}
In this article, we give new means of constructing and distinguishing closed exotic four-manifolds.  Using Heegaard Floer homology, we define new closed four-manifold invariants that are distinct from the Seiberg--Witten and Bauer--Furuta invariants and can remain distinct in covers. Our constructions include exotic definite manifolds with fundamental group $\Z/2$,  infinite families of exotic manifolds that are related by knot surgeries on Alexander polynomial 1 knots,  and exotic manifolds that contain square-zero spheres.
\end{abstract}

\maketitle
\vspace{-5pt}
\section{Introduction}\label{sec:intro}

\vspace{-5pt}\noindent The first exotic smooth structures on closed, oriented 4-manifolds were demonstrated by Donaldson in 1987 \cite{Donaldson}, who used  $SU(2)$ gauge theory to show that there is a Dolgachev surface which is an exotic copy of $\cptwo\conn 9\cptwobar$.  In 1989,  Kotschick \cite{Kotschick} produced simply connected exotica with smaller $b_2$ by showing  that the Barlow surface \cite{Barlow} is an exotic $\cptwo\conn 8\cptwobar$.  In the mid-nineties,  techniques for building infinitely many smooth structures on these manifolds were developed by Fintushel, Friedman, Stern and Szab\'o in \cite{FintushelStern,Friedman,Szabo}.
The development of exotic smooth structures on simply connected  4-manifolds with even smaller $b_2$ was however stalled for many years due to a lack of examples. Then, in 2005, J. Park \cite{Park} used Fintushel--Stern's rational blowdown techniques \cite{FSblowdown} to construct an exotic smooth structure on $\cptwo\conn 7\cptwobar$.  This jump-started the industry; using (generalizations of) Park's techniques exotic smooth structures were found on $\cptwo\conn 6\cptwobar$ \cite{StipsiczSzabo},  $\cptwo\conn 5\cptwobar$ \cite{ParkStipsiczSzabo}, $\cptwo\conn 3\cptwobar$ \cite{AP3},  and finally in 2010 $\cptwo\conn 2\cptwobar$ \cite{AP2}.  There is a largely parallel history for producing small simply connected exotic manifolds with $b_2^+=3$, see \cite{AP3} for details.  All of these theorems rely on gauge theory to provide the diffeomorphism obstruction.  The development of exotic smooth structures on still smaller simply connected closed 4-manifolds has again been stalled for the ensuing years due to a lack of examples.

The aim of the present paper is to introduce a new perspective on closed, oriented 4-manifolds.  We will build up our smooth 4-manifolds explicitly out of 2-handles (in contrast to the prior literature, which built examples by cutting down elliptic surfaces).  Our diffeomorphism obstructions will come from Heegaard Floer homology, and will be explicitly computed handle by handle (in contrast to most computations of gauge theoretic invariants, which are via cut-and-paste formulae).  We produce, among other things, an exotic smooth structure on $\cptwo\conn 9\cptwobar$ (which we expect is not diffeomorphic to Donaldson's original).  Our hope is that the techniques we develop here can simplify and eventually advance  the search for smaller closed exotica.

Computing closed 4-manifold invariants by breaking 4-manifolds into simpler pieces and computing the associated invariants was in fact one of the original major goals of the development of the Heegaard Floer package (and Seiberg--Witten Floer homologies).  However, to present, most applications and computations of Floer homology have been in other settings (e.g. Dehn surgery, knot concordance, contact topology).  Our approach gives new evidence that these tools are indeed well-suited for problems in closed 4-manifold topology, as originally intended.

We already demonstrate several new exotic phenomena as a consequence of our straightforward technology.  Our exotic  $\cptwo\conn 9\cptwobar$  has a free involution; taking the quotient yields the first exotic definite closed oriented 4-manifold (Theorem~\ref{thm:main}).  We also develop a new elementary invariant for 4-manifolds with $b_1>0$ (Definition~\ref{def:adam-intro}), which allows us to distinguish exotic 4-manifolds produced by Fintushel--Stern knot surgery using an Alexander polynomial 1 knot (Theorem~\ref{thm:knot-surgery-intro}). We also demonstrate that this new invariant does not neccesarily vanish on 4-manifolds with embedded homologically essential spheres with self intersection zero, in contrast to the Bauer--Furuta, Seiberg--Witten,  and Donaldson invariants (Theorem \ref{thm:zerospheresintro}).

\subsection{Constructions of closed exotica}

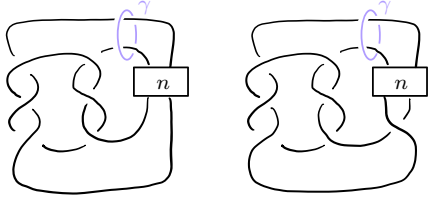
\begin{figure}[htbp]{\scriptsize
\begin{overpic}[tics=20]{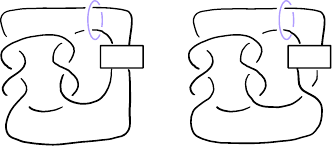}
 \put(147, 40){$n$}
  \put(57, 40){$n$}
    \put(50, 70){$\color{lpurple}\gamma$}
        \put(141, 70){$\color{lpurple}\gamma$}
  \end{overpic}}
  \caption{Pretzel knots $Q_n$ (left) and $P_n$ (right). }
  \label{fig:PQ}
\end{figure}

Part of the appeal of our approach is that our exotic 4-manifolds can be described explicitly in elementary terms, using certain $2$-handle cobordisms between surgeries on pretzel knots.
Specifically, let $P_n = P(3,-3,2n-1)$ and $Q_n = P(3,-3,2n)$, each of which is a ribbon knot. In particular, $P_0$ is the stevedore's knot and $Q_0$ is the square knot. (See Figure \ref{fig:PQ}.) We show in Proposition \ref{prop:PQdual} that for any $n$, there are homeomorphisms
\[
\phi_n \co  S^3_1(P_n) \to S^3_{-1}(Q_{n-4})  \quad \text{and} \quad \psi_n \co  S^3_1(Q_n) \to S^3_{-1}(P_{n-3}).
\]

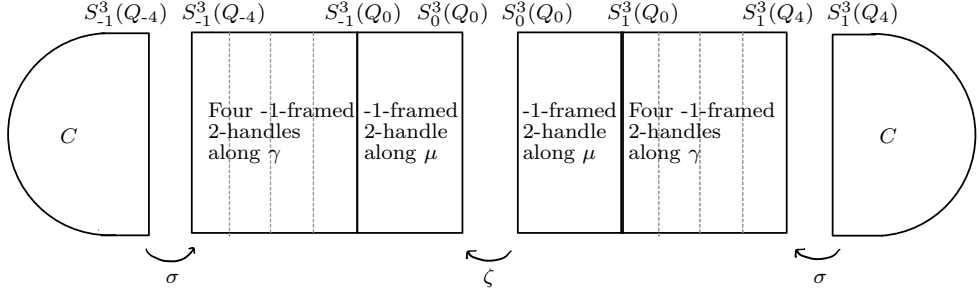
\begin{figure}[htbp]{\scriptsize
\vspace{8pt}
\begin{overpic}[tics=20]{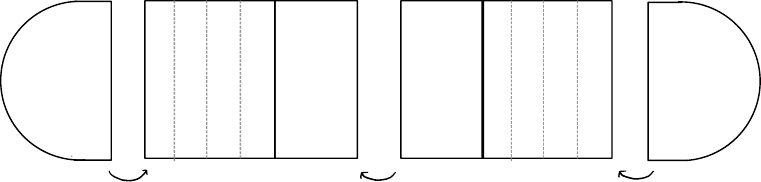}
    \put(29, 92){$S^3_{\text{-} 1}(Q_{\text{-} 4})$}
        \put(66, 92){$S^3_{\text{-} 1}(Q_{\text{-} 4})$}
        \put(120, 92){$S^3_{\text{-} 1}(Q_{0})$}
                \put(155, 92){$S^3_{0}(Q_{0})$}
                                \put(187, 92){$S^3_{0}(Q_{0})$}
    \put(227, 92){$S^3_{1}(Q_{0})$}
        \put(278, 92){$S^3_{1}(Q_{4})$}
        \put(310, 92){$S^3_{1}(Q_{4})$}
                \put(20, 45){$C$}
                        \put(330, 45){$C$}
\put(60, -8){$\sigma$}
\put(305, -8){$\sigma$}
\put(180, -8){$\zeta$}
     \put(76, 55){Four -$1$-framed}
       \put(76, 47){2-handles}
              \put(76, 39){along $\gamma$}
     \put(235, 55){Four -$1$-framed}
       \put(235, 47){2-handles}
              \put(235, 39){along $\gamma$}
     \put(135, 55){-$1$-framed}
       \put(135, 47){2-handle}
              \put(135, 39){along $\mu$}
     \put(195, 55){-$1$-framed}
       \put(195, 47){2-handle}
              \put(195, 39){along $\mu$}
  \end{overpic}}
  \caption{The exotic manifold $\E$ from Construction \ref{constr:simplyconn-intro}. The left half of this figure is $V$. The right hand half is also $V$, but it has been turned upside down.
  }\label{fig:simplyconnintro}
  \end{figure}

\begin{construction} \label{constr:simplyconn-intro}
Let $C$ be the contractible manifold obtained by $-1$ surgery along a ribbon disk for $Q_{-4}$, with boundary $S^3_{-1}(Q_{-4})$, which can be identified with $S^3_1(P_0)$. Via work of Dai, Hedden, and Mallick \cite{DHMCorks}, there is an involution $\sigma$ on $S^3_{-1}(Q_{-4})$ that does not extend over $C$. Now, let $V'$ be the manifold obtained from $C$ by attaching four $-1$-framed $2$-handles along the meridian $\gamma$ of the ribbon band of $Q_{-4}$ and one $-1$-framed $2$-handle along a meridian $\mu$ of $Q_{-4}$; see Figure \ref{fig:simplyconnintro}. The boundary of $V'$ is then identified with $S^3_0(Q_0)$. Let $V$ be obtained from $V'$ by cutting out $C$ and regluing by $\sigma$. Finally, let $\E$ be the manifold obtained by gluing together two copies of $V$ using an orientation-reversing, free involution $\zeta$ of $S^3_0(Q_0)$; such an involution exists because $Q_0$ is an amphichiral knot. This manifold is oriented and simply connected, and it admits a free involution that exchanges the two halves; the quotient $\RR$ is then an oriented manifold with $\pi_1(\RR) = \Z/2$.
\end{construction}

\begin{theorem}\label{thm:main}
$ \ $
\begin{enumerate}
\item The manifold $\E$ is homeomorphic but not diffeomorphic to $\cptwo \conn 9 \cptwobar$.

\item The manifold $\RR$ is homeomorphic but not diffeomorphic to $Z_0 \conn 4 \cptwobar$, where $Z_0$ is a certain rational homology 4-sphere with $\pi_1(Z_0) \cong \Z/2$.
\end{enumerate}
\end{theorem}

To our knowledge, $\RR$ is the first example of an exotic, closed, orientable 4-manifold with definite intersection form.

\begin{remark}
We will see in Section \ref{sec:simply-connected} that if we instead take the union $V' \cup V'$, the resulting manifold is diffeomorphic to $\cptwo \conn 9 \cptwobar$, and the quotient of the corresponding involution is diffeomorphic to $Z_0 \conn 4 \cptwobar$.
\end{remark}

Next, we consider exotic 4-manifolds with positive first Betti number.

\begin{construction} \label{constr:pi1Z-intro}
\begin{figure}[htbp]{\scriptsize
\begin{overpic}[tics=20]{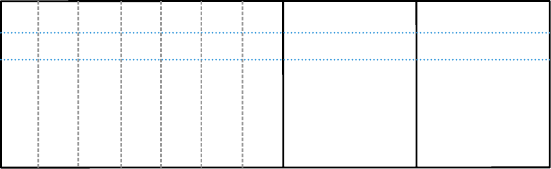}
    \put(-10, 85){$S^3_{-1}(Q_{-4})$}
        \put(125, 85){$S^3_{-1}(Q_{3})$}
         \put(185, 85){$S^3_{1}(Q_{3})$}
                \put(185, -10){$S^3_{-1}(P_{0})$}
                                \put(252, -10){$S^3_{1}(P_{0})$}
                                  \put(252, 85){$S^3_{-1}(Q_{-4})$}
        \put(-10, 60){\color{lblue}$\gamma$}
     \put(30, 35){Seven -$1$-framed}
       \put(30, 28){2-handles along $\gamma$}
     \put(140, 35){Two -$1$-framed}
       \put(140, 28){2-handles}
              \put(140, 19){along $\mu$}
                   \put(204, 35){Two -$1$-framed}
       \put(204, 28){2-handles}
              \put(204, 19){along $\mu_{P_0}$}
  \end{overpic}}
  \caption{The cobordism $\W$ from Construction \ref{constr:pi1Z-intro}.}\label{fig:pi1Zintro}
  \end{figure}
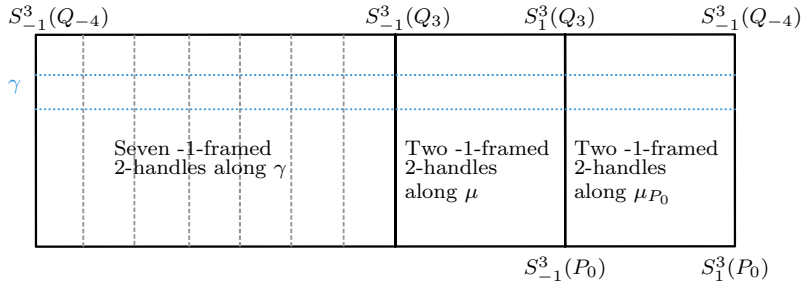

Let $\W$ denote the cobordism from $S^3_{-1}(Q_{-4})$ to itself that is built from a composition of eleven 2-handle attachments as follows (see Figure \ref{fig:pi1Zintro}).
%\begin{equation} \label{eq:self-cobordism-intro}
%\xymatrix@R=0.1in@C=0.2in{
%S^3_{-1}(Q_{-4}) \ar[r] &
%S^3_{-1}(Q_{-3}) \ar[r] &
%\cdots \ar[r] &
%S^3_{-1}(Q_{2}) \ar[r] &
%S^3_{-1}(Q_{3}) \ar[dr] \\
%&&&&& S^3_0(Q_{3}) \ar[dl] \\
%S^3_{-1}(Q_{-4}) &
%S^3_1(P_0) \ar@{=}[l]_-{\phi_0} &
%S^3_0(P_0) \ar[l] &
%S^3_{-1}(P_0) \ar[l] &
%S^3_1(Q_{3}) \ar@{=}[l]_-{\psi_3} &
%}
%\end{equation}
%Each arrow represents a $2$-handle attachment along either a meridian of the ribbon band of $Q_n$ or along a meridian of $Q_3$ or $P_0$.
We first attach seven $-1$-framed $2$-handles along parallel copies of the curve $\gamma$  and two $-1$-framed $2$-handles along copies of the meridian, which produces a cobordism from $S^3_{-1}(Q_{-4})$ to $S^3_{1}(Q_3)$. We identify the latter with $S^3_{-1}(P_0)$ using the homeomorphism $\psi_3$. Attaching two more $2$-handles along meridians of $P_0$ gives a cobordism to $S^3_1(P_0)$, which in turn is identified with $S^3_{-1}(Q_{-4})$ via $\phi_0$.

The cobordism $\W$ is simply connected and has $b_2^+(\W)=2$ and $b_2^-(\W)=9$. Let $\X$ be the closed manifold obtained by gluing the ends of $\W$ by the identity map of $S^3_{-1}(Q_{-4})$; this manifold has $\pi_1(\X) \cong \Z$. More generally, for each $p \ge 1$, let $\X_p$ be the manifold obtained by stacking $p$ copies of $\W$ and gluing the ends.
\end{construction}

\begin{theorem} \label{thm:pi1Z-intro}
For each $p \ge 1$, the manifold $\X_p$ is homeomorphic, but not diffeomorphic, to $S^1 \times S^3 \conn 2p \cptwo \conn 9p \cptwobar$.
\end{theorem}

We now describe several additional exotic phenomena that arise out of Construction \ref{constr:pi1Z-intro}.

First, the self cobordism $\W$ contains an embedded annulus with boundary the copies of $\gamma$ in either copy of $S^3_{-1}(Q_{-4})$, indicated by the horizontal dashed lines in Figure \ref{fig:pi1Zintro}. The ends can be glued together such that this annulus closes up to an embedded, square-zero torus $T_p$ in $\X_p$, which is contained inside a fishtail neighborhood. For any knot $K \subset S^3$, we define $\X_p^K$ to be the result of a particular Fintushel--Stern knot surgery \cite{FintushelStern} on $T_p$ (see Proposition \ref{prop:fishtail} for the definition).
%deleting a neighborhood of $T_p$ and gluing in $(S^3 - \operatorname{nbd}(K)) \times S^1$ in a particular way.
This operation does not change the homeomorphism type of these manifolds.

\begin{theorem}\label{thm:knot-surgery-intro}
For any $p\in\N$ and any non-trivial knot $K \subset S^3$, $\X_p^K$ is an exotic $S^1 \times S^3 \conn 2p \cptwo \conn 9p \cptwobar$ and is not diffeomorphic to $\X_p$. Moreover, if $K$ and $K'$ are knots with $\dim \HFKh(K) \ne \dim \HFKh(K')$, then $\X_p^K$ and $\X_p^{K'}$ are not diffeomorphic.
\end{theorem}

\begin{remark}
In particular, by performing knot surgery on $T$ using Whitehead doubles \cite{HeddenWhitehead}, we find an infinite family of exotic 4-manifolds related by Fintushel--Stern knot surgery using knots with Alexander polynomial 1.
Fintushel and Stern showed in some settings that the knot surgery operation changes the Seiberg--Witten invariants by multiplication by the Alexander polynomial of $K$.  (See also Mark \cite{MarkSurfaces} for a similar result for the Ozsv\'ath--Szab\'o 4-manifold invariant.)  
 It was previously unknown whether knot surgery with Alexander polynomial 1 knots can produce exotica.
\end{remark}

Second, we may produce additional exotica using a ``circle sum'' operation. For any $3$-manifold $M$, let $\X_p(M)$ denote the manifold obtained from $\X_p(M)$ by deleting a neighborhood of a circle generating $\pi_1(\X_p)$, and gluing in $(M - B^3) \times S^1$ (see Construction \ref{constr:circlesum} for the gluing map).

\begin{theorem}\label{thm:zerospheresintro}
Let $M$ be a 3-manifold with $b_1 = 0$,  or $S^2 \times S^1$, or $T^3$.  Then $\X_p(M)$ is an exotic $(M \times S^1) \conn 2p \cptwo \conn 9p \cptwobar$.  In particular, if $M = S^2 \times S^1$, then both $\X_p(M)$ and $(M \times S^1) \conn 2p \cptwo \conn 9p \cptwobar$ contain embedded homologically essential spheres with self intersection zero; thus, the Bauer--Furuta, Seiberg--Witten, and Donaldson invariants of all of these 4-manifolds vanish, as do those of all finite covers.
\end{theorem}

Finally, by inserting copies of the cobordism $\W$ from Construction \ref{constr:pi1Z-intro} into the manifold $V$ from Construction \ref{constr:simplyconn-intro}, we prove:
\begin{theorem}\label{thm:unlabeledthmintro}
For any $p \ge 0$, there exists an exotic $\conn (2p+1) \cptwo \conn 9(p+1) \cptwobar$.
\end{theorem}
These phemonena are not new, but we use this as a means of demonstrating the versatility of our approach.

\subsection{New invariants for smooth 4-manifolds}

The proofs of Theorems \ref{thm:main} and \ref{thm:unlabeledthmintro} rely on a computation of the Ozsv\'ath--Szab\'o closed 4-manifold invariant \cite{OSz4Manifold, OSzSymplectic}, which is essentially the Heegaard Floer analogue of the Seiberg--Witten invariant. However, the remaining theorems rely on a new and simpler invariant of closed 4-manifolds with $b_1 > 0$, which we now describe.  The construction is flexible, so we define it in a broad setting.   To keep the introduction simple, assume that $b_1 = 1$ and let $F$ denote a Floer homology theory which assigns a finite-dimensional vector space to each closed, oriented 3-manifold, e.g. $\HFh$, $\HFred$, $I^\sharp$, etc.

\begin{definition}\label{def:adam-intro}
Fix a finite-dimensional Floer homology theory $F$.  Let $X$ be a closed, oriented 4-manifold with $b_1(X) = 1$.  Define $\adam(X)$ to be the minimum dimension of $F(Y)$,  where $Y$ is a connected, non-separating hypersurface in $X$.
\end{definition}

For manifolds with $b_1(X)>1$, this invariant depends on a choice of primitive class in $H^1(X)$.  In general, it appears difficult to compute $\adam(X)$ since it requires taking a minimum over all of the non-separating hypersurfaces. However, we show in Proposition~\ref{prop:adam-iso} that for $F = \HFred$, if the cobordism $W \co Y \to Y$ gotten by cutting $X$ along $Y$ induces an isomorphism on Floer homology, then $\adam(X) = \dim \HFred(Y)$. The proof of Theorem \ref{thm:pi1Z-intro} is then surprisingly elementary: The surgery exact triangle implies that each 2-handle attachment in Construction~\ref{constr:pi1Z-intro} induces an isomorphism on $\HFred$, and thus the entire cobordism $\W$ does as well. We thus deduce that
\[
\adam(\X_p) \neq 0,
\]
while $\adam(S^1 \times S^3 \conn 2p \cptwo \conn 9p \cptwobar) = 0$, and hence $\X_p$ is exotic. This argument does not require any specific Heegaard Floer homology computations other than for $S^3, S^2 \times S^1$, and the fact that $-1$-surgery on a nontrivial slice knot has non-vanishing Floer homology. (See Remark \ref{rmk:easy-adam-argument}.) Moreover, this strategy would work equally well using instanton or monopole Floer homology in place of Heegaard Floer homology.  (We point out that our $\adam$ invariant is very closely related to Gadgil's direct limit invariant applied to the infinite cyclic cover \cite{Gadgil}.)

For Theorem \ref{thm:knot-surgery-intro}, we use an $\adam$ invariant defined in terms of $\HFh$, which we compute precisely for $\X^K_p$ in terms of $\dim \HFKh(K)$ (see \eqref{eq:adam(Xpk)}). This relies on an extensive computation using bordered Heegaard Floer homology, given in Section \ref{sec:bordered}.

\begin{remark}
Ciprian Manolescu pointed out to us that this is a rare instance in which $\HFh$ can detect exotic smooth structures on closed 4-manifolds.  However, our computation for $\X_p^K$ currently relies implicitly on the existence of the $\HFred$ theory (see Proposition~\ref{prop:adamhat-iso}), although this could plausibly be avoided.
\end{remark}

\begin{remark} \label{rem:reglue}
We can also produce exotic pairs with the same $\adam$-invariant but different Ozsv\'ath--Szab\'o invariants. For example, we can cut $\X_p$ along $S^3_{-1}(Q_{-4})$ and reglue using the involution $\sigma$ that was also used in Construction \ref{constr:simplyconn-intro}. The resulting manifold is diffeomorphic to neither  $\X_p$ nor $S^1 \times S^3 \conn 2p \cptwo \conn 9p \cptwobar$.  (See Proposition \ref{prop:pi1Z-trace}.)
\end{remark}

\begin{remark}
The idea to capture 4-manifold invariants by the exact triangle was used originally by Fintushel and Stern in \cite{FSK3} to compute Donaldson invariants of the $K3$ surface, which was then redone by Ozsv\'ath and Szab\'o in the context of Heegaard Floer homology \cite{OSzSymplectic}.
\end{remark}

\subsection{Additional results}
Finally, we collect a few other results that appear throughout the paper that may be of independent interest.

First, the homeomorphisms $\phi_n, \psi_n$ used above are special cases of a more general observation about knots with homeomorphic Dehn surgeries, shown below in Section~\ref{subsec:fusion}:
\begin{theorem} \label{thm:fusion-duals-intro}
Let $K$ be a fusion number 1 ribbon knot.  Then there exists another fusion number 1 ribbon knot $J$ such that $S^3_1(K) = S^3_{-1}(J)$.
\end{theorem}
Note that the Cosmetic Surgery Conjecture predicts that $K$ and $J$ cannot be isotopic, unless $K$ and $J$ are unknotted.  In fact, using the invariant $\adam(X)$, a counterexample to this conjecture would produce an interesting exotic 4-manifold via the following:
\begin{proposition}\label{prop:cosmetic-exotic}
If there exists a non-trivial knot $K$ such that $S^3_1(K) = S^3_{-1}(K)$, then there exists an exotic $S^1 \times S^3 \conn S^2 \times S^2$.
\end{proposition}
Similarly, we outline a potential strategy for constructing an exotic $S^2 \times S^2$ or $\cptwo \conn \cptwobar$ in Remark \ref{rem:S2xS2}.

Second, as a consequence of the techniques in this paper, we obtain a new example of a smooth 4-manifold that is homotopy equivalent to $S^2$ but does not contain a piecewise-linear spine. (Recall that a PL spine is a PL submanifold which carries the homotopy type.) See \cite{LevineLidman, HaydenPiccirillo} for the first known constructions.

\begin{theorem}\label{thm:spineless}
There exists a smooth manifold that is homeomorphic to the $0$-trace of the $P(3,-3,-8)$ pretzel knot and admits no PL spine.
\end{theorem}

Like the examples from \cite{HaydenPiccirillo}, the above example contains a topologically locally flat spine since $P(3,-3,-8)$ is smoothly slice and hence the 0-trace contains a smoothly embedded sphere carrying the homotopy type.

Our final application is to exotic surfaces.  As mentioned above, the simply-connected 4-manifold $\E$ comes with a free $\Z/2$-symmetry.  In fact, this manifold has various other $\Z/2$-symmetries with non-trivial fixed point sets.

\begin{theorem}\label{thm:surfaces}
The manifold $\E$ from Construction \ref{constr:simplyconn-intro} arises as the branched double cover of a surface in $S^4$ that is topologically but not smoothly isotopic to the standard $\# {10} \mathbb{R}P^2$ in $S^4$ with normal Euler number $16$.
\end{theorem}

The existence of such surfaces was first proven by Kreck \cite{Kreck}; we do not know whether ours agree with those of \cite{Kreck} or \cite{FKV}. The topological isotopy in Theorem \ref{thm:surfaces} can be shown using the results of \cite{Kreck} or \cite{ConwayOrsonPowell}. We note that there are already exotic embeddings of $\#6\R P^2$ in the literature \cite{Finashin}, which similarly come from taking a (more subtle) quotient of the \cite{ParkStipsiczSzabo} exotic $\cptwo\#5\cptwobar s$. The obviously symmetric nature of our constructions make it likely that if these techniques produced smaller closed 4-manifolds, they would produce smaller exotic surfaces as well.  We also comment that there are many other involutions one can use to take quotients of our pairs of simply connected exotic manifolds; some of these give oriented surfaces in small smooth manifolds (e.g. a $T^2$ nullhomologously embedded in $\# 4\cptwobar$), but because these surfaces have complicated fundamental group,  we do not know how to compare them topologically.

\subsection*{Organization}

In Section \ref{sec:background}, we provide a brief background on the Heegaard Floer techniques that will be used in the paper. In Section \ref{sec:building-blocks}, we introduce elementary cobordisms associated to certain ribbon knots, including the $P_n$ and $Q_n$ families above. We discuss the $\adam$ invariant in detail in Section \ref{sec:new-invts} and then use it in Section \ref{sec:construct} to prove Theorems \ref{thm:pi1Z-intro}, \ref{thm:knot-surgery-intro}, and \ref{thm:zerospheresintro} (modulo the bordered computations). In Section \ref{sec:simply-connected}, we then return to the simply-connected examples and prove Theorems \ref{thm:main}, \ref{thm:unlabeledthmintro}, and \ref{thm:spineless}. We then study exotic surfaces in Section \ref{sec:surfaces}, proving Theorem \ref{thm:surfaces}, and produce explicit handlebody diagrams for the exotic manifold $\E$ in Section \ref{sec:handles}. Finally, in Section \ref{sec:bordered}, we provide the bordered computations for Theorem \ref{thm:knot-surgery-intro}.

\subsection*{Acknowledgements}
We thank Keegan Boyle, Anthony Conway, Daniel Galvin, Jonathan Hanselman, Jen Hom, Robert Lipshitz, Ciprian Manolescu, Tom Mrowka,  Mark Powell, and Danny Ruberman for numerous helpful discussions.

\section{Heegaard Floer background} \label{sec:background}

We begin with a very brief recap of the necessary details on Heegaard Floer homology. All Heegaard Floer homology groups will be taken with coefficients in $\F = \Z/2\Z$. For a recent survey, see \cite{HomHFnotes}.

For a closed, oriented $3$-manifold $Y$ with a spin$^c$ structure $\mfs$, let $\HFm(Y,\mfs)$ denote the minus version of Heegaard Floer homology, which is a finitely generated module over $\F[U]$. Let $\HFred(Y,\mfs)$ denote the $U$-torsion submodule of $\HFm(Y,\mfs)$ (i.e., the set of all elements that are annihilated by some power of $U$). Equivalently, $\HFred(Y,\mfs)$ is equal to $\im \partial$ in the long exact sequence
\begin{equation} \label{eq: flavors-exact}
\dots \to \HFi(Y,\mfs) \to \HFp(Y,\mfs) \xrightarrow{\partial} \HFm(Y,\mfs) \to \dots.
\end{equation}
We define $\HFred(Y)$ to be the direct sum of $\HFred(Y,\spincs)$ over all spin$^c$ structures on $Y$.  Note that $\HFred(Y,\mfs) = 0$ for all but finitely many spin$^c$ structures on $Y$.
Let $\HFmc(Y,\mfs)$ denote the $U$-completion of $\HFm(Y,\mfs)$:
\[
\HFmc(Y,\mfs) = \HFm(Y,\mfs) \otimes_{\F[U]} \F[[U]].
\]
The $U$-torsion submodule of $\HFmc(Y,\mfs)$ is equal to $\HFred(Y,\mfs)$. (For more information about $\HFmc$, see \cite{ManolescuOzsvathLink}.)

If $\mfs$ is a torsion spin$^c$ structure, then $\HFm(Y,\mfs) / \HFmr(Y,\mfs)$ is a free $\F[U]$-module of some rank $k$ depending only on $Y$. (In particular, $k$ equals $1$ if $b_1(Y)=0$ and $2$ if $b_1(Y)=1$, which are the only cases that will be relevant here.) Thus, there is a (non-canonical) direct sum decomposition
\[
\HFm(Y,\mfs) \cong \HFred(Y,\mfs) \oplus \F[U]^k.
\]
Moreover, $\HFm(Y,\mfs)$ has an absolute $\Q$-grading that lifts a relative $\Z$-grading, with $\gr(U)=-2$. We obtain $\HFmc(Y,\mfs)$ by replacing each $\F[U]$ summand with $\F[[U]]$.

On the other hand, when $\mfs$ is a non-torsion spin$^c$ structure, there is a canonical projection $\Pi \co \HFm(Y,\spincs) \to \HFred(Y,\spincs)$, given by multiplication by $(1-U^{nd/2})$, where $n \gg 0$ and $d$ is the divisibility of $c_1(\spincs)$ (see \cite[Section 2.3]{OSzSymplectic}). Moreover, in this case, we have $\HFmc(Y,\spincs) = \HFred(Y,\spincs)$, so no projection is necessary.

When $Y$ is an integer homology sphere, so that there is a unique spin$^c$ structure (which we suppress from the notation), we have
\[
\HFm(Y) \cong \HFred(Y) \oplus \F[U]_{(d(Y))},
\]
where $d(Y)$ is an even integer, and the subscript indicates the grading of $1 \in \F[U]$. (Here we follow the grading convention from \cite[Remark 2.3]{HomHFnotes}.)  In particular, we have $\HFm(S^3) \cong \F[U]_{(0)}$. Likewise, if $M$ is a 3-manifold with $H_1(M) \cong \Z$, and $\spincs_0$ is the unique torsion spin$^c$ structure, we have
\[
\HFm(M,\spincs_0) \cong \HFred(M,\spincs_0) \oplus \F[U]_{(d_+(M))} \oplus \F[U]_{(d_-(M))},
\]
where $d_\pm(M)$ are rational numbers satisfying $ d_\pm(M) \equiv \pm1/2 \pmod 2$. In particular, for $M = S^2 \times S^1$, $d_{\pm} = \pm 1/2$ and $\HFred(M,\spincs) = 0$ for every spin$^c$ structure $\spincs$.  Also, for any knot $K \subset S^3$, we have $d_+(S^3_0(K)) = d(S^3_1(K)) + \frac12 $ and $d_-(S^3_0(K)) = d(S^3_{-1}(K))  -\frac12$. All of $d, d_+, d_-$ are homology cobordism invariants.

Given a spin$^c$ cobordism $(W,\mft)\co  (Y_1, \mfs_1) \to (Y_2, \mfs_2)$, there is an induced map $F^-_{W,\mft}\co \HFm(Y_1, \mfs_1) \to \HFm(Y_2,\mfs_2)$. Note that the sum of these maps over all spin$^c$ structures on $W$
is not necessarily defined (in contrast with the situation for $\HFp$ --- see \cite[Theorem 3.3]{OSz4Manifold}). However,
since $F^-_{W,\mft}$ takes $U$-torsion elements to $U$-torsion elements, it restricts to a map $F_{W,\mft} \co \HFred(Y_1,\mfs_1) \to \HFred(Y_2,\mfs_2)$. (Throughout, our convention is that the maps on $\HFred$ are written without the $-$ superscript.) It follows from \cite[Theorem 3.3]{OSz4Manifold} that there are only finitely many $\mft$ for which $F_{W,\mft}\ne0$, and thus, the total map
\[
F_W = \sum_{\mft} F_{W,\mft} \co \HFred(Y_1) \to \HF(Y_2)
\]
is well-defined. The total maps satisfy a simple composition law: if $W_1 \co Y_1 \to Y_2$ and $W_2 \co Y_2 \to Y_3$ are cobordisms, then $F_{W_1 \cup W_2} = F_{W_2} \circ F_{W_1}$.

On the other hand, if we consider one spin$^c$ structure on $W$ at a time, the composition laws become a bit more complicated:
$F^-_{W_2, \mft_2} \circ F^-_{W_1, \mft_1}$ is equal to the sum of all cobordism maps $F^-_{W_1 \cup W_2, \mfu}$ where $\mfu|_{W_1} = \mft_1$ and $\mfu|_{W_2} = \mft_2$. This sum is indexed by a set which is an affine copy of the image of the coboundary map from $H^1(Y_2)$ to $H^2(W_1 \cup W_2)$ in the Mayer--Vietoris sequence.  In particular, if $Y_2$ is a rational homology sphere, then there is a unique $\mfu$ restricting to $\mft_1$ and $\mft_2$ on $W_1$ and $W_2$ respectively, and we simply have $F^-_{W_2, \mft_2} \circ F^-_{W_1, \mft_1} = F^-_{W_1 \cup W_2, \mfu}$.

Likewise,  we can define $\mathbf F^-_{W,\spinct}$ to be the map on $\HFmc$ induced by $F^-_{W,\spinct}$. If
we define  $\HFmc(Y)$ to be the direct product sum of $\HFmc(Y,\spincs)$ over all spin$^c$ structures,  the total map $
\mathbf F^-_W \co  \HFmc(Y_1) \to \HFmc(Y_2)$ then makes sense,  and the composition law behaves in the same manner.

There is also a canonical involution $\iota_Y \co \HFm(Y,\mfs) \to \HFm(Y,\bar\mfs)$, called the spin$^c$ conjugation isomorphism, which extends to $\HFmc$ and restricts to $\HFred$. The involution commutes with cobordism maps in the following sense: given a spin$^c$ cobordism $(W,\mft) \co (Y_1, \mfs_1) \to (Y_2, \mfs_2)$, then
\begin{equation}\label{eq:conjugation-cobordism}
\iota_{Y_2} \circ F^-_{W,\mft} = F^-_{W,\bar \mft} \circ \iota_{Y_1}.
\end{equation}

Our main computational tool is the exact triangle for surgeries, which holds for the completed version $\HFmc$ (but not for uncompleted $\HFm$ because it requires a sum over spin$^c$ structures). In particular, for a knot $K$ in a homology sphere $Y$, there is an exact triangle of $\F[[U]]$-modules:
\begin{equation} \label{eq:HFm-exact}
\dots \to \HFmc(Y) \xrightarrow{f} \HFmc(Y_{-1}(K)) \xrightarrow{g} \HFmc(Y_0(K)) \xrightarrow{h} \HFmc(Y) \to \dots
\end{equation}
where the maps are those induced by the corresponding 2-handle cobordisms. As noted above, these maps each induce maps on the corresponding $\HFred$ submodules, which we also denote by $f$, $g$, and $h$. We will make frequent use the following lemma:

\begin{lemma} \label{lemma:HFred-exact} Suppose $Y$ is a homology sphere and $K \subset Y$ is a knot for which $d(Y_{-1}(K)) = d(Y)$ and $d_\pm (Y_0(K)) = d(Y) \pm \frac12$. Then the sequence
\begin{equation} \label{eq:HFred-exact}
\dots \to \HFred(Y) \xrightarrow{f} \HFred(Y_{-1}(J)) \xrightarrow{g} \HFred(Y_0(J)) \xrightarrow{h} \HFred(Y) \to \ldots
\end{equation}
is exact.
\end{lemma}

\begin{proof}
From \eqref{eq:HFm-exact}, let $\bar f, \bar g, \bar h$ denote the maps on the quotients $\HFmc/\HFred$. By considering the corresponding maps on $\HFi$ (see, eg., \cite[Proposition 4.11]{OSzAbsolute}), we deduce the following:
\begin{itemize}
  \item  The map $\bar g$ is homogeneous of degree $-\frac12$ and takes $\F[[U]]_{(d(Y_{-1}))}$ injectively into $\F[[U]]_{d_-(Y_0)}$.
  \item  The map $\bar h$ is homogeneous of degree $-\frac12$, vanishes on $\F[[U]]_{(d_-(Y_0(K)))}$, and takes $\F[[U]]_{(d_+(Y_0))}$ injectively into $\F[[U]]_{d(Y)}$.
  \item We have $\bar f =0$.
\end{itemize}
Thus, under the hypotheses of the lemma, the sequence
\[
0 \to \F[[U]]_{(d(Y))} \xrightarrow{\bar g} \F[[U]]_{(d_+(Y_0))} \oplus \F[[U]]_{(d_-(Y_0))} \xrightarrow{\bar h} \F[[U]]_{(d(Y_1))} \to 0
\]
is short exact.  By some simple diagram chasing, we can then show that \eqref{eq:HFred-exact} is exact.
\end{proof}

In particular, note that when $Y = S^3$ and $K$ is a slice knot, the hypotheses of Lemma \ref{lemma:HFred-exact} are automatically satisfied.

As a special case, we have the following lemma.

\begin{lemma} \label{lemma:genus1-exact}
Suppose $Y$ is a homology 3-sphere and $K \subset Y$ is a knot with the following properties:
\begin{enumerate}
\item $d(Y) = d(Y_{-1}(K)) =0$;
\item $Y_0(K) \cong S^1 \times S^2$; and
\item $g(K)=1$.
\end{enumerate}
Let $W\co Y \to Y'=Y_{-1}(K)$ be the cobordism given by attaching a $-1$-framed $2$-handle along $K$. Then $F_W \co \HFred(Y) \to \HFred(Y')$ is an isomorphism. Furthermore, for each odd integer $m$, let $\mft_m$ be the spin$^c$ structure on $W$ characterized by $\gen{c_1(\mft_m), [\hat \Sigma]} = m$, where $\hat \Sigma$ is a capped-off Seifert surface for $K$. Then the map
\[
F_{W,\mft_1} + F_{W, \mft_{-1}} \co \HFred(Y) \to \HFred(Y')
\]
is grading-preserving and $\iota$-equivariant. However, the individual terms
$F_{W,\mft_1}$ and $F_{W,\mft_{-1}}$ are not isomorphisms.
\end{lemma}

\begin{proof}
As above, the assumption on $d$ invariants shows that the restriction of $F_W^-$ to $\HFred$ fits into an exact triangle
\[
\HFred(S^1 \times S^2) \to \HFred(Y) \xrightarrow{F_W} \HFred(Y') \to \HFred(S^1 \times S^2).
\]
Since $\HFred(S^1 \times S^2)=0$, it follows that $F_W$ is an isomorphism.

For ease of notation, denote the map
\[
F^-_{W, \mft_m} \co \HFm(Y) \to \HFm(Y')
\]
by $f_m$. By abuse of notation, we use the same symbol for the map on $\HFmc$, so that $\mathbf F^-_W = \sum_{m \text{ odd}} f_m$. The grading shift of $f_m$ is $\frac{1-m^2}{4}$; in particular, $f_1$ and $f_{-1}$ are grading-preserving. Because $\mft_m$ and $\mft_{-m}$ are conjugate spin$^c$ structures, we see that $f_m + f_{-m}$ is $\iota$-equivariant for any $m$, i.e.,
\[
(f_m + f_{-m}) \iota_Y  = \iota_{Y'} (f_m + f_{-m}).
\]

Now, we claim that $\mathbf F^-_W$ is completely determined by $f_{-1}$ and $f_1$, in the following manner:
\begin{equation} \label{eq: powerseries}
\mathbf F^-_W = P(U) (f_1 + f_{-1}), \text{ where } P(U) = \sum_{i=0}^\infty U^{i(i+1)/2}.
\end{equation}
To see this, for any integer $i \ge 1$, we may stabilize $\hat\Sigma$ $i-1$ times to give a surface $\hat\Sigma^{(i)}$ in the same homology class with genus $i$. We have
\[
\gen{c_1(\mft_{2i+1}), -[\hat \Sigma^{(i)}]} = -2i-1 = -2g(\hat\Sigma^{(i)}) + [\hat\Sigma^{(i)}]^2.
\]
Hence, by the adjunction formula \cite[Theorem 3.1]{OSzSymplectic},
\[
F^-_{W, \mft_{2i+1}} = U^i F^-_{W, \mft_{2i+1} + \PD[\hat \Sigma] }  =  U^i F^-_{W, \mft_{2i-1}}.
\]
By induction, we see that
\[
f_{2i+1}  = U^{i(i+1)/2} f_1.
\]
Likewise, we obtain
\[
f_{-2i-1} = U^{i(i+1)/2} f_{-1}.
\]
Taking the infinite sum, \eqref{eq: powerseries} then follows.

Since $F^-_W$ is injective on $\HFred(Y)$, \eqref{eq: powerseries} shows that $f_1+f_{-1}$ is also injective on $\HFred(Y)$. For surjectivity, let $\xi \in \HFred(Y')$. Then there is some unique $\eta \in \HFred(Y)$, not necessarily homogeneous, for which
\[
\xi = F^-_W (\eta) = P(U)(f_1 + f_{-1})(\eta) = (f_1 + f_{-1})(P(U) \eta).
\]
Since $\eta \in \HFred(Y)$, it is killed by sufficiently large powers of $U$. Thus, the sum $P(U) \eta$ makes sense as an element of $\HFm(Y)$, which is necessarily in the submodule $\HFred(Y)$. Thus, we deduce that $(f_1+f_{-1})|_{\HFred(Y)}$ is surjective, and hence an isomorphism.

On the other hand, we claim that $f_1$ and $f_{-1}$ are not individually isomorphisms on $\HFred$. Namely, because $g(K)=1$, there are exact triangles coming from \cite[Proof of Corollary 4.5]{OzSzHFK} and \cite[Lemma 6.7]{NiFibred}
\[
\HFKh(Y,K, \pm 1) \to \HFm(Y) \xrightarrow {f_{\pm1} } \HFm(Y') \to \HFKh(Y,K, \pm1),
\]
and since $\HFKh(Y,K, \pm 1) \ne 0$, $f_{\pm1}$ has nontrivial kernel. Furthermore, by considering the behavior on $\HFi$, we see that the induced map
\[
\bar f_{\pm 1} \co \HFm(Y)/\HFred(Y) \to \HFm(Y')/\HFred(Y')
\]
is a grading-preserving, nonzero map between copies of $\F[U]_{(0)}$, hence an isomorphism. Thus, the restriction of $f_{\pm1}$ to $\HFred$ must have nontrivial kernel, as required.
\end{proof}

\section{Topological building blocks} \label{sec:building-blocks}

In this section, we will introduce the basic topological building blocks for our constructions. We work in somewhat more generality than is needed for the specific examples discussed in the introduction, with an eye toward future applications of the same elementary pieces.

\subsection{Fusion number 1 ribbon knots}\label{subsec:fusion}

The 3-manifolds that we consider throughout this paper will be built out of surgeries on fusion number 1 ribbon knots in $S^3$.

\begin{figure}[htbp]{\scriptsize
\begin{overpic}[tics=20]{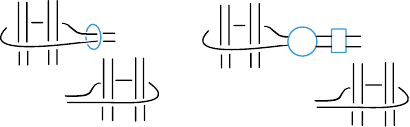}
\put(142,38){\color{lblue}$K$}
\put(161,38){\color{lblue}$n$}
\put(48,49){\color{lblue}$\gamma$}
  \end{overpic}}
  \caption{Left: schematic for a fusion number 1 ribbon knot $J$, along with the curve $\gamma$. Right: the satellite knot $J_n^K$. }
  \label{fig:ribbonpattern}
\end{figure}

\begin{definition} \label{def:fusion1}
A nontrivial knot $J \subset S^3$ is called \emph{fusion number 1} if it is obtained by a band sum on a two-component unlink in $S^3$. Such a knot is necessarily ribbon. Let $\gamma_J$ denote the meridian of the band. See Figure \ref{fig:ribbonpattern}. By removing a neighborhood of $\gamma_J$, $J$ can be thought of as a satellite operator.  For any companion knot $K$ and twisting parameter $n$ we define the fusion number 1 satellite knot $J_n^K$, by tying $K$ into the band and inserting $n$ full twists.
\end{definition}

\begin{remark}
We will frequently consider $\gamma_J$ as a knot in the surgered manifold $S^3_{\pm1}(J)$. We record a few useful facts about this knot.
\begin{itemize}
\item
Since $\lk(J,\gamma_J)=0$, the Seifert framing for $J$, considered as a knot in $S^3_{\pm1}(J)$, agrees with the Seifert framing in $S^3$ (i.e. the blackboard framing in Figure \ref{fig:ribbonpattern}).

\item
Performing 0-surgery on $\gamma_J$ in $S^3_{\pm 1}(J)$ returns $S^1\times S^2$. We may see this by performing handleslides of $J$ over $\gamma_J$ to eliminate all of the clasps, transforming $J$ into a $\pm1$ framed unknot.

\item
Observe that $\gamma_J$ bounds a genus-1 surface in the complement of $J$, and hence in $S^3_{\pm1}(J)$. Moreover, we claim that $\gamma_J$ is not the unknot in $S^3_{\pm1}(J)$. If so, then the $0$-surgery on $S^3_{\pm1}(J)$ would be homeomorphic to $S^3_{\pm1}(J) \conn S^1 \times S^2$, and hence $S^3_{\pm1}(J) \cong S^3$. By the knot complement problem \cite{GordonLuecke}, this would imply that $J$ is unknotted, which contradicts our assumption. Thus, $g(\gamma_J) = 1$.
\end{itemize}

\end{remark}

\begin{definition}\label{def:contract}
For a slice knot $J$ with slice disk $D$, define $C^\pm_D$ to be the contractible manifold formed from the exterior of $D$ by attaching a $\mp 1$ framed 2-handle along the meridian of $J$, with $\partial C^\pm_D = S^3_{\pm1}(J)$. If $J$ is fusion number 1 and $D_J$ is the ribbon disk obtained from a given fusion number 1 presentation, we write $C^\pm_J$ for $C^\pm_{D_J}$. We will sometimes omit the $\pm$ superscript when clear from context.
\end{definition}

Note that $C_J^\pm$ has a handle diagram with two 1-handles and two 2-handles, as shown in the leftmost frame of Figure \ref{fig:cobs}.

\begin{figure}[htbp]{\scriptsize
\begin{overpic}[tics=20]{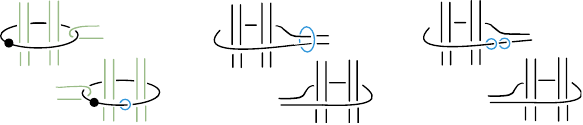}
\put(227,17){$\langle -1 \rangle$}
\put(131,17){$\langle \pm 1 \rangle$}
\put(37,4){\color{lgreen}$0$}
\put(59,-5){\color{lblue}$\mp 1$}
\put(150,48){\color{lblue}$-1$}
\put(237,47){\color{lblue}$-1$}
\put(237,55){\color{lblue}$-1$}
  \end{overpic}}
  \caption{From left to right, the building blocks $C^\pm_J$, $B_J^\pm$, and $T_J$, all exhibited for a fusion number 1 ribbon knot $J$.}
  \label{fig:cobs}
\end{figure}
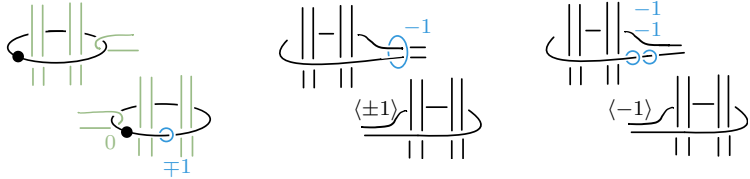

The key principle, stated above as Theorem \ref{thm:fusion-duals-intro}, is that $+1$ surgery on any fusion number 1 ribbon knot is homeomorphic to $-1$ surgery on another fusion number 1 ribbon knot. We make this precise as follows.

\begin{figure}[htbp]{\scriptsize
\begin{overpic}[tics=20]{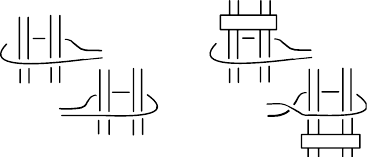}
 \put(114, 63){$-2$}
    \put(154, 5){$-2$}
     \put(-6, 45){$1$}
               \put(180, 20){$-1$}
  \end{overpic}}
  \caption{Pairs of fusion number 1 ribbon knots with homeomorphic $\pm 1$ surgeries.  }
  \label{fig:refcalc}
\end{figure}
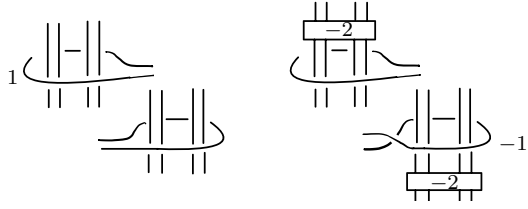

\begin{definition} \label{def:fusion-duals}
Given a fusion number 1 ribbon knot $J$, as shown in the left side of Figure \ref{fig:refcalc}, let $J^*$ be the knot shown on the right side of Figure \ref{fig:refcalc}.
%(In words, if $J$ is obtained from a band sum from a two-component unlink $L_1 \cup L_2$ using band $b$, then $J'$ is obtained from $J$ by inserting two full positive twists along $L_1$ and two full positive twists along $L_2$, and adding one negative half-twist in the band $b$.)
\end{definition}

\begin{proposition} \label{prop:fusion-duals}
%For $J$ and $J^*$ as in Definition \ref{prop:fusion-duals}, there is a diffeomorphism from $C^+_J$ to $C^-_{J^*}$ that takes $\gamma_J$ to $\gamma_{J^*}$.
For $J$ and $J^*$ as in Definition \ref{def:fusion-duals}, there is a diffeomorphism $\phi_J \co S^3_1(J) \to S^3_{-1}(J^*)$ that takes $\gamma_J$ to $\gamma_{J^*}$ (preserving the blackboard framing). Moreover, $\phi_J$ extends to a diffeomorphism from $C^+_J$ to $C^-_{J^*}$, which we also denote $\phi_J$.
\end{proposition}

\begin{figure}[htbp]{\scriptsize
\begin{overpic}[tics=20]{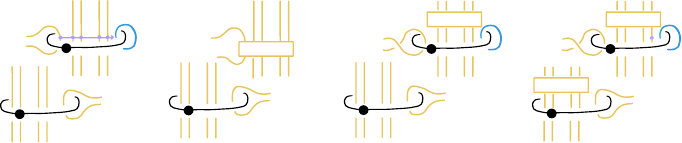}
\put(68,35){$\cong$}
\put(150,35){$\cong$}
\put(245,35){$\cong$}
\put(125,42){$\color{lyellow}-1$}
\put(214,57){$\color{lyellow}-2$}
\put(301,57){$\color{lyellow}-2$}
\put(266,25){$\color{lyellow}-2$}
\put(43,9){$\color{lyellow}0$}
\put(123,9){$\color{lyellow}-1$}
\put(205,9){$\color{lyellow}-2$}
\put(300,9){$\color{lyellow}0$}
\put(60,58){$\color{lblue}-1$}
\put(240,58){$\color{lblue}1$}
\put(328,58){$\color{lblue}1$}
%\put(237,4){$\color{lgreen}0$}
%\put(259,-5){\color{lblue}$\mp 1$}
%\put(150,48){\color{lblue}$-1$}
%\put(36,47){\color{lblue}$-1$}
%\put(36,55){\color{lblue}$-1$}
  \end{overpic}}
  \caption{Kirby moves showing a diffeomorphism from $C^+_J$ to $C^-_{J*}$.}
%  The second diffeomorphism is obtained by adding a 1-2 pair, and performing five 2-handle slides. The third diffeomorphism is obtained by spinning the lower $1$-handle. }
  \label{fig:extends}
\end{figure}
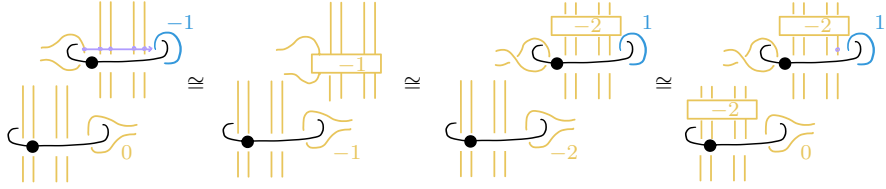

\begin{proof}
On the 4-manifold level, we define $\phi_J \co C^+_J \to C^-_{J^*}$ by the sequence of Kirby moves in Figure \ref{fig:extends}. The first figure represents $C^+_J$. The first diffeomorphism is obtained by performing a sequence of 2-handle slides (as indicated by the purple arrow) and then cancel a 1-2 pair. The second diffeomorphism is obtained by introducing a new 1-2 pair (where the new 2-handle has framing coefficient 1) and then performing 2-handle slides. The third diffeomorphism is obtained by twisting the lower 1-handle (see \cite[Figure 5.42]{GompfStipsicz}). The final figure then represents $C^-_{J^*}$. Note that the curve $\gamma_J$, which is a meridian of the band, is preserved throughout this process.
\end{proof}

\begin{remark}
Because $\phi_J$ takes $\gamma_J$ to $\gamma_{J^*}$, it extends to a diffeomorphism from $S^3_1(J_n^K)$ to $S^3_{-1}(J^{* K}_n)$, for any knot $K \subset S^3$ and any $n \in \Z$. This is the same diffeomorphism as if we had applied Proposition \ref{prop:fusion-duals} to $J_n^K$ directly.

%Likewise, $\phi_J$ also induces a diffeomorphism from $B^+_J$ to $B^-_{J^*}$, for the same reason. \footnote{LP: Adam explained why this remark is interesting, and we agreed to move it...where?}

%On the other hand, note that $\phi_J$ does \emph{not} take a meridian of $J$ to a meridian of $J^*$. Rather... \todo{What goes here? Does this help later in drawing handlebody diagrams?}
\end{remark}

\subsection{Elementary cobordisms}\label{subsec:cobs}

We now introduce two classes of cobordisms, which generalize the handle attachments used in Constructions \ref{constr:simplyconn-intro} and \ref{constr:pi1Z-intro}. The key property is that they both induce isomorphisms on $\HFred$.

\begin{definition} \label{def:twist-cob}
For a fusion number 1 ribbon knot $J \subset S^3$, define $B_J^{\pm}$ to be the manifold obtained from $S^3_{\pm 1}(J)\times [0,1]$ by attaching a $-1$-framed 2-handle to  $S^3_{\pm 1}(J)\times \{1\}$ along $\gamma$, as shown in the middle drawing in Figure \ref{fig:cobs}.  On the level of surgery diagrams, blowing down this $-1$-framed unknot adds a twist into the band and leaves the framing of $J$ unchanged, so the upper boundary is canonically identified with $S^3_{\pm1}(J_1)$. That is, $B_J^{\pm}$ is a cobordism from $S^3_{\pm1}(J)$ to $S^3_{\pm1}(J_1)$. We will sometimes omit the $\pm$ superscript on $B_J^{\pm}$ when clear from context.
\end{definition}

\begin{remark}\label{rem:BJ}\noindent
\begin{itemize}
\item Notice that $H_2(B_J) \cong \Z$, generated by a genus~$1$ surface with self-intersection~$-1$.

\item
For knots $J$ and $J^*$ as in Definition \ref{def:fusion-duals}, because $\phi_J$ takes $\gamma_J \subset S^3_1(J)$ to $\gamma_{J^*} \subset S^3_{-1}(J^*)$, it induces a diffeomorphism from $B^+_J$ to $B^-_{J^*}$.
\end{itemize}
\end{remark}

\begin{proposition}\label{prop:twist-iso}
Let $J$ be any fusion number 1 ribbon knot. Then the twist cobordism $B_J^\pm \co S^3_{\pm 1}(J) \to S^3_{\pm 1}(J_1)$ induces an isomorphism on $\HFred$. Moreover, if $\spinct_1$ and $\spinct_{-1}$ denote the spin$^c$ structures on $B_J^\pm$ whose first Chern classes evaluate to $\pm 1$ on a generator of $H_2(B_J^\pm)$, the map
\[
F_{B_J^\pm, \spinct_1} + F_{B_J^\pm, \spinct_{-1}} \co \HFred(S^3_{\pm 1}(J)) \to \HFred(S^3_{\pm 1}(J_1))
\]
is a grading-preserving, $\iota$-equivariant isomorphism. Moreover, the individual summands $F^-_{B_J^\pm, \spinct_1}$ and $F^-_{B_J^\pm, \spinct_1}$ are not isomorphisms on $\HFmr$.
\end{proposition}

\begin{proof}
This is an immediate application of Lemma \ref{lemma:genus1-exact} in light of Remark~\ref{rem:BJ}.
\end{proof}

The second class of cobordisms makes sense for any knot, although we will still primarily be focused on the fusion number 1 case.

\begin{definition}\label{def:Tye-cob}
For any knot $J \subset S^3$, let $T_J^- \co S^3_{-1}(J) \to S^3_0(J)$ and $T_J^+ \co S^3_0(J) \to S^3_1(J)$ be the cobordisms obtained in each case by attaching a $-1$-framed $2$-handle meridian along the meridian of $J$.
Let $T_J \co S^3_{-1}(J) \to S^3_1(J)$ denote the composite of these two cobordisms; this is obtained by attaching two $-1$-framed 2-handles to  $S^3_{-1}(J)\times \{1\}$ along parallel copies of the meridian $\mu_J$, as shown in the right-hand image in Figure \ref{fig:cobs}.
\end{definition}

\begin{remark}\label{rem:TJ}
Because $T_J$ is obtained by attaching two $2$-handles to a homology sphere, we have $H_2(T_J) \cong \Z^2$. An explicit basis is as follows: Let $\Xi$ denote the class of a capped-off Seifert surface inside of $S^3_0(K)$ (which sits inside of $T_J$). Let $\Theta$ be the class obtained from an annulus connecting the two copies of $\mu_J$ and capping it off with the cores of the two 2-handles to produce a sphere. Then with suitable orientations, $\Xi^2=0$, $\Theta^2=-2$, and $\Xi \cdot \Theta = 1$. In particular, we deduce that the intersection form of $T_J$ is isomorphic to that of $S^2 \times S^2$, and $T_J$ is spin.
%\todo{The old version said that $T_J$ has the intersection form of $\cptwo \conn \cptwobar$, but I think that's not right.} LP: yes you are correct, so cleared the comment

Additionally, because $\mu_J$ normally generates $\pi_1(S^3_{-1}(J))$, attaching a $2$-handle along $\mu_J$ kills the fundamental group. Thus, $T_J$ is simply connected.
%\begin{enumerate}
%\item One can show that there is a graded isomorphism from the relative homology $H_{*}(T_K,\partial)$ to $H_{*}(\cptwo\#\cptwobar).$
%\item Notice that the 0-surgery on $\mu_J$ in $S^3_{\pm 1}(J)$ returns $S^3$. From this one can show that $T_J$ is simply connected.
%\end{enumerate}
\end{remark}

\begin{proposition} \label{prop:tye-iso}
Let $J \subset S^3$ be a slice knot, or more generally a knot such that $d(S^3_{-1}(J)) = d(S^3_1(J)) = 0$.   Then $T_J^-$ and $T_J^+$ each induce isomorphisms on $\HFred$, as does the composite cobordism $T_J$.
\end{proposition}

\begin{proof}%[Proof of Proposition~\ref{prop:tye-iso}]
By \cite[Proposition 4.12]{OSzAbsolute}, we see that $d_\pm(S^3_0(J)) = \pm1/2$. The proposition then follows immediately from two applications of Lemma \ref{lemma:HFred-exact}. Namely, for the first cobordism, we take $Y = S^3$ and $K=J$. For the second, we take $Y = S^3_1(J)$ and $K$ to be the dual knot of $J$, so that $Y_{-1}(K)=S^3$ and $Y_0(K) = S^3_0(J)$. In each case, since $\HFred(S^3)=0$, the relevant map in the exact sequence \eqref{eq:HFred-exact} is an isomorphism.
\end{proof}

\begin{remark}
One can also easily prove analogues of Propositions \ref{prop:twist-iso} and \ref{prop:tye-iso}  in the setting of instanton Floer homology.  The instanton Floer homologies of $S^3$ and $S^2 \times S^1$ (with an admissible bundle) vanish, so the desired isomorphisms are immediate from the exact triangle.  In the case that $J$ is nontrivial, the instanton Floer homology groups of $S^3_{\pm1}(J)$ are always non-zero by Kronheimer--Mrowka's proof of Property P \cite{PropertyP}.%
\end{remark}

\begin{remark} \label{rem:boring-compose}
Since each boundary component of $C_J^{\pm}$, $B_J^\pm$, and $T_J$ has a canonical identification with the relevant surgery on $J$ or $J_1$, there is a canonical way to compose the cobordisms end-to-end where applicable. These Dehn surgery identifications also allow us to meaningfully discuss gluing together the two ends of a cobordism.  We will denote such a composition of cobordisms with the $\cup$ sign, without specifying the gluing map. (However, if a different gluing map $f$ is used, we will write $\cup_f$.)

For instance, $B^\pm_J \cup B^\pm_{J_1}$ makes sense as a cobordism from $S^3_{\pm1}(J)$ to $S^3_{\pm1}(J_2)$, given by attaching two $-1$-framed handles along parallel, unlinked copies of $\gamma_J$. Likewise, the composition $C_J^\pm \cup B_J^\pm$ gives a manifold with boundary $S^3_{\pm 1}(J_1)$. Because the $2$-handle of $B_J^\pm$ is attached along $\gamma_J$, which does not link the $1$-handles of $C_J^\pm$, we may blow it down to see that
\begin{equation} \label{eq:boring-twist}
C_J^\pm \cup  B_J^{\pm} \cong C^+_{J_1} \conn \cptwobar.
\end{equation}
In a similar fashion, $C_J^-$ can be glued to the cobordism $T_J^- \co S^3_{-1}(J) \to S^3_0(J)$, and we have
\begin{equation} \label{eq:boring-tye}
C_J^- \cup T_J^- \cong X_0(J),
\end{equation}
where $X_0(J)$ denotes the $0$-trace of $J$.
\end{remark}

\subsection{The \texorpdfstring{$P_n$}{P\_n} and \texorpdfstring{$Q_n$}{Q\_n} families} \label{subsec:PQ}

%\begin{example}\label{ex:P-Q-dual}
As in the introduction, we now introduce two explicit families of fusion number 1 ribbon knots that will feature prominently in the arguments below.  Let $Q_n$ be the $(3,-3,2n)$ pretzel knot,  and let $P_n$ be the $(3,-3,2n-1)$ pretzel knot, where in this notation the parameters correspond to half twists in the strands. Each of these is a fusion-number 1 ribbon knot, where the right-hand twist region is understood to be the band. As above, let $P_n^K$ and $Q_n^K$ denote the result of tying the ribbon  band
%with $2n-1$ (resp.~$2n$) twists
into a knot $K$. Note that
\begin{equation} \label{eq:PQ-symmetry}
Q_{-n} = \overline{Q_n} \quad \text{and} \quad  P_{-n} = \overline{P_{n+1}}.
\end{equation}

The families $P_n$ and $Q_n$ are ``dual'' in the following sense:
\begin{proposition}\label{prop:PQdual}
For any knot $K$ and integer $n$, there are homeomorphisms
\begin{align*}
\phi_n^K \co & S^3_1(P_n^K) \to S^3_{-1}(Q_{n-4}^K)  \\
\psi_n^K \co & S^3_1(Q_n^K) \to S^3_{-1}(P_{n-3}^K),
\end{align*}
preserving the $\gamma$ curves.
\end{proposition}

\begin{proof}
Applying Proposition \ref{prop:fusion-duals} to $P_0$ and $Q_0$, we see that $P_0^*=Q_{n-4}$ and $Q_0^* = P_{n-3}$.
\end{proof}

The Heegaard Floer homologies of $\pm1$ surgeries on $P_n$ and $Q_n$ are as follows:
\begin{lemma} \label{lemma:HFm-Pn-Qn}
For any integer $n$, we have graded isomorphisms
\begin{align}
\label{eq:HFm-S31-Pn} \HFm(S^3_{1}(P_n)) \cong \HFm(S^3_{-1}(Q_n)) &\cong \F[U]_{(0)} \oplus \F_{(0)} \oplus \F_{(0)} \\
\label{eq:HFm-S31-Qn} \HFm(S^3_{1}(Q_n)) \cong \HFm(S^3_{-1}(P_n)) &\cong \F[U]_{(0)} \oplus \F_{(1)} \oplus \F_{(1)}
\end{align}
where $\iota$ acts by fixing $\F[U]$ and interchanging the two copies of $\F$ in $\HFred$.
\end{lemma}

\begin{proof}
First, since each of these manifolds is $\pm1$ surgery on a ribbon knot, the $d$ invariant is $0$, so $\HFm$ contains a copy of $\F[U]_{(0)}$. By Proposition \ref{prop:tye-iso}, the graded isomorphism type of $\HFred$ (i.e. the remaining summand), along with the induced action of $\iota$, is independent of $n$. Proposition \ref{prop:PQdual} gives the first isomorphism in each line.

The group $\HFm(S^3_1(P_0))$ with its involutive structure was computed explicitly by Dai--Hedden--Mallick \cite[Example 4.7]{DHMCorks}: it is generated over $\F[U]$ by generators $v_1,v_2,v_3$ in grading $0$, with $Uv_1 = Uv_2 = Uv_3$, and $\iota$ given by $\iota(v_1) = v_3$, $\iota(v_2) = v_2$, $\iota(v_3) = v_1$. In particular, $\HFred(S^3_1(P_0))$ is generated by $v_1+v_2$ and $v_2+v_3$, which are interchanged under $\iota$. This completes the proof of \eqref{eq:HFm-S31-Pn}. Since $S^3_1(Q_n) = \overline{S^3_{-1}(Q_{-n})}$, one can then obtain \eqref{eq:HFm-S31-Qn} using orientation reversal and the exact sequence \eqref{eq: flavors-exact}.
\end{proof}

We now describe the cobordism maps on $\HFred$ associated to the $T_J$ and $B_J^{\pm}$ cobordisms when $J = P_n$ or $Q_n$ in more detail, paying attention to the spin$^c$ decomposition in each case. A \emph{good basis} for $\HFred(S^3_{\pm1}(P_n))$ or $\HFred(S^3_{\pm1}(Q_n))$ is an ordered basis $(a,b)$ such that $\iota(a)=b$ and $\iota(b)=a$. Such a basis is unique up to exchanging $a$ and $b$.

We start with the $B_J^{\pm}$ cobordisms. (The same reasoning will apply for $J = P_n$ or $Q_n$, and with either choice of sign.) As in Lemma \ref{lemma:genus1-exact} above, a choice of generator for $H_2(B^\pm_J)$ gives an indexing of the spin$^c$ structures on $B_J^\pm$ by $\spinct_m$ for odd integers $m$. Choosing the opposite generator negates this indexing.
\begin{lemma} \label{lemma:twist-spinc}
Let $J_n = P_n$ or $Q_n$. Given a good basis $(a,b)$ for $\HFred(S^3_{\pm1}(J_n))$, we may find a choice of generator for $H_2(B^\pm_{J_n})$ and a good basis $(a',b')$ for $\HFred(S^3_{\pm1}(J_{n+1}))$ such that the following holds:
\begin{align*}
F_{B_J^\pm, \spinct_1} (a)  &= a' &  F_{B_J^\pm, \spinct_{-1}} (a) &= 0 \\
F_{B_J^\pm, \spinct_1} (b) &= 0 &  F_{B_J^\pm, \spinct_{-1}} (b) &= b'.
\end{align*}
\end{lemma}

\begin{proof}
For convenience, we write $f_m = F_{B_J^\pm, \spinct_m}$. By Proposition \ref{prop:twist-iso}, $f_1+f_{-1}$ is an isomorphism on $\HFred$, while $f_1$ and $f_{-1}$ individually have rank $1$. For any good basis $(a,b)$ for $\HFred(S^3_{\pm1}(J_n))$, $(f_1+f_{-1})(a)$ and $(f_1+f_{-1})(b)$ are distinct, nonzero elements of $\HFred(S^3_{-1}(J_{n+1}))$ that are interchanged by $\iota$. Up to reversing the indexing of spin$^c$ structures, we may assume that $f_1(a) \ne 0$; we call this element $a'$. Since the rank of $f_1$ is $1$, either $f_1(b) = 0 $ or $f_1(b) = a'$. In the latter case, we compute:
\[
(f_1+f_{-1})(a) = (f_1+ \iota f_1 \iota)(a) =  f_1(a) + \iota f_1(b) = (1+\iota) a'.
\]
But this means that $(f_1+f_{-1})(a)$ is an $\iota$-equivariant element, which contradicts our earlier statement. Thus, $f_1(b)=0$. It follows that
\begin{align*}
f_{-1}(a) &= \iota f_1 \iota(a) = \iota f_1(b) = 0 \\
f_{-1}(b) &= \iota f_1 \iota(b) = \iota (a').
\end{align*}
Moreover, we have $a' = (f_1+f_{-1})(a)$, so $\iota(a') \ne a'$. Defining $b' = \iota(a')$ completes the proof.
\end{proof}

For the $T_J$ cobordisms, recall the basis $\Xi,\Theta$ for $H_2(T_J)$ from Remark \ref{rem:TJ}: $\Xi$ is represented by a capped-off Seifert surface in $S^3_0(J)$, and $\Theta$ is represented by a $-2$-sphere. For any knot $J$, we may index the spin$^c$ structures on $T_J$ by $\spinct_{i,j}$, where $\gen{c_1(\spinct_{i,j}), \Xi} = 2i$ and $\gen{c_1(\spinct_{i,j}), \Theta} = 2j$. Again, negating both $\Xi$ and $\Theta$ reverses this indexing. We can compute that
\begin{align*}
\spinct_{i,j} + PD(\Xi) &= \spinct_{i,j+1} \\
\spinct_{i,j} + PD(\Theta) &= \spinct_{i+1,j-2}.
\end{align*}

\begin{lemma} \label{lemma:tye-spinc}
\begin{enumerate}
\item For any $n$, the map $F_{T_{P_n}, \spinct_{i,j}} \co \HFred(S^3_{-1}(P_n)) \to \HFred(S^3_1(P_n))$ is nonzero only for $i=j=0$; this map is an $\iota$-equivariant isomorphism.

\item For any $n$, the map $F_{T_{Q_n}, \spinct_{i,j}} \co \HFred(S^3_{-1}(Q_n)) \to \HFred(S^3_1(Q_n))$ is nonzero only for $i=\pm 1, j=0$. Up to negating $\Xi$ and $\Theta$, we may find good bases $(a,b)$ and $(a,b')$ such that
\begin{align*}
F_{T_{Q_n}, \spinct_{1,0}} (a)  &= a' &  F_{T_{Q_n}, \spinct_{-1,0}} (a) &= 0 \\
F_{T_{Q_n}, \spinct_{1,0}} (b) &= 0 &  F_{T_{Q_n}, \spinct_{-1,0}} (b) &= b'.
\end{align*}
\end{enumerate}
\end{lemma}

\begin{proof}
For the first statement, since $g(P_n)=1$, $\Xi$ is represented by a square-0 torus. By adjunction, this implies that $F_{T_{P_n}, \spinct_{i,j}}=0$ when $i \ne 0$. For $j > 0$, represent $\Theta$ by a surface $\Sigma$ of genus $j-1$. Then
\[
\gen{c_1(\spinct_{0,j}), -\Theta} - \Theta^2 = -2(j-1)
\]
so again by adjunction \cite[Theorem 3.1]{OSzSymplectic}
\[
F_{T_{P_n}, \spinct_{0,j}} = U^{j-1} F_{T_{P_n}, \spinct_{0,j}-PD(\Theta)} = U^{j-1} F_{T_{P_n}, \spinct_{-1,j+2}} = 0.
\]
The $j<0$ case follows symmetrically. Thus, the only nonzero map is $F_{T_{P_n}, \spinct_{0,0}}$, which must be an isomorphism on $\HFred$ by Proposition \ref{prop:tye-iso}, and is $\iota$-equivariant because $\spinct_{0,0}$ is self-conjugate.

For the second statement, note that
\[
c_1(\spinct_{i,j})^2 = \mat{2i & 2j} \mat{0 & 1 \\ 1 & -2}^{-1} \mat{2i \\ 2j}
%= 4 \mat{i & j} \mat{2 & 1 \\ 1 & 0} \mat{i \\ j}
%= 4 \mat{i & j} \mat{2i+j \\ i }
%= 4(2i^2 + ij + ij)
=8i(i+j),
\]
and hence the grading shift of $F_{T_J, \spinct_{i,j}}$ is
\[
\frac{c_1(\spinct_{i,j})^2 - 2\chi(T_J) - 3\sigma(T_J)}{4} = 2i(i+j)-1.
\]
Since $\HFred(S^3_{-1}(Q_n))$ is supported in grading $0$ and $\HFred(S^3_1(Q_n))$ is supported in grading $1$, if $F_{T_{Q_n}, \spinct_{i,j}}$ is nonzero on $\HFred$, we must have $2i(i+j)-1 = 1$, hence $i(i+j)=1$, which implies that either $i=1, j=0$ or $i=-1, j=0$. This implies that $F_{T_{Q_n}, \spinct_{\pm 1,0}}$ factors through $\HFred(S^3_0(Q_n), \spincs_{\pm1})$ which is isomorphic to $\F$ since $Q_n$ is fibered. Thus, $F_{T_{Q_n}, \spinct_{1,0}}$ and $F_{T_{Q_n}, \spinct_{-1,0}}$ each have rank at most $1$. Since their sum is an isomorphism, each has rank exactly $1$. The rest of the argument proceeds exactly as in the proof of Lemma \ref{lemma:twist-spinc}.
\end{proof}

In order to prove Theorem~\ref{thm:knot-surgery-intro} about knot surgering certain four-manifolds, we will also need to understand the groups $\HFh(S^3_{-1}(Q_n^K))$. %By Proposition \ref{prop:twist-iso}, the isomorphism type of this group is independent of $n$.

\begin{proposition} \label{prop:HF-Qn}
For any integer $n$ and any knot $K \subset S^3$, we have
\begin{equation} \label{eq:HF-QnK}
\dim \HFh(S^3_{\pm1}(P_n^K)) = \dim \HFh(S^3_{\mp1}(Q_n^K)) = 1 + 4 \dim \HFKh(K).
\end{equation}
This equals $5$ when $K$ is the unknot, and is at least $13$ when $K$ is a nontrivial knot.
\end{proposition}
This will be proven in Section \ref{sec:bordered}, via an extensive bordered Floer homology computation.

\subsection{Symmetries} \label{subsec:symmetry}
We now focus on the specific example of $S^3_1(P_0)$, which will arise in several places throughout the paper. We denote this manifold by $Y_0$, and the contractible manifold $C^+_{P_0}$ by $C_0$.  Proposition \ref{prop:PQdual} provides a homeomorphism $\phi_0 \co Y_0 \to S^3_{-1}(Q_{-4})$. As noted in the proof of Lemma \ref{lemma:HFm-Pn-Qn}, $\HFm(Y_0)$ is generated over $\F[U]$ by generators $v_1,v_2,v_3$ in grading $0$, with $Uv_1 = Uv_2 = Uv_3$, and $\iota$ is given by $\iota(v_1) = v_3$, $\iota(v_2) = v_2$, $\iota(v_3) = v_1$. Such a basis is unique up to interchanging $v_1$ and $v_3$ and/or replacing $v_2$ with $v_1+v_2+v_3$.

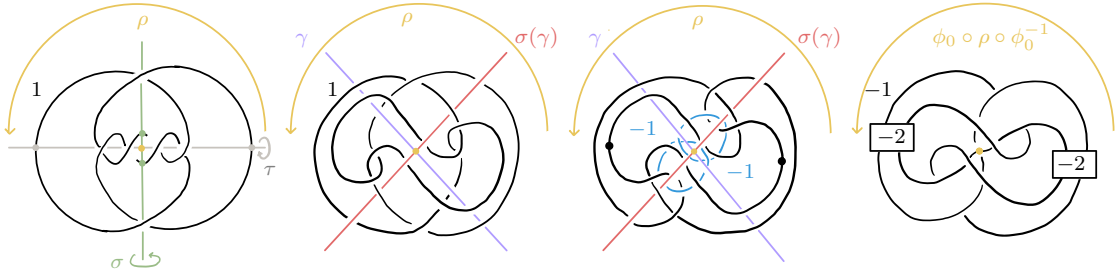
\begin{figure}[htbp]{\scriptsize
\begin{overpic}[tics=20]{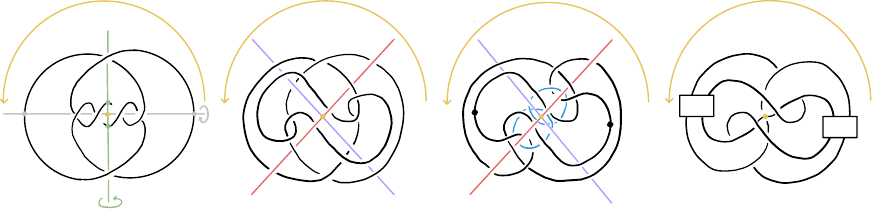}
%\put(10, 20){$1$}
%\put(35, -3){$\color{lgreen}{\sigma}$}
%\put(5, 70){$\color{lyellow}{\rho}$}
\put(10,65){$1$}
\put(50,92){$\color{lyellow}\rho$}
\put(40,0){$\color{lgreen}\sigma$}
\put(98,35){$\color{gray}\tau$}
%\put(10,85){$\color{lpurple}\gamma$}
%\put(85,85){$\color{lred}\sigma(\gamma)$}
%second frame
\put(122,65){$1$}
\put(150,93){$\color{lyellow}\rho$}
\put(110,85){$\color{lpurple}\gamma$}
\put(193,85){$\color{lred}\sigma(\gamma)$}
%third frame
\put(236,50){$\color{lblue}-1$}
\put(273,34){$\color{lblue}-1$}
\put(260,92){$\color{lyellow}\rho$}
\put(221,85){$\color{lpurple}\gamma$}
\put(300,85){$\color{lred}\sigma(\gamma)$}
%   final frame
\put(325,65){$-1$}
\put(330,47){$-2$}
\put(398,37){$-2$}
\put(351,85){$\color{lyellow}\phi_0\circ\rho\circ\phi_0^{-1}$}
  \end{overpic}}
  \caption{In the left frame we exhibit three involutions on $S^3_1(P_0)$.  The final three frames exhibit how $\rho$ interacts with the topology of our building blocks.}
  %which acts nontrivially on $\HFm(S^3_1(P_0))$.
  \label{fig:DHMsym}
\end{figure}

A special feature of $Y_0$ is that the action of its mapping class group on its Heegaard Floer homology is completely understood, thanks to the work of Dai, Hedden, and Mallick \cite{DHMCorks}.
We can isotope $P_0$ in $\R^3$ to be symmetric with respect to $180^\circ$ rotations around each of the three coordinate axes, which we denote by $\rho$, $\sigma$, and $\tau$ as shown in Figure \ref{fig:DHMsym}. Each of these gives rise to an involution of $Y_0$, which we also denote by the same symbols. Note that $\rho$, $\sigma$, and $\tau$ all commute, forming a Klein four group. (Using SnapPy \cite{SnapPy}, one can prove that $\{\operatorname{id},\sigma,\tau,\rho\}$ is in fact the whole mapping class group of $Y_0$, although we will not need this here.)

\begin{proposition} \label{prop:P0-symmetry}
With respect to any basis $(v_1,v_2,v_3)$ as above, the actions of $\sigma$ and $\tau$ on $\HFm(S^3_1(P_0))$ are given by
\begin{align*}
\rho_*(v_1) &= v_3  & \rho_*(v_2) &= v_2 & \rho_*(v_3) &= v_1 \\
\sigma_*(v_1) &= v_1  & \sigma_*(v_2) &= v_1+v_2+v_3 & \sigma_*(v_3) &= v_3 \\
\tau_*(v_1) &= v_3 & \tau_*(v_2) &= v_1 + v_2 + v_3 &  \tau_*(v_3) &= v_1.
\end{align*}
\end{proposition}

\begin{proof}
The computation for $\sigma$ and $\tau$ is given in \cite[Example 4.7]{DHMCorks}; the one for $\rho$ follows immediately by composing. Note that each of these statements is independent of the basis $(v_1,v_2,v_3)$ subject to the above conditions.
\end{proof}

\begin{remark} \label{rem:rho}
Note that the curve $\gamma = \gamma_{P_0}$ can be isotoped to be symmetric with respect to $\rho$, as shown in the second frame of Figure \ref{fig:DHMsym}. Indeed, the entire presentation of $P_0$ as a fusion number 1 knot can be made symmetric. As a consequence, the symmetry $\rho$ extends smoothly over the contractible manifold $C^+_{P_0}$, as seen in the third frame. Using the Montesinos trick, we may recognize $Y_0$ as the branched double cover of a knot in $S^3$ and $C_0$ as the branched double cover of a slice disk for this knot. (See Section \ref{sec:surfaces} for further details.)

Furthermore, the homeomorphism $\phi_0 \co S^3_1(P_0) \to S^3_{-1}(Q_{-4})$ can be constructed to be equivariant with respect to $\rho$ and a symmetry of $S^3_{-1}(Q_{-4})$ that comes from a strong inversion of $Q_{-4}$, as seen in the fourth frame of Figure \ref{fig:DHMsym}. Indeed, the $4$-dimensional extension $\phi_0 \co C^+_{P_0} \to C^-_{Q_{-4}}$ has the same property. In particular, note that $\phi_0^{-1}(\mu_{Q_{-4}})$ can be arranged to be symmetric with respect to $\rho$.
\end{remark}

\begin{remark} \label{rem:cork}
In contrast to the previous remark, $\gamma$ is not preserved under $\sigma$ or $\tau$. Furthermore, Dai, Hedden, Mallick proved using Proposition \ref{prop:P0-symmetry} that neither $\sigma$ nor $\tau$ can extend smoothly over the contractible manifold $C^+_{P_0}$, or more generally over any homology ball $Z$ with boundary $Y_0$; the key point is that $\sigma_*$ and $\tau_*$ each exchange $v_2$ with $v_1+v_2+v_3$. However, by work of Freedman \cite{Freedman}, each of these maps does extend to a continuous homeomorphism of $C^+_{P_0}$. In Section \ref{sec:simply-connected}, we will see that cutting out $C^+_{P_0}$ and regluing by either $\sigma$ or $\tau$ can produce an exotic smooth structure on $\cptwo \conn 9 \cptwobar$.

Additionally, note that $\sigma_*$ and $\rho_*$ restrict to nontrivial involutions on $\HFred(Y_0)$, while $\tau_*$ restricts to the identity on $\HFred$. We will see in Section \ref{subsec:cut-reglue} that $\sigma$ and $\rho$ (but not $\tau$) can be used to change the smooth structure of our exotic 4-manifolds with $\pi_1=\Z$.
\end{remark}

There is also an important symmetry on $S^3_0(Q_0)$ that will be crucial for constructing the exotic definite four-manifold with $\pi_1 = \Z/2$.
\begin{proposition}\label{prop:square-invol}
There is an orientation-reversing free involution $\zeta$ on $S^3_0(Q_0)$
\end{proposition}

\begin{proof}
Recall that $Q_0$ is the square knot, i.e. the connected sum of the right- and left-handed trefoils. Therefore, it is strongly negatively amphichiral, i.e., fixed under an orientation-reversing symmetry $\zeta\co S^3 \to S^3$ with two fixed points, both contained on $K$. Since $\zeta$ takes longitudes to longitudes, it induces a free involution on the surgery solid torus of $S^3_0(Q_0)$, which we also denote by $\zeta$.
\end{proof}

\section{A new four-manifold invariant from hypersurfaces }\label{sec:new-invts}

In this section, we define an elementary invariant $\adam$ of 4-manifolds with $b_1 > 0$, which we will use in subsequent sections to construct exotic 4-manifolds with $\pi_1 = \Z$, square zero spheres, and those related by knot surgery on Alexander polynomial 1 knots.  As described in the introduction, the $\adam$ invariant can be defined using a variety of Floer homology theories.  In this paper, we will restrict to Heegaard Floer homology, in particular $\HFred$ and $\HFh$, for its computability.  Ultimately, most applications will use $\HFh$, but the invariant is more easily analyzed from the perspective of $\HFred$, so we present invariants using both flavors in this section.

\subsection{The \texorpdfstring{$\adam$}{\textalpha} invariant and its applications}

\begin{definition} \label{def:adaminvt}
Let $X$ be a closed, oriented, connected smooth 4-manifold, and let $\eta$ be a primitive element of $H_3(X)$. We define $\adam(X, \eta)$ to be the minimal $\F$-dimension of $\HFred(Y)$, where $Y$ is a smoothly embedded, closed, connected, oriented $3$-manifold representing the homology class $\eta$.
\end{definition}

Because $\dim \HFred(Y) = \dim \HFred(-Y)$, it is immediate that $\adam(X, \eta) = \adam(X,-\eta)$. Moreover, note that $\adam(X,\eta)$ does not depend on the orientation of $X$; in this regard, it is very different from most other gauge-theoretic invariants of $4$-manifolds. In the case where $b_1(X) = 1$, we simply write $\adam(X)$ for $\adam(X,\eta)$, where $\eta$ is either generator of $H_3(X) \cong \Z$.

\begin{example}\label{ex:trivial-alpha}
Let $X = {S^1 \times S^3} \conn Z$, where $Z$ is any closed, oriented $4$-manifold with $b_1(Z)=0$. Then $H_3(X) \cong \Z$, with the generator represented by $S^3$, and hence $\adam(X) = 0$.
\end{example}

Invariants defined by taking a minimum over all embedded representatives of a homology class are typically difficult to compute (e.g. the minimal genus of a surface representing a fixed second homology class in a 4-manifold). However, the following proposition provides a straightforward criterion to show that some $Y$ representing $\eta$ in fact determine $\adam(X,\eta)$, in certain settings.

\begin{proposition}\label{prop:adam-iso}
Let $Y$ be a closed, connected, oriented $3$-manifold, and suppose $W$ is a cobordism from $Y$ to itself for which the map
\[
F_W \co \HFred(Y) \to \HFred(Y)
\]
is an isomorphism (summing over all spin$^c$ structures). Let $X$ be a closed, oriented $4$-manifold obtained by gluing the ends of $W$ by a diffeomorphism, and let $\eta$ denote the class of $Y$ in $H_3(X)$.  Then $\adam(X, \eta) = \dim \HFred(Y)$; that is, $Y$ minimizes the dimension of $\HFred$ in its homology class. Moreover, if $\HFred(Y) \ne 0$, then  $X$ does not contain any embedded $2$-spheres with self-intersection $1$ or $-1$.
\end{proposition}

%
%If $ is an isomorphism, then $\adam(X, [Y]) = \dim HF_{red}(Y)$.  If $Y$ is admissible in $X$, then $\adam(X, \mft, [Y]) = \dim HF_{red}(Y,\mfs)$ if $F^-_{W,\mft \mid_W}$ is an isomorphism on reduced Floer homology.

%(For the Spin$^c$-phobic mathematician, we do not need these at all for the applications in this section, and so should not think of these at all.  The more Spin$^c$-conscious should note that we are summing over all possible Spin$^c$ structures, and so the composition law $F_{W \circ W'} = F_W' \circ F_W$ holds for cobordism maps.)  Note also that since $\dim HF_{red}(-Y) = \dim HF_{red}(Y)$, the generator of $H_3$ for which $\adam(X)$ is minimized can also be replaced by the opposite sign.

%Much like the Seiberg-Witten invariants, the $\adam$ invariant can show that certain 4-manifolds do not contain $\cptwo$ or $\cptwobar$ summands.
%\begin{proposition}
%Let $W : Y \to Y$ be a cobordism of $b_1(Y) = 0$ such that $F^-_W$ induces an isomorphism on $HF_{red}(Y)$.  Let $X$ be the closed up 4-manifold obtained by gluing up the ends.  Then, $X$ does not contain any $-1$-spheres or $+1$-spheres.
%\end{proposition}
%\begin{proof}
%
%\end{proof}
%In fact, more generally, one can show that $X$ cannot contain a connected summand with $H_1 = 0$ and non-trivial $H_2$.

\begin{proof}
Suppose that $Z$ is a hypersurface in $X$ with $[Z] = [Y]$ in $H_3(X)$. In the special case where $Z$ is disjoint from $Y$, $Z$ then separates $W$ into a cobordism from $Y$ to $Z$ followed by a cobordism from $Z$ to $Y$. Then $F_W \co \HFred(Y) \to \HFred(Y)$ is an isomorphism which necessarily factors through $\HFred(Z)$.  It follows that $\dim \HFred(Z) \geq \dim \HFred(Y)$.

In the general case, $Z$ and $Y$ will intersect. Let $p \co \tilde X \to X$ be the infinite cyclic cover corresponding to $PD[Y] \in H^1(X)$, obtained by stacking infinitely many copies of $W$ end to end. For any loop $\gamma \subset Z$, we have \[
\gen{PD[Y],[\gamma]} = [\gamma] \cdot [Y] = [\gamma] \cdot [Z] = 0.
\]
Thus, $Z$ lifts to an embedded, separating hypersurface in $\tilde X$. Since $Z$ is compact, it must be contained in the composition of $n$ copies of $W$, for some $n$.  Write this as $W^n$.  The map $F_{W^n} \co \HFred(Y) \to \HFred(Y)$ is just $(F_W)^n$, hence still an isomorphism, and it factors through $\HFred(Z)$.  Once again, we see that $\dim \HFred(Z) \geq \dim \HFred(Y)$, as required.

For the final statement, suppose $S$ is a $\pm 1$-sphere in $X$. As before, a compactness argument shows that $S$ lifts to an embedded sphere in $W^n$ for some $n$, so write $W^n = N \conn \pm \cptwo$.  Then $F_{W^n}$ factors through the 2-handle cobordism map associated to a $+1$ or $-1$-framed unknot (i.e, the cobordism $(Y \times I) \conn \pm \cptwo$), which is necessarily zero. (Specifically, in the $+1$ case, every individual spin$^c$ summand vanishes, while in the $-1$ case, the blowup formula \cite[Theorem 3.7]{OSz4Manifold} implies that the summands come in cancelling pairs.) If $\HFred(Y) \ne 0$, this contradicts the fact that $F_{W^n}$ is an isomorphism on $\HFred(Y)$.
\end{proof}

For example, if $W = Y \times I$, then $F_W$ is the identity map. Therefore, if $X$ is any $Y$-bundle over $S^1$, then Proposition~\ref{prop:adam-iso} implies $\adam(X,[Y]) = \dim \HFred(Y)$.

However, for the purposes of constructing exotic $4$-manifolds, we will want to restrict our attention to manifolds which have been topologically classified, such as when $\pi_1(X) \cong \Z$.  To arrange this,  the following proposition will be useful:

\begin{proposition} \label{prop:pi1Z}
Let $Y$ be a homology sphere. Let $W$ be a simply-connected cobordism from $Y$ to itself, and assume that the intersection form of $W$ is odd. Let $X$ be the manifold obtained by gluing the ends of $W$ by some diffeomorphism. Then $X$ is homeomorphic to $S^1 \times S^3 \conn m \cptwo \conn n \cptwobar$, where $m = b_2^+(W)$ and $n= b_2^-(W)$. Thus, if $\HFred(Y) \ne 0$ and $W$ induces an isomorphism on $\HFred(Y)$, then $X$ is an exotic $S^1 \times S^3 \conn m \cptwo \conn n \cptwobar$.
\end{proposition}

\begin{proof}
Because $\pi_1(W)$ is trivial and $\pi_1(X)$ is an HNN extension of $\pi_1(W)$,  we see that $\pi_1(X) \cong \Z$.  We would like to show that $X$ is homeomorphic to $S^1 \times S^3 \conn m \cptwo \conn n \cptwobar$.  By work of \cite{FQ, Wang},  two smooth, closed, oriented 4-manifolds with $\pi_1 = \Z$ are homeomorphic if and only if the equivariant intersection forms on $H_2$ are isomorphic.  Since the intersection form on $H_2(S^1 \times S^3 \conn m \cptwo \conn n \cptwobar ;\Z[\Z])$ is given by $m\langle 1 \rangle \oplus n \langle -1 \rangle$, we just need to see the same for $H_2(X;\Z[\Z])$.

First, since $W$ is a self-cobordism of homology spheres, the intersection form of $W$ is diagonalizable.
(The indefinite case is well-known while the definite case follows from applying Donaldson's diagonalizability theorem to the closed-up 4-manifold \cite{DonaldsonArbitrarypi1}.)
Therefore, the (integer) intersection form of $W$ is equivalent to that on $m \cptwo \conn n \cptwobar$.  The universal cover of $X$ is obtained by gluing infinitely many copies of $W$ end to end.  Since $Y$ is a homology sphere, we see that $H_2(X;\Z[\Z])$ is naturally identified with $H_2(W) \otimes \Z[t,t^{-1}]$, and the intersection form respects this splitting, i.e. if $\alpha, \beta \in H_2(W)$, then $\alpha \cdot t^k \beta = 0$ unless $k = 0$.   Thus, the equivariant intersection form on $H_2(X;\Z[\Z])$ is $m \langle 1 \rangle \oplus n \langle -1 \rangle$ as a form over $\Z[\Z]$.  This completes the proof.
\end{proof}

\begin{remark}
Similar to gauge-theoretic invariants of 4-manifolds, the proof of Proposition \ref{prop:adam-iso} provides constraints on the self-intersections of surfaces based on their genera.  However, the statements are more technical than the usual adjunction inequality as one needs to understand the lift of the surface to the infinite cyclic cover.
\end{remark}

\subsection{A variant for \texorpdfstring{$\HFh$}{HF-hat}}
We will need an analogue of the $\adam$-invariant using $\HFh$ in order to detect the effect of Fintushel-Stern knot surgery on the manifolds $X_p$ as in Theorem~\ref{thm:knot-surgery-intro}.  Recall that $\HFh(Y)$ is a finite-dimensional $\F$-vector space that we associate to a three-manifold and it fits into a long exact sequence
\[
\ldots \to \HFh(Y) \to \HFm(Y) \xrightarrow{U} \HFm(Y) \to \ldots
\]
making $\HFh(Y)$ isomorphic to the direct sum of the kernel of $U$ and the cokernel of $U$.  In the case that $Y$ is an integer homology sphere, if $\HFm(Y) \cong \F[U] \oplus \bigoplus_{i=1}^n \F[U]/U^{b_i}$, then $\dim \HFh(Y) = 1 + 2n$.  Like $\HFm$, $\HFh$ is functorial under cobordisms and so is the long exact sequence above.  Consequently, if $W:Y \to Z$ induces an isomorphism on $\HFred$, then the rank of $\widehat{F}_W$ is at least $2n$.

We now consider the $\adam$-invariant defined using $\HFh$.
\begin{definition}
Let $X$ be a closed 4-manifold with $b_1(X) \geq 1$ and choose a primitive element $\eta \in H^1(X)$.  We define $\adamhat(X,\eta)$ to be the minimal dimension of $\HFh(Y)$ where $Y$ is a non-separating hypersurface satisfying $PD[Y] = \eta$.
\end{definition}

We have the following analogue of Proposition~\ref{prop:adam-iso}, which still relies on the cobordism map on $\HFred$.
\begin{proposition}\label{prop:adamhat-iso}
Let $Y$ be a homology sphere, and suppose $W$ is a cobordism from $Y$ to itself for which the map
\[
F_W \co \HFred(Y) \to \HFred(Y)
\]
is an isomorphism (summing over all spin$^c$ structures). Let $X$ be the closed $4$-manifold obtained by gluing the ends of $W$ by a diffeomorphism, and let $\eta$ denote the class of $Y$ in $H_3(X)$. Then $\adamhat(X, \eta) = \dim \HFh(Y)$.
\end{proposition}
\begin{proof}
The proof is similar to Proposition~\ref{prop:adam-iso}.  Suppose $Z$ is another hypersurface in $X$ homologous to $Y$, and so $Z$ embeds in the $k$-fold composition of $W$ as a separating hypersurface for some $k$.  Write $\HFred(Y) = \bigoplus_{i=1}^n \F[U]/U^{k_i}$, so that $\dim \HFh(Y) = 1+ 2n$.  By assumption, $F_{W^k}$ factors through $\HFred(Z)$, which implies that $\HFred(Z)$ has at least $n$ cyclic summands, and thus $\dim \HFh(Z) \geq 2n$.  Finally, $\HFm(Z)$ contains at least one free $\F[U]$-summand by \cite{LidmanThesis} (or, alternately, by the analogous result in monopole Floer homology \cite[Corollary 35.1.4]{KMBook} together with the $HF=HM$ isomorphism  \cite{CGH, KutluhanLeeTaubes, Taubes}).  Therefore, we see that in fact $\dim \HFh(Z) \geq 2n+1$.  This implies $\adamhat(X,\eta) = \dim \HFh(Y)$.
\end{proof}

\begin{remark}
We have defined $\adam$ using the total dimension of $\HFred$ or $\HFh$, but one could use other aspects of Floer homology, such as the minimal $U$-torsion in $\HFred$.  One can also define $\adam$ invariants for 4-manifolds with boundary, say using sutured Floer homology of properly embedded non-separating hypersurfaces.  It could also be interesting to study a spin$^c$ refinement of $\adam$.
%Finally, Lipshitz observed that one can do a similar construction using Khovanov homology or knot Floer homology for surfaces in $S^3 \times S^1$.  It would be interesting to construct some exotic surfaces in $S^3 \times S^1$ (or other 4-manifolds) using these techniques.
\end{remark}

\section{Exotic 4-manifolds via the \texorpdfstring{$\adam$}{\textalpha} invariant }\label{sec:construct}

We now use the $\adam$ invariant and the building blocks from Section~\ref{sec:building-blocks} to produce a variety of exotic 4-manifolds with $b_1 > 0$, proving Theorems \ref{thm:pi1Z-intro}, \ref{thm:knot-surgery-intro}, and \ref{thm:zerospheresintro} from the introduction.

\subsection{Exotica with \texorpdfstring{$\pi_1=\Z$}{\textpi\_1=Z}}\label{subsec:pi1Z}
In this section we construct our first examples of what we will later show are exotic manifolds.  This was given loosely in the introduction as Construction \ref{constr:pi1Z-intro}; we revisit it now in the language of the building blocks in Section \ref{subsec:cobs}.

The main idea of the construction is to use the elementary cobordisms in Section \ref{sec:building-blocks} to build a simply-connected cobordism $\W$ from $Y_0$ to itself.
%such that the induced map $F^-_W : HF^-_{red}(Y) \to HF^-_{red}(Y)$ is an isomorphism.
Our $\pi_1=\Z$ manifold will be formed from $\W$ by identifying the boundary components using an orientation preserving homeomorphism.
%and the condition on $F^-_W$ will lead to the computation of the invariant.
\vspace{12pt}
\begin{figure}[htbp]
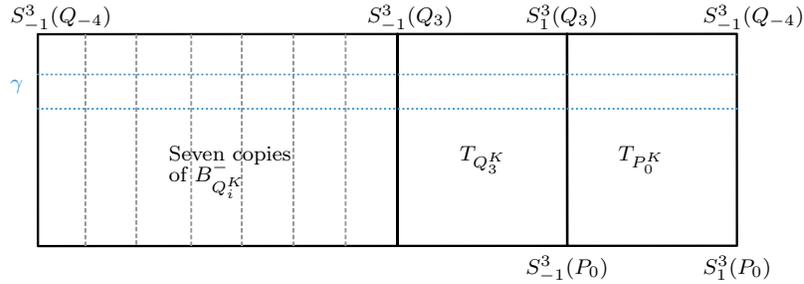
{\scriptsize
\begin{overpic}[tics=20]{Pi1Zintro}
    \put(-10, 85){$S^3_{-1}(Q_{-4})$}
        \put(125, 85){$S^3_{-1}(Q_{3})$}
         \put(185, 85){$S^3_{1}(Q_{3})$}
                \put(185, -10){$S^3_{-1}(P_{0})$}
                                \put(252, -10){$S^3_{1}(P_{0})$}
                                  \put(252, 85){$S^3_{-1}(Q_{-4})$}
        \put(-10, 60){\color{lblue}$\gamma$}
     \put(50, 33){Seven copies}
       \put(50, 25){of $B^-_{Q^K_{i}}$}
     \put(160, 33){$T_{Q^K_3}$}
                   \put(220, 33){$T_{P^K_0}$}
  \end{overpic}}
  \caption{The cobordism $\W$ from Construction \ref{constr:pi1Z}.}\label{fig:pi1Z}
  \end{figure}

\begin{construction}\label{constr:pi1Z}
%Throughout the construction, refer to Figure \ref{fig:pi1z}.
For any knot $K$, let $\W^K \co S^3_{-1}(Q^K_{-4}) \to S^3_{-1}(Q^K_{-4})$ be the cobordism defined in Figure \ref{fig:pi1Z}.
%by the following composition:
%\begin{equation} \label{eq:self-cobordism}
%\xymatrix@C=0.35in{
%S^3_{-1}(Q^K_{-4}) \ar[r]^-{B^-_{Q^K_{-4}}} &
%S^3_{-1}(Q^K_{-3}) \ar[r]^-{B^-_{Q^K_{-3}}} &
%\cdots \ar[r]^-{B^-_{Q^K_{2}}} &
%S^3_{-1}(Q^K_{3}) \ar[dll]^(.3){T_{Q^K_3}} \\
%& S^3_{1}(Q^K_{3}) \ar@{=}[r]_{\psi^K_3} &
%S^3_{-1}(P^K_0) \ar[r]_{T_{P^K_0}} & S^3_1(P^K_0) \ar@{=}[r]_{\phi^K_0} & S^3_{-1}(Q^K_{-4})
%}
%\end{equation}
In words, we first stack seven cobordisms $B^-_{Q^K_{i}}$ for $i=-4, \dots, 2$, followed by the cobordism $T_{Q^K_3}$, to obtain a cobordism from $S^3_{-1}(Q^K_{-4})$ to $S^3_1(Q^K_3)$. We then identify $S^3_1(Q^K_3)$ with $S^3_{-1}(P^K_0)$ via the homeomorphism $\psi^K_3$ from Proposition \ref{prop:PQdual}, and attach a copy of $T_{P^K_0}$; the resulting cobordism has top boundary $S^3_1(P^K_0)$. Identifying $S^3_1(P^K_0)$ with $S^3_{-1}(Q^K_{-4})$ using the homeomorphism $\phi^K_0$, we thus have a cobordism from $S^3_{-1}(Q^K_{-4})$ to itself.

Finally, define $\W^K_p$ to be $p$ stacked copies of $\W^K$, and define $\X^K_p$ to be formed from $\W^K_p$ by gluing $\partial^+$ to $\partial^-$ by the identity map. When $K$ is the unknot, we will just denote these as $\W_p$ and $\X_p$ respectively.  If $p = 1$, and $K$ is still possibly knotted, we denote these as $\W^K$ and $\X^K$.
\end{construction}

%
%We now show that $X^K_p$ are related by knot surgery in $X_p$ along a particular framed embedded torus $T$, as promised in the introduction.  We will also demonstrate that our $T$ is embedded in a fishtail neighborhood which is embedded in $X_p$.
%%footnote{LP: I think teasing this result here gets buried, perhaps we should state it in a theorem enviornment.  TL: I also put this fact about the fishtail neighborhood into the theorem statement in the intro.LP:happy}

\begin{remark} \label{rem:pi1Z}\noindent
\begin{enumerate}
\item It follows readily from the construction that $\W^K$ is simply connected and that its intersection form is isomorphic to $2\gen{1} \oplus 9 \gen{-1}$. By Proposition \ref{prop:pi1Z}, we deduce that $\X^K_p$ is homeomorphic to $S^1 \times S^3 \conn 2p \cptwo \conn 9p \cptwobar$, independent of the knot $K$.

%
%
%$\pi_1(X^K_p)\cong \Z$ and that there is a graded isomorphism from $H_{*}(\conn {2p}\cptwo\conn {9p}\cptwobar)$, preserving inter
%
%
%
%    to the relative homology $H_{*}(W_p^K,\partial)$ and hence from~$H_{*}(\conn {2p}\cptwo\conn {9p}\cptwobar\# S^1\times S^3)$ to $H_{*}(X_p^K)$.
%
%
%
%\item the equivariant intersection form $\lambda_{X^K_p}$ is identical to that of $\conn {2p}\cptwo\conn {9p}\cptwobar\# S^1\times S^3$.  It then follows from \cite{FQ, SW}\footnote{TL: Finalize citation} that $X_p^K$ is homeomorphic to $\conn {2p}\cptwo\conn {9p}\cptwobar\# S^1\times S^3$, and this is independent of the knot $K$.
\item To prove Theorem \ref{thm:pi1Z-intro}, one can glue up the ends of $\W_p$ with any orientation-preserving homeomorphism.  However, for Proposition \ref{prop:fishtail} we require that $\X_p$ be defined using a gluing that preserves the meridian of the ribbon band to realize this as a knot surgery along a suitable torus.

\end{enumerate}

\end{remark}

\begin{proposition}\label{prop:fishtail}
There is an embedded torus $T \subset \X_p$ with the property that for each knot $K$, $\X_p^K$ is obtained by knot surgery on $T$ using $K$. Moreover, $T$ is contained in a fishtail neighborhood embedded in $\X_p$.
\end{proposition}

\begin{proof}
We begin with identifying the torus $T$. Observe that in each of the elementary cobordisms $T_{J}$ and $B_J$, there is a concordance between the $\gamma$ curves in the two ends. Further, Proposition~\ref{prop:PQdual} tells us that the homeomorphisms $\phi^K_n$ and $\psi^K_n$ also preserve the $\gamma$ curves. Thus, there is a natural concordance in $\W$ between the $\gamma$ curves in the two boundary components. When we glue $p$ copies of $\W$ and close up with the given identification of the ends, following these identifications, the concordance closes up to an embedded torus $T$ (and not a Klein bottle). Moreover, because the homeomorphisms from Proposition \ref{prop:PQdual} preserve the 0-framing on $\gamma$, $T$ has self intersection 0.

Because the $-1$ framed $2$-handles of the elementary cobordisms $B_J$ are attached along a parallel of $\gamma$,  we can think of attaching one of these 2-handles to the tubular neighborhood $T^2\times D^2$ of $T$ to see $T$ as the fiber of a fishtail neighborhood smoothly embedded in $\X_p$.

Now, we perform knot surgery on $T$ with knot $K$ and direction $\gamma$. Precisely, we replace the neighborhood of $T$, which we can parametrize as $\gamma\times S^1_c\times D^2$, with $S^3\setminus\nu(K)\times S^1_d$ using the gluing map which takes $\gamma$ to $\mu_K$,  $S^1_c$ to $S^1_d$, and $\mu_T$ to $\lambda_K$.  This surgery can be thought of as replacing $\nu(\gamma)$ with $S^3\setminus\nu(K)$ at every level set of $\W$ (with respect to the height function we used to define $\W$).  At every level set, the gluing restricts to send $\mu_\gamma$ to $\lambda_K$ and $\lambda_\gamma$ to $\mu_K$.  This has exactly the effect of turning $\W$ into $\W^K$, and hence $\X_p$ into $\X_p^K$.
\end{proof}

The next result we prove immediately implies Theorem~\ref{thm:pi1Z-intro} and Theorem~\ref{thm:knot-surgery-intro}.

\begin{theorem}\label{thm:pi1Z}
For any knot $K$ and $p \ge 1$,
$\X_p^K$ is an exotic~$S^1 \times S^3 \conn {2p} \cptwo \conn {9p} \cptwobar$. Furthermore, for any two knots, $K, K'$, if $\dim \HFKh(K) \neq \dim \HFKh(K')$, then $\X_p^K$ and $\X_p^{K'}$ are not diffeomorphic.  In particular, $\X_p$ is not diffeomorphic to $\X_p^K$ for any nontrivial knot $K$.
\end{theorem}
\begin{proof}
First, because each of the building block cobordisms induces an isomorphism on $\HFred$ by Propositions \ref{prop:twist-iso} and ~\ref{prop:tye-iso}, so does the composite cobordism $\W^K$. Thus, by Propositions \ref{prop:HF-Qn} and \ref{prop:adamhat-iso}, we deduce that
\begin{equation} \label{eq:adam(Xpk)}
\adamhat(\X_p^K) = \dim \HFh(S^3_{-1}(Q_{-4}^K)) = 1 + 4 \dim \HFKh(K).
\end{equation}
In particular, $\adamhat(\X_p) = 5$, and for every nontrivial knot $K$, $\adamhat(\X_p^K) \ge 13$, since $\dim \HFKh(K) \geq 3$ for any non-trivial knot. Since $\adamhat(S^1 \times S^3 \conn {2p} \cptwo \conn {9p} \cptwobar)=0$, we deduce that $\X_p^K \not \cong S^1 \times S^3 \conn {2p} \cptwo \conn {9p} \cptwobar$ for any $K$ and that $\X_p^K \not \cong \X_p^{K'}$ if $ \dim \HFKh(K) \neq \dim \HFKh(K')$.
\end{proof}

\begin{remark}\label{rmk:easy-adam-argument}
If one is interested in only showing that each $\X^K_p$ is an exotic $S^1 \times S^3 \conn {2p} \cptwo \conn {9p} \cptwobar$, the argument does not require serious Heegaard Floer homology computations (eg. computing $\HFh(S^3_{-1}(Q^K_{-4}))$ using Proposition~\ref{prop:HF-Qn}).  It is well-known that $\HFred(S^3_{-1}(J)) \neq 0$ for any nontrivial slice knot $J$, and the exact triangle (specifically Propositions~\ref{prop:twist-iso} and \ref{prop:tye-iso}) implies that $\W^K_p$ induces an isomorphism on $\HFred$. Thus, Proposition~\ref{prop:adam-iso} implies that $\alpha(\X^K_p) \neq 0$ for all $K, p$.
\end{remark}

\subsection{The circle sum construction and exotica with square-0 spheres}

Choose a small $B^3$ in $S_{-1}^3(Q^K_{-4})$ away from the surgery solid torus, and observe that there is a natural vertical cobordism $B^3\times I$ in $\W_p^K$ gotten by avoiding all the handles of $\W_p^K$.  This gives rise to a natural framed circle embedding $\Gamma$ in $\X_p^K$. The framing gives a parametrization of the boundary of a regular neighborhood of $\Gamma$ given by $\partial B^3\times S^1$.

\begin{construction} \label{constr:circlesum}
For any 3-manifold $M$, define the circle sum $\X^K_p \conn_{S^1} (M\times S^1)$
to be the manifold obtained by gluing $\X^K_p\setminus \nu(\Gamma)$ and $(M \setminus B^3) \times S^1 $ along their boundaries using a gluing homeomorphism which acts as the identity map on natural parametrizations of $\partial\nu(C)$ and $\partial((M\setminus B^3)\times S^1))$. For brevity, we will write $\X_p^K(M)$ for $\X^K_p \conn_{S^1} (M\times S^1)$.

In particular, when $M=S^2\times S^1$,  define
\begin{equation} \label{eq:SKp}
\S^K_p =\X^K_p\#_{S^1} (S^2\times S^1 \times S^1)\cong \X^K_p\#_{S^1} (S^2\times T^2).
\end{equation}
Observe that $\S^K_p$ has the homology type of $S^2\times T^2 \conn 2p \cptwo \conn 9p\cptwobar$, and additionally, there is a non-trivial class in $H_2$ represented by an embedded 2-sphere with self intersection $0$ which lifts to all covers.
\end{construction}

The following theorem is a generalization of Theorem~\ref{thm:zerospheresintro} from the introduction.

\begin{theorem}\label{thm:zerospheres-adam}
Let $M$ be $S^2 \times S^1$, $T^3$, or a 3-manifold with $b_1(M) = 0$.  Then, for any knot $K$ and any $p \ge 1$, $\X^K_p(M)$ is an exotic $(M \times S^1) \conn 2p \cptwo \conn 9p \cptwobar$.
\end{theorem}
\begin{proof}
First, we show that $\X^K_p(M)$ is homeomorphic to $(M \times S^1) \conn 2p \cptwo \conn 9p \cptwobar$.  Recall from Remark \ref{rem:pi1Z} that there is a homeomorphism $h$ from $\X^K_p$ to $S^3\times S^1 \conn 2p \cptwo \conn 9p \cptwobar$.  Since homotopy implies isotopy for curves in 4-manifolds, we can assume that the homeomorphism maps $\Gamma$  onto $\{\pt\}\times S^1$.  If $h$ respects the vertical framings on $\Gamma$ and $\{\pt \}\times S^1$, then $h$ extends to a homeomorphism from $\X^K_p(M)$ to $(M \times S^1) \conn 2p \cptwo \conn 9p \cptwobar$.  If $h$ does not respect the framings, define a self-homeomorphism $g$ of $S^3 \times S^1 \conn 2p \cptwo \conn 9p \cptwobar$ which Dehn twists the $S^3$ factor as one walks in the $S^1$ direction. Then $g\circ h$ respects the vertical framings on $\Gamma$ and $\{\pt\}\times S^1$, so it extends to a homeomorphism from $\X^K_p(M)$ to $M \times S^1  \conn 2p \cptwo \conn 9p \cptwobar$.

For the $M$ in the hypotheses of the theorem, every primitive element of $H_3(M \times S^1 \conn 2p \cptwo \conn 9p \cptwobar)$ is represented by $M$. This is immediate when $b_1(M)=0$; in the other two cases, every primitive element of $H_3(S^2 \times T^2)$ (resp.~$H_3(T^4)$) can be represented by an embedded $S^2 \times S^1$ (resp. $T^3$), and this remains true after blowing up. It therefore suffices to show that $\alpha(\X^K_p(M), \eta)  > \dim \HFred(M)$, where $\eta$ is the obvious class represented by $S^3_{-1}(Q_{-4}^K) \conn M$.

To prove this, first note that $\X^K_p(M)$ can alternatively be described as taking a pathwise connected sum of $\W_p^K$ and $M \times I$, and closing up the ends.  Call this pathwise connected sum cobordism $\W^K_p(M)$; it is a self-cobordism of $M \conn S^3_{-1}(Q_{-4}^K)$.  Since $F_{\W_p^K}$ is an isomorphism on $\HFred(S^3_{-1}(Q_{-4}^K))$, as seen in the proof of Theorem \ref{thm:pi1Z}, the behavior of the Heegaard Floer homology cobordism maps under pathwise connected sum \cite[Proposition 4.4]{OSzAbsolute} implies that $F_{\W^p_K(M)}$ is injective on the submodule
\[
\HFred(S^3_{-1}(Q_{-4}^K)) \otimes_{\F[U]} \HFm(M) \subset \HFred(S^3_{-1}(Q_{-4}^K) \conn M).
\]
Moreover, the same is true for all powers of $F_{\W^p_K(M)}$. The same argument as in the proof of Proposition~\ref{prop:adam-iso} shows that 
\[
\adam(\X_p^K(M), \eta) \ge \dim_\F \left( \HFred(S^3_{-1}(Q_{-4}^K)) \otimes_{\F[U]} \HFm(M) \right),
\]
which is strictly larger than $\dim_\F \HFred(M)$. 
\end{proof}

\begin{remark}
One can define a $\Z/2$-graded $\adam$-invariant.  Define $\adam_0$ to be the minimal dimension of the even-graded part of $\HFred$, and similarly for $\adam_1$.  We have that $\adam_0(\X_p) = \adam_1(\X_p) = 0$, using the cuts $S^3_{-1}(P_0)$ and $S^3_{-1}(Q_0)$ respectively, so this invariant does not distinguish $\X_p$ from $S^1 \times S^3 \conn 2p \cptwo \conn 9p \cptwobar$.  However, for a nontrivial knot, we claim that these invariants can distinguish $\X_p^K$ from
$S^1 \times S^3 \conn 2p \cptwo \conn 9p \cptwobar$.  The argument is as follows.  Recall that for a homology sphere, $\lambda(Y) = -d(Y)/2 - \chi(\HFred(Y))$ by \cite[Theorem 1.3]{OSzAbsolute}, where $\lambda$ is the Casson invariant.  (To get the signs correct, we use that the Euler characteristic of $\HFred(Y)$ changes sign when computing with $\HFm$ instead of $\HFp$.)  Because $d(S^3_{-1}(Q_0)) = 0$ and $\chi(\HFred(S^3_{-1}(Q_0))) = 2$, we thus have that $\lambda(S^3_{-1}(Q_0)) = -2$.  Since $d(S^3_{-1}(Q_0^K)) = 0$ and the Casson invariant is unchanged under splicing with knots in $S^3$ \cite{BoyerNicas, FukuharaMaruyama}, $\lambda(S^3_{-1}(Q_0^K)) = -2$ for all $K$.  Therefore, we see that $\chi(\HFred(S^3_{-1}(Q_0^K))) = 2$ for all $K$.  Using splittings along $S^3_{-1}(Q_0^K)$ and $S^3_{-1}(P_0^K)$, the previous work of this section shows that
\[
\adam_1(\X_p^K) = \adam_0(\X_p^K) = \frac{\adam(\X_p^K) - 2}{2} = \frac{\dim \HFred(S^3_{-1}(Q_0^K)) - 2}{2} \neq 0.
\]
This is non-zero because $\dim \HFh(S^3_{-1}(Q_0^K)) > 5$ implies $\dim \HFred(S^3_{-1}(Q_0^K)) > 2$. As a topological consequence, we see that the generator of $H_3(\X_p^K)$ cannot be represented by any Seifert fibered homology sphere, as $\HFred$ of such a manifold is supported in a single parity of gradings by \cite[Corollary 1.4]{OSzPlumbed}. 

It is likely one could build 4-manifolds analogous to $\X_p^K$ stemming from different fusion number one knots whose $\pm 1$-surgeries have larger Casson invariants and produce exotic 4-manifolds with the same $\adam$ but different $\adam_i$.  We do not pursue this here.
\end{remark}

\subsection{Cutting and regluing} \label{subsec:cut-reglue}

Next, we consider a variation of the above construction when $K$ is the unknot and $p=1$. Recall the commuting involutions $\{\rho,\sigma,\tau\}$ on $S^3_1(P_0)$, shown in Figure \ref{fig:DHMsym}. Identifying $S^3_1(P_0)$ with $S^3_{-1}(Q_{-4})$ by the homeomorphism $\phi_0$, these may also be viewed as involutions on $S^3_{-1}(Q_{-4})$. For each $f \in \{\rho, \sigma, \tau\}$, let $\X^f$ be the manifold obtained by gluing the ends of $\W$ using $f$ rather than the identity. The algebraic topology discussion in Remark \ref{rem:pi1Z} goes through identically in this case, so $\X^f$ is homeomorphic to $\X$.

\begin{proposition}\label{prop:pi1Z-trace}
The manifolds $\X$ and $\X^\sigma$ are homeomorphic but not diffeomorphic.
\end{proposition}

\begin{proof}
We will distinguish between $\X$ and $\X^\sigma$ using the Ozsv\'ath--Szab\'o 4-manifold invariant from \cite{OSz4Manifold}, which is defined for any closed $4$-manifold with $b^+ \ge 2$, and can be recast as a trace due to Zemke \cite{ZemkeDuality}.

First, we claim that there exists a conjugate pair of spin$^c$ structures $\spincs_0, \bar\spincs_0$ on $\W$, and a pair of good (ordered) bases $(a,b)$ and $(a',b')$ for $\HFred(S^3_{-1}(Q_{-4}))$ such that
\begin{align*}
F_{\W,\spincs_0}(a) &= a' & F_{\W, \spincs_0}(b) = 0 \\
F_{\W,\bar\spincs_0}(a) &= 0 & F_{\W, \bar \spincs_0}(b) = b',
\end{align*}
and that for every other spin$^c$ structure $\spincs$ on $W$, $F_{\W,\spincs}=0$ on $\HFred$. These two bases are necessarily equal as sets but not necessarily as ordered sets.

The claim follows from repeated use of Lemmas \ref{lemma:twist-spinc} and \ref{lemma:tye-spinc}. Specifically, for the first 8 cobordisms in Construction~\ref{constr:pi1Z} (namely $B^-_{Q_i}$ for $i=-4, \dots, 2$ and $T_{Q_3}$), there are exactly two spin$^c$ structures (which are conjugates) that induce nontrivial maps on $\HFred$. Moreover, in any pair of consecutive cobordisms, only two of the four possible pairs of such spin$^c$ structures compose nontrivially. Inductively, we see that there are only two spin$^c$ structures on the composition of these 8 cobordisms inducing nonzero maps, and that the maps behave as indicated on good bases. Furthermore, this behavior is unchanged after composition with the isomorphism $F_{T_{P_0}}$, which is supported in a single spin$^c$ summand, and under the isomorphisms $\phi_{0*} \co \HFred(S^3_1(P_0)) \to \HFred(S^3_{-1}(Q_{-4}))$ and $\psi_{3*} \co \HFred(S^3_1(Q_3)) \to \HFred(S^3_{-1}(P_0))$.

Since the sets $\{a,b\}$ and $\{a',b'\}$ are equal, one of the following is true:
\begin{enumerate}
\item $a'=a$ and $b'=b$; or
\item $a'=b$ and $b'=a$.
\end{enumerate}
The formal principles used up until here are not sufficient to determine which of these cases holds. Additionally, note that by Proposition \ref{prop:P0-symmetry}, the isomorphism $\sigma_* \co \HFred(S_{-1}(Q_{-4})) \to \HFred(S_{-1}(Q_{-4}))$ exchanges $a$ and $b$.

For each spin$^c$ structure $\spincs$ on $\W$, let $\spinct_\spincs$ (resp.~$\spinct^\sigma_\spincs$) denote the unique spin$^c$ structure on $\X$ (resp.~$\X^\sigma$) that restricts to $\spincs$. (A Mayer--Vietoris argument shows that this exists and is unique.) By a theorem of Zemke \cite[Theorem 1.1]{ZemkeDuality}, the Ozsv\'ath--Szab\'o 4-manifold invariants of $(\X,\spinct_\spincs)$ and $(\X^\sigma,\spinct_\spincs^\sigma)$ are determined by the Lefschetz number of $F_{\W,\spincs}$, as follows:
\begin{align*}
  \Phi_{\X,\spinct_\spincs} &= \operatorname{Lef}(F_{\W,\spincs}) \\
  \Phi_{\X^\sigma,\spinct_\spincs^\sigma} &= \operatorname{Lef}(\sigma_* \circ F_{\W,\spincs})
\end{align*}
Since we are working over $\F$, the Lefschetz number is equal to the trace.

For convenience, let us write $\spinct_0 = \spinct_{\spincs_0}$,  $\bar\spinct_0 = \spinct_{\bar\spincs_0}$, $\spinct_0^\sigma = \spinct_{\spincs_0^\sigma}$, and $\bar \spinct_0^\sigma = \spinct_{\bar \spincs_0^\sigma}$. In case (1), the various compositions are given with respect to the basis $(a,b)$ by:
\begin{align*}
F_{\W,\spincs_0} &= \mat{1 & 0 \\ 0 & 0}  &
\sigma_* \circ F_{\W,\spincs_0} &= \mat{0 & 0 \\ 1 & 0}  \\
F_{\W,\bar\spincs_0} &= \mat{0 & 0 \\ 0 & 1} &
\sigma_* \circ F_{\W,\bar\spincs_0} &= \mat{0 & 1 \\ 0 & 0}
\end{align*}
By taking traces, we see that $\Phi_{\X, \spinct_0} = \Phi_{\X,\bar\spinct_0} = 1$, while $\Phi_{\X^\sigma, \spinct_0^\sigma} = \Phi_{\X^\sigma, \bar \spinct_0^\sigma} = 0$. Moreover, for any spin$^c$ structure on $\X$ or $\X^\sigma$ other than these, the Ozsv\'ath--Szab\'o invariant necessarily vanishes. Thus, $\X$ has two basic classes while $\X^\sigma$ has none. In case (2), we similarly compute that $\Phi_{\X, \spinct_0} = \Phi_{\X,\bar\spinct_0} = 0$, while $\Phi_{\X^\sigma, \spinct_0^\sigma} = \Phi_{\X^\sigma, \bar \spinct_0^\sigma} = 1$, so $\X^\sigma$ has two basic classes and $\X$ has none. In either case, we see that $\X$ and $\X^\sigma$ cannot be diffeomorphic.
\end{proof}

\begin{remark} \label{rem:Xsigma-Xtau}
Unlike $\sigma$, the symmetry $\tau$ acts by the identity on $\HFred(S^3_{-1}(Q_{-4}))$. Thus, the above argument also shows that $\X \not \cong \X^\rho$, but it cannot distinguish $\X$ from $\X^\tau$, or $\X^\sigma$ from $\X^{\rho}$.
\end{remark}

\subsection{Cosmetic surgeries and exotica}

We conclude this section with a discussion of exotic 4-manifolds that would result from certain open Dehn surgery questions.  We illustrate this with a recipe to build an exotic $S^1 \times S^3 \conn S^2 \times S^2$ as mentioned in the introduction.

\begin{proof}[Proof of Proposition~\ref{prop:cosmetic-exotic}]
Let $K$ be a (hypothetical) non-trivial knot in $S^3$ such that $S^3_{-1}(K) \cong S^3_1(K)$ as oriented manifolds. Since $K$ is non-trivial, we have that $\HFred(S^3_{\pm 1}(K)) \neq 0$, as it is well-known that no non-trivial knot has both positive and negative L-space surgeries.  Further, by \cite[Theorem 1.12]{OSzAbsolute}, the $d$-invariant of $S^3_{-1}(K)$ is at least 0 and the $d$-invariant of $S^3_1(K)$ is at most 0, so we see the $d$-invariant must be 0. Consider the cobordism $T_K$ from Definition \ref{def:Tye-cob}. By Proposition \ref{prop:tye-iso}, $F_{T_K} \co \HFred(S^3_{-1}(K)) \to \HFred(S^3_1(K))$ is an isomorphism. Let $X$ be the result of gluing the ends of $T_K$ together by some diffeomorphism.  By Proposition~\ref{prop:adam-iso}, $\adam(X) \ne 0$.  However, it is easy to see that $X$ has the same equivariant intersection form as $S^1 \times S^3 \conn S^2 \times S^2$, so the two are homeomorphic as in the proof of Proposition~\ref{prop:pi1Z}.  Hence, we have produced an exotic $S^1 \times S^3 \conn S^2 \times S^2$.
\end{proof}

More generally, if $K_1, \dots, K_n$ are nontrivial knots with the property that $S^3_1(K_i) \cong S^3_{-1}(K_{i+1})$ for each $i$ (indices modulo $n$), then a similar procedure yields an exotic $S^1 \times S^3 \conn n S^2 \times S^2$. We are not aware of any examples of this phenomenon.

\section{Simply connected constructions} \label{sec:simply-connected}

\subsection{Invariants for 4-manifolds with \texorpdfstring{$b^+=1$}{b\^+=1}}  \label{sec:closed-invariants}

We begin this section by reviewing the construction of the Ozsv\'ath--Szab\'o invariant for closed 4-manifolds with $b^+=1$, defined in \cite[Section 2.4]{OSzSymplectic}.

For notation, suppose $(W,\mft)$ is a spin$^c$ 4-manifold with connected boundary $(Y,\mfs)$, and assume that $\mfs$ is non-torsion. Recall from Section \ref{sec:background} that there is a canonical projection map $\Pi \co \HFm(Y,\spincs) \to \HFmr(Y,\spincs)$. The {\em relative invariant} of $(W,\mft)$ is defined as
\[
\Psi_{W,\mft} = \Pi (F^-_{W^\circ, \mft^\circ} (1)),
\]
where $W^\circ = W - B^4$, viewed as a cobordism from $S^3$ to $Y$, and $\mft^\circ = \mft|_{W^\circ}$.

There is a non-degenerate bilinear pairing
\[
\gen{\cdot, \cdot }_Y \co  \HFred(Y,\mfs) \otimes \HFred(-Y,\mfs) \to \F,
\]
coming from the duality between the Floer homologies of $Y$ and $-Y$ \cite[p.376]{OSz4Manifold}. In particular, note that if $\HFred(Y,\mfs) \cong \F$, then this pairing can only be one thing: if $x_{\pm}$ is the unique nonzero element of $\HFred( \pm Y, \mfs)$, then $\gen{x_+,x_-}=1$.

With this, we are ready to recall the Ozsv\'ath--Szab\'o invariant for closed 4-manifolds with $b^+ = 1$.  Let $X$ be a closed 4-manifold with indefinite intersection form, and let $L$ be a 1-dimensional subspace of $H_2(X;\Q)$ on which the intersection form of $X$ vanishes. A spin$^c$ structure $\spinct$ on $X$ is called \emph{allowable} if $c_1(\spinct)|_L \ne 0$. We call an embedded, oriented, separating $3$-manifold $Y \subset X$ an \emph{admissible cut} if the image of $H_2(Y;\Q) \to H_2(X;\Q)$ is equal to $L$. Write $X = W_1 \cup_Y W_2$, where $Y = \partial W_1  = -\partial W_2$ as oriented manifolds. Let $\spinct_1, \spinct_2, \spincs$ denote the restrictions of  $\mft$ to $W_1, W_2, Y$ respectively. If $\spinct$ is allowable, then $\spincs$ is non-torsion. We then define
\[
\Phi_{X,\mft,L} = \gen{ \Psi_{W_1,\mft_1}, \Psi_{W_2, \mft_2} }_Y.
\]
It follows from the duality between cobordism maps on $\HFm$ and $\HFp$ \cite[Theorem 3.5]{OSz4Manifold} that this definition agrees with the definition given in \cite{OSzSymplectic}; this proceeds along the lines of Jabuka--Mark \cite[Theorem 8.17]{JabukaMarkProduct}. By \cite[Proposition 2.7]{OSzSymplectic}, $\Phi_{X,\mft,L}$ is independent of the admissible (with respect to $L$) cut $Y$.

\begin{remark} \label{rem:OSz-invt}
Note that the construction of $\Phi_{X,\spinct,L}$ makes sense even when $b_2^+(L)>1$. In this case, the standard Ozsv\'ath--Szab\'o invariant $\Phi_{X,\spinct}$ (from \cite{OSz4Manifold}) is also defined for each spin$^c$ structure $\spinct$ on $X$, without reference to $L$. By applying \cite{OSzSymplectic}, we deduce that
\[
\Phi_{X,\spinct,L} = \sum_{\substack{\{ \spincu \in \Spin^c(X) : \\ \spincu|_{W_1} = \spinct_1, \ \spincu|_{W_2}=\spinct_2\}}} \Phi_{X,\spincu}.
\]
The terms in this sum correspond to the kernel of the map $H^2(X) \to H^2(W_1) \oplus H_2(W_2)$ in the Mayer--Vietoris sequence, or in other words the image of the coboundary map $H^1(Y) \to H^2(X)$. In particular, observe that if $\Phi_{X,\spinct,L} \ne 0$, then at least one term in the sum on the right must be nonzero.
\end{remark}

As a warm-up with these invariants, we prove a well-known vanishing result.  (Similar results for the Donaldson and Bauer-Furuta/Seiberg-Witten invariants can be found in \cite{KMEmbedded} and \cite{FroyshovMonopoles} respectively.)

\begin{lemma} \label{lemma:nosphere}
Let $X$ be a closed 4-manifold with $b^+ = 1$.  Suppose there exists a square-zero homologically essential embedded sphere or torus in $X$, and let $L \subset H_2(X; \Q)$ be the associated line.  Then, $\Phi_{X,\mft,L} \equiv 0$ for all allowable $\mft$.
\end{lemma}

\begin{proof}
The hypotheses imply that there exists an admissible cut along $S^2 \times S^1$ (in the sphere case) or $T^3$ (in the torus case). Each of these manifolds has vanishing $\HFred$ and the result then follows immediately.
\end{proof}

One difficulty in constructing manifolds for which $\Phi_{X,\mft,L} \ne 0$ is that in addition to showing that $\Psi_{W_1,\spinct_1}$ and $\Psi_{W_2,\spinct_2}$ are both nonzero, one must also prove that the pairing between these two elements is nontrivial; in general the pairing depends on how $\partial W_1$ and $\partial W_2$ are identified. However, there is one case where we can avoid this latter subtlety, namely if $\HFred(Y,\spincs) \cong \F$, so that the pairing is forced. We give a very specific instance of this principle, which is all we will need for our purposes.  If $K$ is a fibered knot of genus $g \ge 2$, then for the spin$^c$ structures $\spincs$ on $S^3_0(K)$ for which $c_1(\spincs)$ evaluates to $\pm(2g-2)$ on a generator of $H_2(S^3_0(K))$, we have $\HFred(S^3_0(K), \spincs) \cong \F$ \cite[Lemma 5.5]{OSzSymplectic}. Additionally, if $K$ is an amphichiral knot, then there is an orientation reversing-diffeomorphism of $S^3_0(K)$.  Consequently, given two 4-manifolds with (oriented) boundary $S^3_0(K)$, we can glue them together using a choice of orientation reversing-diffeomorphism to obtain a closed, oriented manifold on which we can understand the pairing $\Phi_{X,\mft,L}$ as follows.

\begin{lemma} \label{lemma:genus2}
Let $K$ be a genus $g$, fibered, amphichiral knot. Let $W$ be a spin$^c$ 4-manifold with $b_2 > 0$ and connected boundary $S^3_0(K)$. Suppose that $\mfu$ is a spin$^c$ structure on $W$ for which $c_1(\mfu)$ evaulates to $2g-2$ on a generator of $H_2(S^3_0(K))$, and for which $\Psi_{W,\spincu} \ne 0$. Let $X$ be the result of gluing together two copies of $W$ for any choice of orientation-reversing diffeomorphism of $S^3_0(K)$, and let $L \subset H_2(X;\Q)$ be the line corresponding to the inclusion $H_2(S^3_0(K)) \to H_2(X)$. Then there is a spin$^c$ structure $\spinct$ on $X$ for which $\Phi_{X, \spinct,L} \ne 0$.
\end{lemma}

\begin{proof}
Depending on how the diffeomorphism acts on $H_2$, we can use either $\mfu$ on both copies of $W$ or $\mfu$ and $\overline{\mfu}$ to get a glued up spin$^c$ structure $\spinct$ on $X$.  Since $\Psi_{W,\mfu} \ne 0$ if and only if $\Psi_{W,\overline{\mfu}} \ne 0$ \cite[Theorem 3.6]{OSz4Manifold}, we see that in either case, the relative invariants for the two copies of $W$ are non-zero and land in a one-dimensional space.  Therefore, their pairing is non-zero, showing that $\Phi_{X,\spinct,L} \ne 0$ as required.
\end{proof}

\subsection{Exotic simply connected manifolds from elementary cobordisms}\label{subsec:simplyconnected}

We now revisit Construction \ref{constr:simplyconn-intro} in the language of the building blocks from Section \ref{sec:building-blocks}, and prove that we produce an exotic $\cptwo \conn 9 \cptwobar$.

\begin{figure}[htbp]{\scriptsize
\begin{overpic}[tics=20]{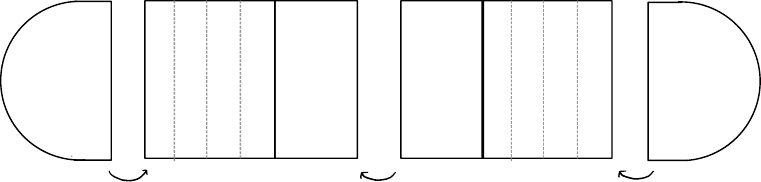}
    \put(29, 92){$S^3_{1}(P_{0})$}
        \put(66, 92){$S^3_{\text{-} 1}(Q_{\text{-} 4})$}
        \put(120, 92){$S^3_{\text{-} 1}(Q_{0})$}
                \put(155, 92){$S^3_{0}(Q_{0})$}
                                \put(187, 92){$S^3_{0}(Q_{0})$}
    \put(227, 92){$S^3_{1}(Q_{0})$}
        \put(278, 92){$S^3_{1}(Q_{4})$}
        \put(310, 92){$S^3_{-1}(-P_{0})$}
                \put(20, 45){$C_0$}
                        \put(330, 45){$C_0$}
\put(45, -10){$\phi_0$ or $\phi_0\circ\sigma$}
\put(290, -10){$\phi_0$ or $\phi_0\circ\sigma$}
\put(180, -10){$\zeta$}
     \put(76, 50){Four copies}
       \put(76, 42){of $B^{-}_{Q_{i}}$}
     \put(235, 50){Four copies}
       \put(235, 42){of $B^{-}_{Q_{i}}$}
     \put(140, 50){$T^-_{Q_0}$}
%       \put(135, 47){2-handle}
%              \put(135, 39){along $\mu_{Q_{0}}$}
     \put(200, 50){$T^+_{Q_0}$}
%       \put(195, 47){2-handle}
%              \put(195, 39){along $\mu_{Q_{0}}$}
  \end{overpic}}
  \caption{The exotic 4-manifolds $\E$ and $\E'$ from Construction \ref{constr:simplyconn} are built using the gluings $\phi_0\circ\sigma$ or $\phi_0$ respectively. The left half of this figure is $V$ (resp. $V'$). The right hand half is also $V$ (resp. $V'$), but it has been turned upside down.  }\label{fig:simplyconn}
  \end{figure}

\begin{construction}\label{constr:simplyconn}
Let $C_0 = C^+_{P_0}$ be the contractible $4$-manifold from Definition \ref{def:contract},  which has boundary $S^3_1(P_0)$. Let $\phi_0 \co S^3_1(P_0) \to S^3_{-1}(Q_{-4})$ be the homeomorphism from Proposition \ref{prop:PQdual}. Additionally, let $\sigma \co S^3_1(P_0) \to S^3_1(P_0)$ be the involution from Proposition \ref{prop:P0-symmetry}.

Define the manifold
\begin{equation}\label{eq:V}
V'=  C_0\cup_{\phi_0} B^{-}_{Q_{-4}} \cup B^{-}_{Q_{-3}} \cup B^{-}_{Q_{-3}} \cup B^{-}_{Q_{-1}} \cup T^-_{Q_0}
\end{equation}
built from the pieces from Section \ref{sec:building-blocks}, which is exhibited in the left half of Figure \ref{fig:simplyconn} (with the $\phi_0$ gluing).

Define $V$ to be the manifold obtained from $V'$ by cutting out $C_0$ and regluing it by $\phi_0\circ\sigma$. That is,
\begin{equation} \label{eq:V'}
V =  C_0\cup_{\phi_0 \circ \sigma} B^{-}_{Q_{-4}} \cup B^{-}_{Q_{-3}} \cup B^{-}_{Q_{-3}} \cup B^{-}_{Q_{-1}} \cup T^-_{Q_0}
\end{equation}
The manifold $V$ is also exhibited in the left half of Figure \ref{fig:simplyconn} (with the $\phi_0\circ\sigma$ gluing).  Because $\sigma$ extends topologically over $C_0$ (see Remark \ref{rem:cork}), $V$ and $V'$ are homeomorphic.

As seen in Proposition \ref{prop:square-invol}, $S^3_0(Q_0)$ admits an orientation-reversing, fixed-point-free involution $\zeta$.  Let $\E = V \cup_\zeta V$ and $\E' = V' \cup_\zeta V'$, see Figure \ref{fig:simplyconn}.  By construction, $\E$ and $\E'$ are homeomorphic.
\end{construction}

\begin{remark} \label{rem:V-standard}
Note that there are several equivalent formulations of $V$ and $V'$. For instance, since $\phi_0$ extends to a diffeomorphism $C^+_{P_0} \to C^-_{Q_{-4}}$, we may write them as
\begin{align*}
V' &= C^-_{Q_{-4}} \cup B^{-}_{Q_{-4}} \cup B^{-}_{Q_{-3}} \cup B^{-}_{Q_{-3}} \cup B^{-}_{Q_{-1}} \cup T^-_{Q_0} \\
V &= C^-_{Q_{-4}} \cup_{\phi_0 \sigma \phi_0^{-1}} B^{-}_{Q_{-4}} \cup B^{-}_{Q_{-3}} \cup B^{-}_{Q_{-3}} \cup B^{-}_{Q_{-1}} \cup T^-_{Q_0}.
\end{align*}
This agrees with the description in Construction \ref{constr:simplyconn-intro}. As an immediate consequence of the diffeomorphisms from \eqref{eq:boring-twist} and \eqref{eq:boring-tye}, we deduce that
\begin{equation} \label{eq:V-standard}
V' \cong X_0(Q_0) \conn 4 \cptwobar.
\end{equation}
%Alternately, by Remark \ref{rem:BJ}, we may also write\footnote{Not sure why we need this second version.}
%\begin{align*}
%V' &= C^+_{P_0} \cup B^+_{P_0} \cup B^+_{P_1} \cup B^+_{P_2} \cup B^+_{P_3} \cup_{\phi_4} T^-_{Q_0} \\
%V &= C^+_{P_0} \cup_\sigma B^+_{P_0} \cup B^+_{P_1} \cup B^+_{P_2} \cup B^+_{P_3} \cup_{\phi_4} T^-_{Q_0}.
%\end{align*}
\end{remark}

\begin{remark}
We may also turn $V'$ upside-down and view it as a cobordism from $-S^3_0(Q_0)$ to $\emptyset$, given by a series of 2-handle attachments followed by a copy of $C_0$ (also upside down). Specifically, we have
\begin{equation}\label{eq:V-upside-down}
V' =  T^+_{Q_0} \cup B^{+}_{Q_{0}} \cup B^{+}_{Q_{1}} \cup B^{+}_{Q_{2}} \cup B^{+}_{Q_{3}} \cup C.
\end{equation}
When we glue two copies of $V'$ to form $\E'$ using $\zeta$, the fact that $\zeta$ preserves the meridian in $S^3_0(Q_0)$ implies that the last $2$-handle of the first copy and the first $2$-handle of the second copy form the cobordism $T_{Q_0}$ (Definition \ref{def:Tye-cob}). Thus, we have decompositions $\E = V_a \cup T_{Q_0} \cup V_b$ and $\E' = V'_a \cup T_{Q_0} \cup V'_b$, where $V_a$ and $V_b$ (resp.~$V'_a$ and $V'_b$) are just copies of $V$ (resp $V'$) without the $T^-_{Q_0}$ component.
\end{remark}

We now focus on proving the first item of Theorem~\ref{thm:main}, which we restate here.
\begin{theorem} \label{thm:simplyconn}
\begin{enumerate}
\item The manifold $\E'$ is diffeomorphic to $\cptwo \conn 9 \cptwobar$.

\item The manifold $\E$ is homeomorphic but not diffeomorphic to $\cptwo \conn 9 \cptwobar$.
\end{enumerate}
\end{theorem}

\begin{remark}
In fact,  $V\cup_\zeta V'$ is diffeomorphic to $\cptwo \conn 9 \cptwobar$; since we do not require this we omit a proof.
\end{remark}

To prove Theorem \ref{thm:simplyconn}, we start with a more technical statement:

\begin{proposition} \label{prop:Q0-exotic-filling}
The manifold $V$ admits a spin$^c$ structure $\spinct$ whose restriction to $S^3_0(Q_0)$ is non-torsion, and for which $\Psi_{V, \spinct} \ne 0$. Thus, $V$ is homeomorphic but not diffeomorphic to $X_0(Q_0) \conn 4 \cptwobar$.
\end{proposition}

\begin{proof}
Choose a basis $(v_1, v_2, v_3)$ for $\HFm(S^3_1(P_0))$ as in Section \ref{subsec:PQ}, with respect to which $F^-_{C_0 - B^4}(1) = v_2$. Indeed, since $C_0$ is acyclic, $F^-_{C_0 - B^4}(1)$ is necessarily in degree 0, non-$U$-torsion (by \cite[p.247]{OSzAbsolute}), and $\iota$-invariant (by \eqref{eq:conjugation-cobordism}).  The only possible options for this element are therefore $v_2$ and $v_1 + v_2 + v_3$, which are related by an $\iota$-equivariant change of basis. By Proposition \ref{prop:P0-symmetry}, we then have $\sigma_* F^-_{C_0 - B^4}(1) = v_1 + v_2 + v_3$. Note also that $(v_1+v_2,v_2+v_3)$ is a good basis for $\HFred(S^3_1(P_0))$. Identify $S^3_1(P_0)$ with $S^3_{-1}(Q_{-4})$ via $\phi_0$, and (by abuse of notation) use $v_1,v_2,v_3$ to denote the corresponding elements of $\HFm(S^3_{-1}(Q_{-4}))$ under this identification.

Let $Z = V \setminus C_0$, viewed as a cobordism from $S^3_{-1}(Q_{-4})$ to $S^3_0(Q_0)$. By Lemma \ref{lemma:twist-spinc} (four times) and Lemma \ref{lemma:tye-spinc}(2), we may find generators $x_{\pm} \in \HFred(S^3_0(Q_0), \spincs_{\pm1})$, and a pair of conjugate spin$^c$ structures $\spincu$ and $\bar\spincu$ on $Z$ such that
\begin{align*}
F^-_{Z,\spincu}  (v_1+v_2) &= x_+ & F^-_{Z,\spincu}  (v_2+v_3) &= 0 \\
F^-_{Z,\bar\spincu} (v_1+v_2) &= 0 & F^-_{Z,\bar\spincu} (v_2+v_3) &= x_- .
\end{align*}
Let $\spinct$ (resp.~$\spinct'$) denote the unique extension of $\spincu$ to $V$ (resp.~$V'$). Then
\begin{align*}
  \Psi_{V',\spinct'} &= \Pi F^-_{Z,\spincu}(v_2) \\
  \Psi_{V,\spinct} &= \Pi F^-_{Z,\spincu}(v_1+v_2+v_3).
\end{align*}
Remark \ref{rem:V-standard} shows that $V'$ contains an essential $0$-framed sphere,  so we have $\Psi_{V',\spinct'} = 0$ by using the composition law and the fact that $\HFred(S^2 \times S^1) = 0$. Thus,
\[
  \Psi_{V,\spinct} = \Pi F^-_{Z,\spincu}(v_1+v_3)  = \Pi F^-_{Z,\spincu}((v_1+v_2)+(v_2+v_3)) = \Pi x_+ = x_+,
\]
as required.
\end{proof}

\begin{proof}[Proof of Theorem \ref{thm:simplyconn}]
By Remark \ref{rem:V-standard}, $\E'$ is diffeomorphic to $A \conn 8 \cptwobar$, where $A = X_0(Q_0) \cup_ \zeta X_0(Q_0)$. Since $\zeta$ comes from an orientation-reversing symmetry of $(S^3, Q_0)$, it extends to an orientation-reversing involution of $X_0(Q_0)$. Thus, $A$ is diffeomorphic to the double of $X_0(Q_0)$, which in turn is diffeomorphic to $S^2 \times S^2$.
%(see, e.g., \cite[Corollary 5.1.6]{GompfStipsicz}).
It follows that
\[
\E' \cong S^2 \times S^2 \conn 8 \cptwobar \cong \cptwo \conn 9 \cptwobar,
\]
proving statement (1).

As noted above, $\E$ and $\E'$ are homeomorphic; we now show they are not diffeomorphic.  From the decomposition $X = V_a \cup T_{Q_0} \cup V_b$ from above, we see that $H_2(\E)$ has a basis
\[
(E_1, \dots, E_8, \Xi, \Theta),
\]
where:
\begin{itemize}
  \item $(E_1, \dots, E_4)$ is a diagonal basis for $H_2(V_a)$;
  \item $(E_5, \dots, E_8)$ is a diagonal basis for $H_2(V_b)$;
  \item $(\Xi,\Theta)$ is the basis for $H_2(T_{Q_0})$ from Remark \ref{rem:TJ}. Specifically, $\Xi$ is represented by a capped-off Seifert surface in $S^3_0(Q_0)$, while $\Theta$ is represented by a sphere of self-intersection $-2$, and $\Xi \cdot \Theta = 1$.
\end{itemize}

By Lemma \ref{lemma:genus2} and Proposition \ref{prop:Q0-exotic-filling}, we deduce that there is a spin$^c$ structure $\spinct$ on $\E$ for which $\Phi_{\E,T,L} \ne 0$, where $L = \Span(\Xi) \subset H_2(\E;\Q)$. By Lemma \ref{lemma:nosphere}, we deduce that $\Xi$ cannot be represented by any embedded surface of genus less than $2$.

Suppose, toward a contradiction, that there is a diffeomorphism from $\E$ to $\cptwo \conn 9 \cptwobar$. By abuse of notation, we will identify $E_i, \Xi, \Theta$ with the corresponding classes in $H_2(\cptwo \conn 9 \cptwobar)$. Let $e_0, \dots, e_9$ denote the standard diagonal basis for $H_2(\cptwo \conn 9 \cptwobar)$, represented by disjoint embedded spheres, where $e_0^2 = 1$ and $e_i^2 = -1$ for $i=1, \dots, 9$. There is an isomorphism $\rho$ of $H_2(\cptwo \conn 9 \cptwobar)$, preserving the intersection form, given by:
\begin{align*}
\rho(\Xi) &= e_0 - e_9  & \rho(E_i) &= e_i \quad (\text{for } i=1, \dots, 7) \\
\rho(\Theta) &= e_8 + e_9  & \rho(E_8) &= e_0 + e_8 - e_9
\end{align*}
Wall \cite[p.~137]{WallDiffeomorphisms} showed that every isomorphism of the intersection form of $\cptwo \conn 9\cptwobar$ can be realized by a diffeomorphism; let $r$ be a diffeomorphism realizing $\rho$. However, $e_0-e_9$ can be represented by an embedded sphere $S$, and hence $r^{-1}(S)$ represents $\Xi$, which contradicts our result above. Thus, $\E$ cannot be diffeomorphic to $\cptwo \conn 9 \cptwobar$.
\end{proof}

\begin{proof}[Proof of Theorem \ref{thm:unlabeledthmintro}]
Let $\W \co S^3_{-1}(Q_{-4}) \to S^3_{-1}(Q_{-4})$ be the cobordism from Construction \ref{constr:pi1Z}. For each $n \ge 1$, let $\E_n$ (resp.~$\E_n'$) be obtained by inserting $n$ copies of $\W$ in between $C$ and $Z$ in the first copy of $V$ (resp.~$V'$). Each of these is homeomorphic to $\conn (2n+1)\cptwo \conn 9(n+1) \cptwobar$. Just as in the proof of Theorem \ref{thm:simplyconn}, we deduce that there is a square-0 line $L \subset H_2(\E_n)$ and a spin$^c$ structure $\spinct$ on $\E_n$ with $\Phi_{\E_n, \spinct, L} \ne 0$. As noted in Remark \ref{rem:OSz-invt}, this implies that the ordinary Ozsv\'ath--Szab\'o invariant of $\E_n$ is nonvanishing. Since the invariant of $(2n+1)\cptwo \conn 9(n+1) \cptwobar$ vanishes by \cite[Theorem 1.3]{OSz4Manifold}, we see that $\E_n$ is an exotic $(2n+1)\cptwo \conn 9(n+1) \cptwobar$, as required.
\end{proof}

\subsection{Exotic closed definite manifolds with \texorpdfstring{$\pi_1\cong \Z_2$}{{\textpi}\_1=Z/2}}\label{subsec:definite}

We begin by building our candidate exotic pair. Note that the manifolds $\E$ and $\E'$ constructed in the previous section each admit an orientation-preserving involution that interchanges the two copies of $V$ (resp.~$V'$), while acting on the common boundary by the homeomorphism $\zeta$ from Proposition \ref{prop:square-invol}.  Because $\zeta$ has no fixed points, the quotients $\RR$ and $\RR'$ are each oriented manifolds with fundamental group $\Z/2$. An equivalent construction is that $\RR$ is obtained as the quotient $V/{\sim}$, where each point of $S^3_0(Q_0)$ is identified with its translate under $\zeta$.

We will also want a `standard' rational homology sphere $Z_0$ with $\pi_1=\Z/2$ to (blow up and) compare with our $\RR$ and $\RR'$.  Towards this, let us regard $S^2 \times D^2$ as the 0-trace of the unknot, and observe that one can build $S^2\times S^2$ by gluing two copies of $S^2 \times D^2$ using a free, orientation-reversing involution $\kappa$ of $S^2\times \partial D^2$.  In a similar spirit to the definition of $\RR$ and $\RR'$, define $Z_0$ to be the quotient of $S^2\times S^2$ by the involution which exchanges the copies.  Then $\pi_1(Z_0) \cong \Z/2$ by construction, and one readily checks that $Z_0$ is a rational homology 4-sphere. An equivalent construction is that $Z_0$ is obtained as the quotient $S^2 \times D^2/{\sim}$, where each point of $S^2 \times \partial D^2$ is identified with its image under $\kappa$.
\begin{theorem}
\begin{enumerate}
\item The manifold $\RR'$ is diffeomorphic to $Z_0 \conn 4 \cptwobar$.
%where $Z_0$ is a rational homology sphere with $\pi_1(Z_0) = \Z/2$.

\item The manifold $\RR$ is homeomorphic but not diffeomorphic to $\RR'$. Thus, $\RR$ is an exotic $4$-manifold with $\pi_1=\Z/2$ and negative-definite intersection form.
\end{enumerate}
\end{theorem}

Note that this proves Theorem \ref{thm:main} item $(2)$.

\begin{proof}
For any strongly negatively amphichiral knot $K$, let $Z_K$ denote the quotient $X_0(K)/{\sim}$, where we quotient out by the fixed-point free involution on $S^3_0(K)$ coming from the strong negative amphichiral structure. It is shown in \cite{LevineSNACK} that the diffeomorphism type of $Z_K$ is independent of $K$; in particular $Z_{Q_0}$ is diffeomorphic to $Z_0$.

Since $V'$ is diffeomorphic to $X_0(Q_0) \conn 4 \cptwobar$, and the equivalence relation used in constructing $\RR'$ from $V$ is the same as that used in constructing $Z_{Q_0}$ from $X_0(Q_0)$, it follows that $\RR'$ is diffeomorphic to $Z_0 \conn 4 \cptwobar$. Since $\RR$ is obtained from $\RR'$ by cutting out and regluing $C^-_{Q_{-4}}$, $\RR$ is homeomorphic to $\RR'$ by Remark \ref{rem:cork}. On the other hand, since the universal covers of $\RR$ and $\RR'$ are $\E$ and $\E'$, respectively, which are not diffeomorphic, it follows that $\RR$ and $\RR'$ are not diffeomorphic.
\end{proof}

\subsection{Knot traces and spineless 4-manifolds} \label{subsec:spineless}

In this section, we prove the existence of the spineless 4-manifolds claimed in Theorem~\ref{thm:spineless}.  

For a knot $K \subset S^3$, let $X_0(K)$ denote the $0$-trace of $K$, let $\spinct_s$ denote the spin$^c$ structure on $X_0(K)$ whose first Chern class evaluates to $2s$ on a capped-off Seifert surface for $K$, and let $\spincs_s = \spinct_s|_{S^3_0(K)}$. Also recall the concordance invariants  $(V_s(K))_{s \in \Z}$ (originally denoted $h_s$ in \cite{RasmussenThesis}), which are a sequence of natural  numbers satisfying
\begin{equation}\label{eq:V-decrease}
V_s(K) \ge V_{s+1}(K) \ge V_s(K) -1.
\end{equation}
By \cite[Lemma 2.5]{HomLidmanZufelt}, for $s>0$, we have
\[
V_{-s}(K) = V_s(K)+s.
\]

The following lemma appears in a retracted preprint of Rasmussen from 2010 \cite{RasmussenMorse}, but we provide a proof here using modern terminology.

\begin{lemma} \label{lemma:0trace-relative}
For any knot $K$ in $S^3$ and any $s \ne 0$, $\Psi_{X_0(K), \spinct_s} \ne 0$ if and only if $V_{\abs{s}}(K) \ne 0$.
\end{lemma}

\begin{proof}
By conjugation symmetry, we may assume that $s > 0$. Let
\[
f_s = \mathbf F^-_{X_0(K)^\circ, \spinct_s} \co \HFmc(S^3) \to \HFmc(S^3_0(K), \spincs_s).
\]
Identifying $\HFmc(S^3) = \F[[U]]$, we have $\Psi_{X_0(K), \spinct_s} = f_s(1)$. By the mapping cone formula \cite{OSzSurgery, ManolescuOzsvathLink}, $f_s$ fits into an exact triangle
\[
H_*(\mathbf A_s^-) \xrightarrow{v_s^- + h_s^-}  \HFmc(S^3) \xrightarrow{f_s} \HFmc(S^3_0(K)) \to H_*(\mathbf A_s^-).
\]
Here $A_s^-$ is the $U$-completion of a certain subcomplex of $\CFKm(S^3,K)$; its homology computes $\HFmc$ of large surgery on $K$ in a particular spin$^c$ structure. In particular, $H_*(\mathbf A_s^-)/(U\text{-torsion}) \cong \F[[U]]$. The maps $v_s^-$ and $h_s^-$ necessarily vanish on the $U$-torsion subgroup,  and the induced maps on $\F[[U]]$ are multiplication by $U^{V_s(K)}$ and $U^{V_{-s}(K)}$. Thus, the sequence
\[
\F[[U]] \xrightarrow{U^{V_s(K)}(1+U^s)}  \F[[U]] \xrightarrow{f_s} \HFmc(S^3_0(K)),
\]
is exact, and we deduce that $f_s(1) = 0$ if and only if $V_s(K)=0$, as required.
\end{proof}

\begin{lemma}\label{lemma:spineless-map}
Suppose that $N$ is a 4-manifold homotopy equivalent to $S^2$ with a spine and whose boundary is homology cobordant to $S^2 \times S^1$.  Then for every spin$^c$ structure $\mft$ on $N$ with $c_1(\mft) \neq 0$, we have $\Psi_{N,\spinct}=0$.
\end{lemma}

\begin{proof}
Since $N$ has a spine, there exists a knot $K$ such that the 0-trace, $X_0(K)$, embeds in $N$.  Furthermore, $N - X_0(K)$ provides a homology cobordism from $S^3_0(K)$ to the boundary of $N$, which in turn, is homology cobordant to $S^2 \times S^1$.  Note that the homology cobordism to $S^2 \times S^1$ implies that $d_{\pm}(S^3_0(K)) = \pm 1/2$.  Therefore, the knot concordance invariants $V_0(K)$ and $V_0(\overline K)$ vanish by \cite[Proposition 4.12]{OSzAbsolute} and \cite[Proposition 1.6]{NiWu}.  By \eqref{eq:V-decrease}, $V_s(K) = 0$ for all $s \ge 0$. By Lemma \ref{lemma:0trace-relative}, $\Psi_{X_0(K), \spincu} = 0$ for every non-torsion $\spincu \in \Spin^c(X_0(K))$. It then follows that for each nontorsion $\spinct \in \Spin^c(N)$, we have
\[
\Psi_{N,\spinct} = F_{N -X_0(K), \spinct|_{N-X_0(K)}} \left( \Psi_{X_0(K), \spinct|_{X_0(K)}} \right) = 0,
\]
as required.
\end{proof}

With this we can easily construct the promised spineless 4-manifolds.
\begin{proof}[Proof of Theorem~\ref{thm:spineless}]
Let $N$ denote the four-manifold with boundary $S^3_0(Q_{-4})$ obtained by attaching a 2-handle to $C_0$ along $\sigma(\mu_{Q_{-4}})$. Cutting $C_0$ out of $N$ and regluing using $\sigma$ does not change the homeomorphism type and results in $X_0(Q_{-4})$ by \eqref{eq:boring-tye}. Note that $N$ is a codimension 0 submanifold of $V$ which carries the singular part of the intersection form; this is easily seen from in the description in Construction \ref{constr:simplyconn-intro}. It follows from Proposition~\ref{prop:Q0-exotic-filling} that there is a non-torsion$^c$ structure $\mft$ on $N$ with non-trivial $c_1$ satisfying $\Psi_{N, \mft} \neq 0$. Thus, by Lemma \ref{lemma:spineless-map}, $N$ is spineless.
\end{proof}

\begin{remark} \label{rem:S2xS2}
We conclude this section with a speculative strategy for constructing smaller exotic closed 4-manifolds. Suppose $K$ and $J$ are knots for which there is an orientation-preserving homeomorphism $\xi \co S^3_0(K) \to S^3_0(J)$. Then the manifold $Z = X_0(K) \cup_\xi \overline{X_0(J)}$ is homeomorphic to either $S^2 \times S^2$ or $\cptwo \conn \cptwobar$, depending on how $\xi$ acts on spin structures. Let $L$ be the image of the inclusion $H_2(S^3_0(K)) \to H_2(Z)$. If one can show that there is a non-torsion spin$^c$ structure $\spincu$ for which $\Phi_{Z,\spincu,L} \ne 0$, this would immediately show that $Z$ is exotic.

Suppose $\spincu$ is such that $\Phi_{Z,\spincu,L} \ne 0$ and $\spincu|_{S^3_0(K)} = \spincs_s$. By conjugation invariance, we may assume $s>0$. Then $\Psi_{X_0(K), \spinct_s} \ne 0$ and $\Psi_{\overline{X_0(J)}, \spinct_s'} \ne 0$ (where we use $\spinct_s'$ to denote the appropriate spin$^c$ structure on $\overline{X_0(J)}$). By Lemma \ref{lemma:0trace-relative}, we must have
\begin{equation} \label{eq:Vs-KJ}
V_s(K) >0 \quad \text{and} \quad V_s(\overline J) > 0.
\end{equation}

Thus, one possible strategy for finding an exotic $S^2 \times S^2$ or $\cptwo \conn \cptwobar$ is to find pairs of knots $K,J$ for which $S^3_0(K) \cong S^3_0(J)$ and for which \eqref{eq:Vs-KJ} holds for some $s>0$. Of course, \eqref{eq:Vs-KJ} is not sufficient; one must also show that
\begin{equation} \label{eq:pairing}
\gen{\Psi_{X_0(K), \spinct_s}, \ \xi_*^{-1} \Psi_{\overline{X_0(J)}, \spinct_s'} } \ne 0.
\end{equation}
The optimal situation would be that $\dim \HFred(S^3_0(K), \spincs_s) = 1$, in which case \eqref{eq:pairing} would be automatic. Alternately, if the mapping class group of $S^3_0(K)$ acts sufficiently nontrivially on $\HFred(S^3_0(K), \spincs_s)$ --- for instance, if the span of the translates of $\Psi_{X_0(K),\spinct_s}$ is equal to all of $\HFred(S^3_0(K), \spincs_s)$ --- then there will be some gluing $\xi$ for which \eqref{eq:pairing} holds. Indeed, if $K$ is a knot for which $V_s(K) \ne 0$ and $V_s (\overline K) \ne 0$, then one may take $J=K$; however, in this case, the gluing map $\xi$ would necessarily have to be nontrivial, as the identity gluing (or any gluing that extends over $X_0(K)$) would produce the double of $X_0(K)$, which is diffeomorphic to $S^2 \times S^2$.

In \cite{RasmussenMorse}, it was claimed that this strategy cannot succeed: namely, that there cannot be knots $K,J$ satisfying the above conditions. However, this argument relied on the incorrect assertion that $V_0(K)\ne0$ implies $\tau(K)>0$. (For a counterexample to this implication, see Bodn\'ar--Celoria--Golla \cite[Proposition 2.4]{BodnarCeloriaGolla}.) However, the converse is true; indeed, by \cite[Proposition 2.3]{HomWu}, for any knot $K$ with $\tau(K) > 0$, we have $V_{\tau(K)-1} \ne 0$. Thus, if $K$ and $J$ are knots with $S^3_0(K) \cong S^3_0(J)$, $\tau(K) \ge 2$, and $\tau(J) \le -2$, then \eqref{eq:Vs-KJ} is automatically satisfied for $s=1$. Although there are many examples of pairs of knots with homeomorphic $0$-surgeries and different values of $\tau$ (see, e.g., \cite{YasuiConcordance, HMP}), we are not aware of any pairs in which the values of $\tau$ have opposite signs.
\end{remark}

\section{Exotic surfaces from symmetries} \label{sec:surfaces}

In this section, we show that the manifolds $\E$ and $\E'$ from Construction \ref{constr:simplyconn} can be realized as the branched double covers of embedded, non-orientable surfaces $\Sigma$ and $\Sigma'$ in $S^4$, as stated in Theorem \ref{thm:surfaces}. Each of these surfaces is homeomorphic to $\conn 10 \R P^2$, with normal Euler number 16. Using recent work of Conway--Orson--Powell \cite{ConwayOrsonPowell}, we will show that $\Sigma$ and $\Sigma'$ are topologically isotopic. (An earlier theorem of Kreck \cite{Kreck} also applies in this specific case.) The key technical observation is that the complements of $\Sigma$ and $\Sigma'$ each have fundamental group isomorphic to $\Z/2$. On the other hand, we see immediately that $\Sigma$ and $\Sigma'$ cannot be smoothly equivalent, because their branched double covers are not diffeomorphic.

\hspace{-20pt}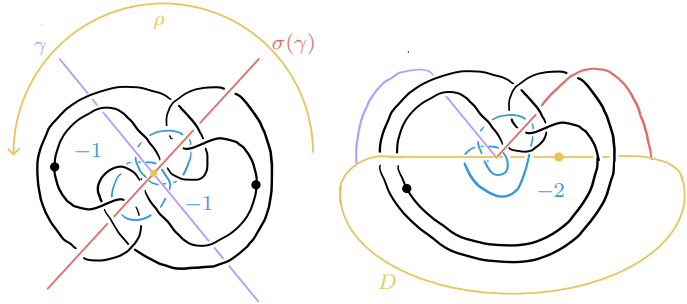
\begin{figure}[htbp]{\scriptsize
\begin{overpic}[tics=20]{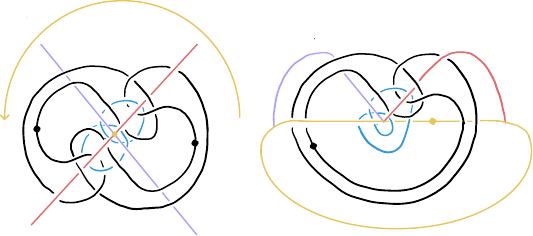}
\put(25,55){$\color{lblue}-1$}
\put(67,35){$\color{lblue}-1$}
\put(55,105){$\color{lyellow}\rho$}
\put(10,95){$\color{lpurple}\gamma$}
\put(100,96){$\color{lred}\sigma(\gamma)$}
% final frame
\put(200,40){$\color{lblue}-2$}
\put(140,5){$\color{lyellow}D$}
%\put(225,85){$\color{lpurple}\gamma$}
%\put(270,80){$\color{lred}\tau(\gamma)$}
  \end{overpic}}
  \caption{The symmetry $\rho$ extends over $C^+_{P_0}$; taking the quotient we obtain the branch disk $D\hookrightarrow B^4$ marked in yellow. }
  \label{fig:branch}
\end{figure}

\begin{proof}[Proof of Theorem \ref{thm:surfaces}]
Recall the involution $\rho \co S^3_1(P_0) \to S^3_1(P_0)$ from Section \ref{subsec:symmetry}, which extends to an involution of $C^+_{P_0}$. By Remark \ref{rem:rho}, the attaching curves for the $2$-handles of $\E$ can be simultaneously isotoped to be symmetric with respect to $\rho$. It follows immediately that $\rho$ extends over $\E$. Likewise, because $\rho$ commutes with $\sigma$, the attaching curves for the $2$-handles of $\E'$ can likewise be made symmetric, and $\rho$ extends over $\E'$ as well. By abuse of notation, we denote these extensions by $\rho$ and $\rho'$, respectively.

Define $\Sigma$ (resp.~$\Sigma'$) to be the branch surface downstairs after taking the quotient of the involution $\rho$ (resp.~$\rho'$) on $\E$ (resp.~$\E'$).  In the right-hand frame of Figure \ref{fig:branch} we have demonstrated the branch disk $D$ in $B^4$ of the involution restricted to $C_0$.  Using the Montesinos trick, we know that $\Sigma$ and $\Sigma'$ can be built on top of $D$ by attaching bands according to the 2-handles used to construct $\E$ and $\E'$ (and eventually capping off with a second copy of $D$). This certifies that the quotient manifolds $\E/\rho$ and $\E'/\rho'$ are each build from a pair of $4$-balls, hence diffeomorphic to $S^4$.  The ten 2-handles yield 10 band attachments in the construction of these surfaces, so $\Sigma$ and $\Sigma'$ each have Euler characteristic $-8$.  Since the $B$-cobordisms go between homology spheres, the corresponding bands downstairs go between knots (and not links), and hence make $\Sigma$ and $\Sigma'$ non-orientable.  It follows that each is abstractly homeomorphic to $\# {10}\R P^2$.  Since both surfaces have double branched covers with signature $-8$,  we know that the normal Euler numbers of $\Sigma$ and $\Sigma'$ are both 16 (see, e.g. the final page of \cite{Rokhlin}).

In order to compare the topological isotopy type of $\Sigma$ and $\Sigma'$, we need to compute $\pi_1$ of the complements.  We begin with $\Sigma'$. Consider the branch surface $\Gamma'$ in $B^4$ of the involution $\rho'$ restricted to $C_0 \cup B^-_{Q_{-4}}$. This surface is obtained from $D$ by attaching a (nonorientable) band along the curve marked in purple in the right frame of Figure \ref{fig:branch}. (The band has $-1$ half twist, but since we won't need to know the twisting parameter we omit the proof.)  This is a subsurface of $\Sigma'$, and we can check (using the rising water principle) that $\Gamma'$ has fundamental group $\Z/2$. Since the rest of $\Sigma'$ can be built on top of $\Gamma'$ with no new 0-handles, we know $\pi_1(S^4\setminus \Sigma')$ is either $1$ or $\Z/2$.  Since $\Sigma$ is a nonorientable surface embedded in $S^4$, it cannot have simply connected complement.

For $\Sigma$ we argue similarly; define $\Gamma$ in $B^4$ to be the branch surface of the quotient of $\rho$ restricted to $C_0 \cup_\sigma B^-_{Q_{-4}}$; this surface is obtained from $D$ by attaching a band with a negative half twist along the curve marked in red in the right frame of Figure \ref{fig:branch}.  Again we compute that $\pi_1(B^4\setminus \Gamma)\cong \Z/2$, and as before conclude $\pi_1(S^4\setminus \Sigma')$ is $\Z/2$.

We can then conclude that $\Sigma$ and $\Sigma'$ are topologically isotopic by appealing to \cite{Kreck} (see also \cite{ConwayOrsonPowell}). However, $\Sigma$ and $\Sigma'$ cannot be smoothly equivalent, as this would induce a diffeomorphism between their double branched covers, $\E$ and $\E'$.
\end{proof}

\begin{remark}
In fact, $\E$ and $\E'$ admit other $\Z/2$ symmetries as well; in brief, these are obtained by permuting the $2$-handles of $V$ or $V'$. Thus, $\E$ and $\E'$ can be obtained as branched double covers of smaller surfaces in larger $4$-manifolds, such as tori in $\conn 4 \cptwobar$. However, because the fundamental groups of these surface complements may be complicated, we cannot use currently available methods to deduce that these pairs of surfaces are topologically isotopic.
\end{remark}

\section{Handlebody presentations} \label{sec:handles}

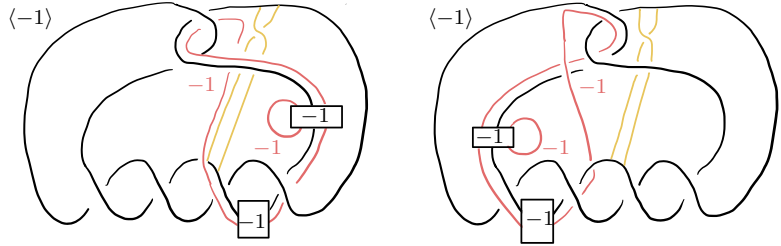
\begin{figure}[htbp]{\scriptsize
\begin{overpic}[tics=20]{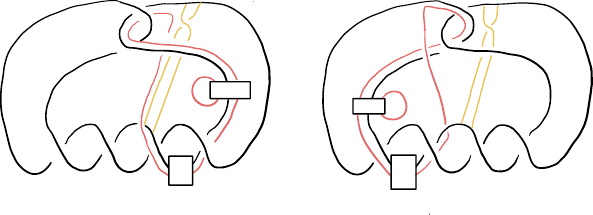}
\put(-6,95){$\langle -1 \rangle$}
\put(153,95){$\langle -1 \rangle$}
\put(105,58){$-1$}
\put(171,50){$-1$}
\put(81,18){$-1$}
\put(190,19){$-1$}
\put(87,46){$\color{lred}-1$}
\put(196,46){$\color{lred}-1$}
\put(62,70){$\color{lred}-1$}
\put(210,69){$\color{lred}-1$}
%\put(237,4){$\color{lgreen}0$}
%\put(259,-5){\color{lblue}$\mp 1$}
%\put(150,48){\color{lblue}$-1$}
%\put(36,47){\color{lblue}$-1$}
%\put(36,55){\color{lblue}$-1$}
  \end{overpic}}
  \caption{Left: an (upside down) handle decomposition of the contractible $C^+_{P_0}$. The slice disk $D_{P_0}$ is indicated by the yellow band. Two 3-handles and a 4-handle are not pictured. Right: the 2-handles of $C^+_{D_{P_0}}$ have been exchanged using the homeomorphism~$\sigma$.}
  \label{fig:fullHD1}
\end{figure}
We now give an explicit handlebody presentation for $\E$, the exotic $\cptwo \conn 9 \cptwobar$ constructed above.  We will work through Construction \ref{constr:simplyconn} from the top down, drawing the handle decomposition as we go.

We begin with an upside down copy of $C^+_{P_0}$, which has bottom boundary $S^3_{-1}(\overline{P_0})$.  This is exhibited in the left frame of Figure \ref{fig:fullHD1}. Note that this decomposition of $C^+_{P_0}$ has two 3-handles and a 4-handle which are not pictured; we will suppress the 3- and 4-handles from the diagrams throughout the section.  We then attach $C^+_{P_0}$ to $S^3_{-1}(\overline{P_0})\times I$ using the involution $\sigma$ from Proposition \ref{prop:P0-symmetry}; this is exhibited in the second frame of Figure \ref{fig:fullHD1}.

\begin{figure}[htbp]{\scriptsize
\begin{overpic}[abs, tics=20]{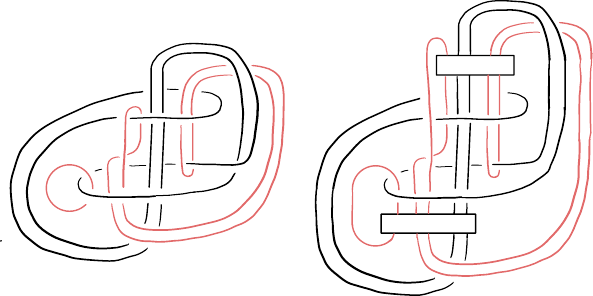}
\put(4,90){$\langle -1 \rangle$}
\put(160,90){$\langle 1 \rangle$}
\put(120,29){$\color{lred}0$}
\put(278,20){$\color{lred}4$}
\put(37,34){$\color{lred}-1$}
\put(187,19){$\color{lred}1$}
\put(226,109){$2$}
\put(203,32){$2$}
%\put(237,4){$\color{lgreen}0$}
%\put(259,-5){\color{lblue}$\mp 1$}
%\put(150,48){\color{lblue}$-1$}
%\put(36,47){\color{lblue}$-1$}
%\put(36,55){\color{lblue}$-1$}
  \end{overpic}}
  \caption{}
  \label{fig:fullHD2}
\end{figure}

We would then like to use the homeomorphism $(\psi_4)^{-1}$ from Proposition \ref{prop:PQdual} (which was a special case of Proposition \ref{prop:fusion-duals}) to glue this to the top of the following stack of elementary cobordisms
\begin{equation}\label{eq:stack}
 B^{-}_{Q_{-4}} \cup B^{-}_{Q_{-3}} \cup B^{-}_{Q_{-3}} \cup B^{-}_{Q_{-1}} \cup T_{Q_0} \cup B^{+}_{Q_{0}} \cup B^{+}_{Q_{1}} \cup B^{+}_{Q_{2}} \cup B^{+}_{Q_{3}}
\end{equation}
which has top boundary the three manifold $S^3_{1}(Q_4)$.  In order to visualize how $(\psi_4)^{-1}$ will act on the 2-handles in the right frame of Figure \ref{fig:fullHD1}, it is convenient to isotope the diagram in the right frame of Figure \ref{fig:fullHD1} until the black curve looks ribbon (with respect to the ribbon band marked in yellow); this results in the left frame of Figure \ref{fig:fullHD2}.  The right frame of Figure \ref{fig:fullHD2} is then obtained from the left frame by applying the homeomorphism $(\psi_{4})^{-1}$ to the red 2-handles. (See Figure \ref{fig:refcalc}.)

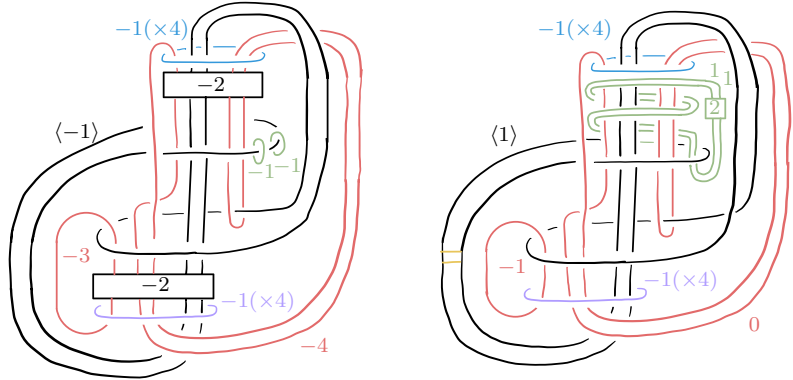
\begin{figure}[htbp]{\scriptsize
\begin{overpic}[abs, tics=20]{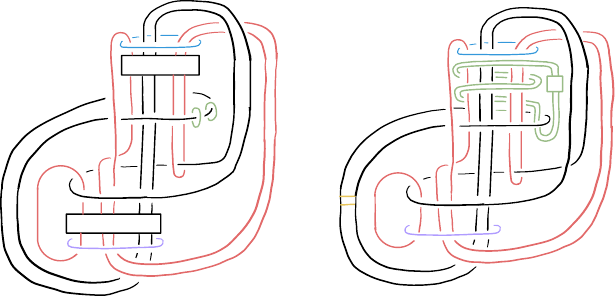}
\put(71,109){$-2$}
\put(50,33){$-2$}
\put(17,91){$\langle -1 \rangle$}
\put(90,76){$\color{lgreen}-1$}
\put(100,78){$\color{lgreen}-1$}
\put(20,44){$\color{lred}-3$}
\put(110,10){$\color{lred}-4$}
\put(40,130){\color{lblue}$-1 (\times 4)$}
\put(80,27){\color{lpurple}$-1 (\times 4)$}
%%%second frame
\put(182,90){$\langle 1 \rangle$}
\put(270,111){$\color{lgreen}1$}
\put(265,114){$\color{lgreen}1$}
\put(265,100){$\color{lgreen}2$}
\put(185,40){$\color{lred}-1$}
\put(280,18){$\color{lred}0$}
\put(200,130){\color{lblue}$-1 (\times 4)$}
\put(240,35){\color{lpurple}$-1 (\times 4)$}
  \end{overpic}}
  \caption{Here the notation $-1(\times 4)$ means that there are four $-1$ framed 2-handles attached along the $(0,4)$ cable of the corresponding curve.}
  \label{fig:fullHD3}
\end{figure}

We are then ready to stack this on top of the stack of elementary cobordisms in \eqref{eq:stack}; this is done in the left frame of Figure \ref{fig:fullHD3}.

The construction finishes by attaching this to $C^+_{P_0}$ using the gluing homeomorphism $\sigma \circ \phi_0^{-1}$.  We first attach the manifold in the left frame of Figure \ref{fig:fullHD3} to $S^3_1(P_0)\times I$ using $\phi_0$, as exhibited in the right frame of Figure \ref{fig:fullHD3}.  In order to see how $\sigma$ acts on the 2-handles, we need to isotope the black curve to look like the standard picture of the stevedore's knot; this has been done (with a few simplifying handleslides) to obtain the left frame of Figure \ref{fig:fullHD4}. We then complete the construction, and the handle diagram, by gluing this to the top boundary of $C^+_{P_0}$ using $\sigma$, which is shown in the right frame of Figure \ref{fig:fullHD4}.

\begin{figure}[htbp]{\scriptsize
\begin{overpic}[abs, tics=20]{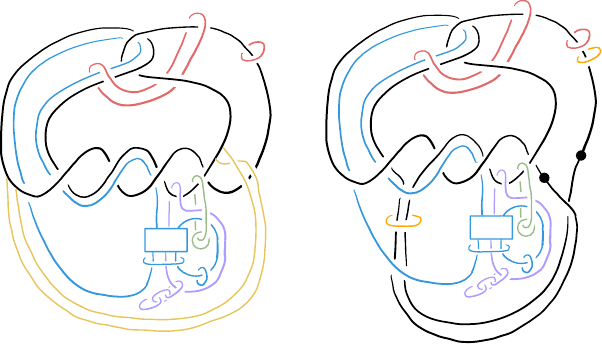}
\put(5,142){$\langle 1 \rangle$}
\put(102,64){$\color{lgreen}-1$}
\put(90,38){$\color{lgreen}-2$}
\put(102,155){$\color{lred}0$}
\put(130,140){$\color{lred}-1$}
\put(75,48){\color{lblue}$-1$}
\put(55,38){\color{lblue}$-2$}
\put(108,35){\color{lblue}$-2$}
\put(109,43){\color{lblue}$-2$}
\put(19,57){\color{lblue}$-1$}
\put(110,50){\color{lpurple}$-1$}
\put(78,16){\color{lpurple}$-2\times 3$}
%%%second frame
\put(255,70){$\color{lgreen}-1$}
\put(248,44){$\color{lgreen}-2$}
\put(250,138){$\color{lred}0$}
\put(275,154){$\color{lred}-1$}
\put(170,70){\color{lblue}$-1$}
\put(230,53){\color{lblue}$-1$}
\put(212,42){\color{lblue}$-2$}
\put(278,32){\color{lblue}$-2$}
\put(279,40){\color{lblue}$-2$}
\put(250,25){\color{lpurple}$1$}
\put(215,15){\color{lpurple}$-2\times 3$}
\put(286,143){\color{lorange}$-1$}
\put(206,57){\color{lorange}$0$}
  \end{overpic}}
  \caption{Here the notation $-2\times 3$ means that the three adjacent purple 2-handles are all $-2$ framed. The right frame is an exotic $\cptwo\conn 9\cptwobar$.}
  \label{fig:fullHD4}
\end{figure}

\begin{remark}
It is not difficult to see that the decomposition in the right frame of Figure \ref{fig:fullHD4} can be further simplified to (simultaneously) cancel both 1-handles and both 3-handles. It seems somewhat difficult to actually perform this simplification.
\end{remark}

\section{Bordered computations }\label{sec:bordered}

In this section, we complete the proof of Proposition \ref{prop:HF-Qn}, which we restate as follows:

\begin{theorem} \label{thm:knot-surgery}
For any knot $K \subset S^3$ and any $n \in \Z$, we have
\[
\dim \HFh(S^3_{\pm1} (P_n^K)) = \dim \HFh(S^3_{\pm1}(Q_n^K))  = 4 \dim \HFKh(K) + 1.
\]
\end{theorem}

By Proposition \ref{prop:twist-iso}, the isomorphism types of $\HFm(S^3_{\pm1}(P_n^K))$ and $\HFm(S^3_{\pm1}(Q_n^K))$ are independent of $n$, and hence the same is true for $\HFh$. Furthermore, by the symmetries from Proposition \ref{prop:PQdual}, the isomorphism type is also independent of the choice of sign and of $P$ or $Q$. Thus, it suffices to compute $\HFh(S^3_1(P_{-1}^K))$ for all $K$. (The reason for this specific choice will become apparent.)

Let $Y = S^3_1(P_{-1})$ and $Y^K = S^3_1(P_{-1}^K)$. As in Definition \ref{def:fusion1}, let $\gamma \subset Y$ denote the meridian of the band of $P_{-1}$. Let $M = Y \minus \nbd(\gamma)$. Let $\lambda$ and $\mu$ be the Seifert longitude and meridian of $\gamma$, respectively, viewed as curves in $\partial M$. (Because $\lk(P_{-1},\gamma)=0$, $\lambda$ is the same longitude as the one induced from the Seifert longitude of $\gamma$ in $S^3$.) By construction, $Y^K$ is obtained by gluing $M$ to $S^3 \minus \nbd(K)$, where $\lambda$ is identified with $\mu_K$ and $\mu$ is identified with $\lambda_K$.

We begin by computing the bordered invariant of $M$, using Hanselman--Rasmussen--Watson's immersed curves reformulation of bordered Floer homology \cite{HRWImmersed}. Following the notation of \cite{HRWImmersed}, let $T_M$ be the complement of a point in $\partial M$. Let $\overline T_M$ be the covering space from \cite[Definition 1.6]{HRWImmersed}, which is an infinite cylinder punctured in countably many points, with deck group $\Z$. Note that $\lambda$ lifts to a closed multicurve in $\overline T_M$, while $\mu$ does not. It is shown in \cite{HRWImmersed} that the bordered Floer homology of $M$ can be described as an immersed multicurve $\HFh(M) \subset T_M$, with a lift to $\overline T_M$ that is well-defined up to an overall deck transformation.

Our main computational result is the following:
\begin{proposition} \label{prop: HF(M)}
The bordered invariant $\HFh(M)$, viewed as an immersed multicurve in $\overline T_M$, is as shown on the left side of Figure \ref{fig: HF(M)}.
\end{proposition}

\begin{figure}
\subfigure{\includegraphics[width=1.5in]{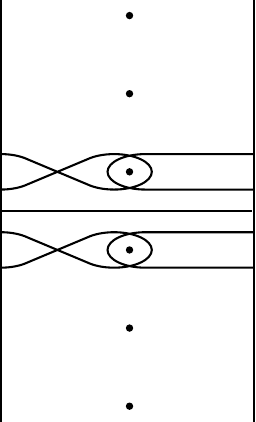}} \hspace{0.5in}
\subfigure{
\begin{tikzpicture}
	\begin{scope}[thin, black!20!white]
		\draw [<->] (-2, 0.5) -- (3, 0.5);
		\draw [<->] (0.5, -2) -- (0.5, 3);
	\end{scope}
	\draw[step=1, black!50!white, very thin] (-1.9, -1.9) grid (2.9, 2.9);
	\filldraw (0.3, 0.5) circle (2pt) node[] (A){};
	\filldraw (0.5, 0.5) circle (2pt) node[] (B){};
	\filldraw (0.7, 0.5) circle (2pt) node[] (C){};
	\filldraw (0.5, 0.3) circle (2pt) node[] (D){};
	\filldraw (0.5, 0.7) circle (2pt) node[] (E){};
	\filldraw (-0.5, 0.5) circle (2pt) node[] (F){};
	\filldraw (0.5, -0.5) circle (2pt) node[] (G){};
	\filldraw (1.5, 0.5) circle (2pt) node[] (H){};
	\filldraw (0.5, 1.5) circle (2pt) node[] (I){};

	\draw [very thick, ->] (A) -- (F);
	\draw [very thick, ->] (D) -- (G);
	\draw [very thick, ->] (H) -- (C);
	\draw [very thick, ->] (I) -- (E);
\end{tikzpicture}}
\caption{Left: The immersed multicurve $\HFh(M)$, viewed in the infinite cylinder $\overline T_M$ (viewed as a strip with sides identified). Right: The doubly filtered complex $\CFK_{UV=0}(Y,\gamma)$.}
\label{fig: HF(M)}
\end{figure}
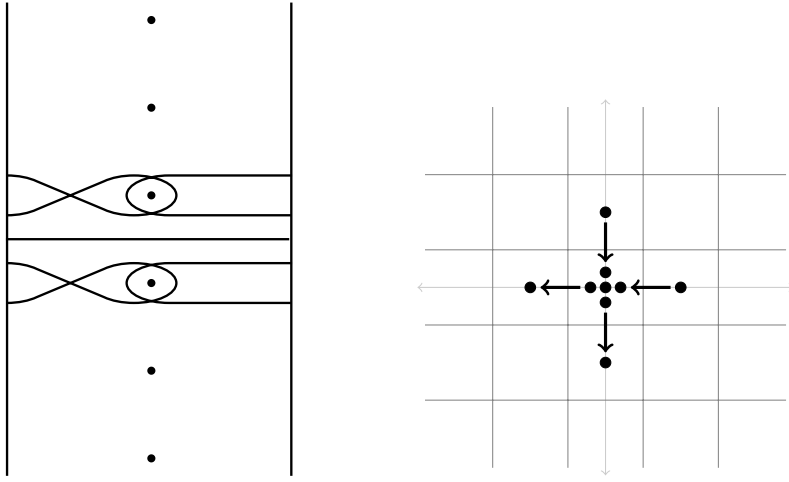

To prove Proposition \ref{prop: HF(M)}, we will rely on a computation from \cite{LevinePL}.  Let $L = L_1 \cup L_2$ denote the two-component link shown in the right frame of Figure \ref{fig:Mazur}, where $L_2$ is the yellow component. Because $L_2$ is the unknot, $0$-surgery on $L_2$ yields $S^1 \times S^2$. Let $L_1^* \subset S^1 \times S^2$ denote the knot obtained from $L_1$. Let $\lambda_L$ and $\mu_L$ denote the longitude and meridian of $L_1^*$ induced from the Seifert longitude and meridian of $L_1$.
% (Since $\lk(L_1,L_2)=1$, $L_1^*$ represents a nontrivial homology class in $S^1 \times S^2$, so it does not have a preferred choice of longitude when viewed as a knot in $S^1 \times S^2$.)

The following lemma is known to experts, eg. it is inherent in \cite{AkbulutKirbyMazur}; we include a proof for completeness.

\begin{lemma} \label{lemma: mazur}
There is a diffeomorphism from $M$ to $S^1 \times S^2 \minus \nbd(L_1^*)$ that takes $\lambda$ to $\mu_L$ and $\mu$ to $\lambda_L$.
\end{lemma}

\begin{figure}[htbp]{%\scriptsize
\begin{overpic}[abs, tics=20]{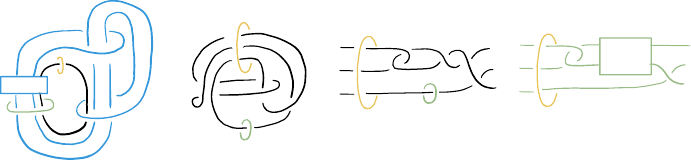}
\put(8,62){$\color{lblue}-1$}
\put(4,32){$\color{lblue}-1$}
\put(27,18){-$1$}
\put(30,33){$\color{lyellow}0$}
\put(0,15){$\color{lgreen}*$}
\put(127,65){\medmuskip=0mu
\thinmuskip=0mu
\thickmuskip=0mu $3$}
\put(100,63){$\color{lyellow}0$}
\put(112,4){$\color{lgreen}*$}
\put(200,57){\medmuskip=0mu
\thinmuskip=0mu
\thickmuskip=0mu $3$}
\put(163,57){$\color{lyellow}0$}
\put(208,22){$\color{lgreen}*$}
\put(296,46){\medmuskip=0mu
\thinmuskip=0mu
\thickmuskip=0mu $\color{lgreen}4$}
\put(270,60){$\color{lyellow}0$}
\put(300,23){$\color{lgreen}*$}
  \end{overpic}}
  \caption{A homeomorphism from $M$ to $S^1\times S^2\setminus \nu(L_1^*)$. Here we use an asterisk to denote a complement.}
  \label{fig:Mazur}
\end{figure}
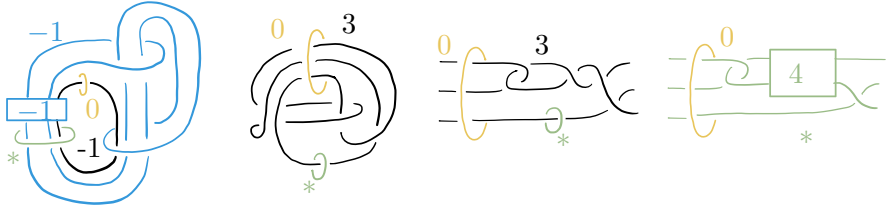

\begin{proof}
The left frame of Figure \ref{fig:Mazur} is $M$ drawn with an additional cancelling pair, and with $\lambda$ and $\mu$ given by the (blackboard) longitude and meridian of the green curve, respectively.  Notice that the black and yellow surgeries are canceling.  To obtain the second frame of Figure \ref{fig:Mazur}, we perform two slides of blue over black, then blow down the resulting blue unknot. The third frame is than obtained by an isotopy. The final frame is obtained by twisting the essential sphere in the $S^1\times S^2$ from the yellow surgery curve \cite[Figure 5.42]{GompfStipsicz} until the black surgery curve is $0$-framed, then isotoping the green complement into the black surgery solid torus.  Notice that this final isotopy exchanges the roles of meridians and longitudes on the green curve.
\end{proof}

\begin{proof}[Proof of Proposition \ref{prop: HF(M)}]
Let $X_L$ denote the exterior of $L$. The bordered bimodule $\CFDD(X_L)$ was computed in \cite{LevinePL} (in which $L$ is denoted $L_Q$). Following the notation used there, we will write this as a left-left bimodule over two copies of the torus algebra, which we denote by $\AA_\rho$ and $\AA_\sigma$. The bimodule $\CFDD(X_L)$ is generated by $\{g_1, \dots, g_{34}\}$, with associated idempotents and differential as specified in \cite[Theorem 3.4]{LevinePL} (which, for brevity, we do not reproduce here).

By Lemma \ref{lemma: mazur}, we may obtain $\CFD(M)$ by taking the tensor product of $\CFDD(X_L)$ with $\CFA$ of a suitably-framed solid torus $V_0$ that realizes the $S^1 \times S^2$ surgery on $L_2$:
\[
\CFD(M, \mu_L, \lambda_L) \simeq \CFA(V_0) \boxtimes_{\AA_\sigma} \CFDD(X_L).
\]
Here we are following the notation for the boundary parametrization of a bordered $3$-manifold used in \cite{HRWImmersed}. Because of the interchange of meridian and longitude in Lemma \ref{lemma: mazur}, this will actually describe $\CFD(M, \lambda, \mu)$, whereas the standard description of $\CFD$ of the complement of a knot in a homology sphere (e.g. in \cite[Chapter 11]{LOTBordered}) describes $\CFD(M, \mu, \lambda)$. However, the immersed curves package makes it easy to pass back and forth between the two parametrizations.

The bordered invariant for a solid torus framed for $0$-surgery, is as follows:
\begin{lemma} \label{lemma: CFA(V0)}
The type $A$ module for $\CFA(V_0)$ has a single generator $b$  in idempotent $\iota_1$, with $\AA_\infty$ multiplications given by
\begin{equation}
\label{eq: CFA(V0)} m_{3+i}(b, \sigma_2, \underbrace{\sigma_{12}, \dots, \sigma_{12}}_{i}, \sigma_1) = b
\end{equation}
for all $i \ge 0$.
\end{lemma}
(Compare \cite[Lemma 11.20]{LOTBordered}, which gives $\CFA$ for the framing of the solid torus that gives the meridional filling.)

Taking the box tensor product over $\AA_\sigma$, we find that $\CFD(M,\lambda,\mu)$ has a basis consisting of generators $x_i = g_i \otimes b$ for
\[
i \in \{1, 3, 6, 7, 8, 11, 12, 13, 14, 15, 17, 18, 19, 20, 22, 24, 26, 28, 30, 31, 33\},
\]
with associated idempotents inherited from the $\AA_\rho$ idempotents of the $g_i$ (see the table in \cite[Theorem 3.4]{LevinePL}). The differential can be obtained by the applying the box tensor product as follows. For clarity, terms shown in black come from the pure $\AA_\rho$ terms in the differential of $\CFDD(X_L)$; terms shown in red come from nontrivial sequences of the form $\sigma_2, \sigma_{12}, \dots, \sigma_{12}, \sigma_1$, using the higher $\AA_\infty$ multiplications from \eqref{eq: CFA(V0)}.

\begin{align*} \allowdisplaybreaks
d(x_1) &= \rho_1 x_{24} &
d(x_3) &= \rho_2 x_1 \\
d(x_6) &= \rho_2 x_{30} + {\color{red} \rho_{23} x_{14}} &
d(x_7) &= \rho_3 x_3 + \rho_1 x_{12} + \rho_{123} x_{24}  \\
d(x_8) &= 0 &
d(x_{11}) &= \rho_1 x_{17} + \rho_3 x_{28} \\
d(x_{12}) &= 0 &
d(x_{13}) &= \rho_3 x_{20} + {\color{red} \rho_{123} x_{12}}\\
d(x_{14}) &= 0 &
d(x_{15}) &= {\color{red} \rho_{123} x_{14}} \\
d(x_{17}) &= 0 &
d(x_{18}) &= \rho_3 x_{26}  + {\color{red} \rho_{123} x_{17}} \\
d(x_{19}) &= \rho_1 x_{14} &
d(x_{20}) &= \rho_{23} x_{6} + {\color{red} \rho_2 x_{30}} \\
d(x_{22}) &= {\color{red} x_8} &
d(x_{24}) &= 0 \\
d(x_{26}) &= \rho_2 x_{15} &
d(x_{28}) &= \rho_2 x_{19} \\
d(x_{30}) &= \rho_{123} x_{8} +  \rho_{123} x_{24} &
d(x_{31}) &= 0 \\
d(x_{33}) &= {\color{red} x_{31}} &&
\end{align*}

Now apply a change of basis, replacing $x_6$, $x_7$, and $x_{30}$ with
\begin{align*}
x_6' &= x_6 + x_{14} + x_{20} + \rho_2 x_{13} \\
x_7' &= x_7 + x_{30} + \rho_{123} x_{22} \\
x_{30}' &= x_{30} + \rho_3  x_6 + \rho_{123} x_{22}
\end{align*}
We can compute that the differentials of these elements are:
\begin{align*}
d(x_6') &= (\rho_2  x_{30} + \rho_{23}  x_{14}) + (\rho_{23} x_{6} + \rho_2  x_{30}) + \rho_2 (\rho_3 x_{20} +  \rho_{123} x_{12}) \\
%&= \rho_{23}( x_6 + x_{14} + x_{20} ) \\
%&= \rho_{23}( x_6 + x_{14} + x_{20} + \rho_2 x_{13} ) \\
&= \rho_{23} x_6' \\
d(x_7') &= (\rho_3 x_3 + \rho_1 x_{12} + \rho_{123} x_{24}) + (\rho_{123}  x_{8} +  \rho_{123} x_{24}) + \rho_{123} x_8 \\
&= \rho_3 x_3 + \rho_1 x_{12} \\
d(x_{30}') &= (\rho_{123} x_{8} +  \rho_{123} x_{24}) + \rho_3( \rho_2 x_{30} + \rho_{23} x_{14}) + \rho_{123} x_8 \\
&= \rho_{123} x_{24}.
\end{align*}
Additionally, using the new basis we have:
\begin{align*}
d(x_{20}) &= \rho_{23} x_6 + \rho_2 ( x_{30}' + \rho_3  x_6 + \rho_{123} x_{22}) = \rho_2 x_{30}'.
\end{align*}

We may cancel the trivial summands $x_{22} \to x_8$ and $x_{33} \to x_{31}$, and observe that the remaining generators may be arranged graphically as follows:
\begin{equation} \label{eq: CFD(M)}
\xymatrix{
&& x_{14} & x_{19} \ar[l]_{\rho_1} & x_{28} \ar[l]_{\rho_2} &&  x_{24} & x_1 \ar[l]_{\rho_1} & x_3 \ar[l]_{\rho_2} \\
x_6' \ar@(ul,ur)[]^{\rho_{23}} && x_{15} \ar[u]^{\rho_{123}} &        & x_{11} \ar[u]_{\rho_3} \ar[d]^{\rho_1}  &&  x_{30}' \ar[u]^{\rho_{123}} &  & x_7' \ar[u]_{\rho_3} \ar[d]^{\rho_1} \\
&& x_{26} \ar[u]^{\rho_2} & x_{18} \ar[r]_{\rho_{123}} \ar[l]^{\rho_3} & x_{17}  &&  x_{20} \ar[u]^{\rho_2} & x_{13} \ar[r]_{\rho_{123}} \ar[l]^{\rho_3} & x_{12}
}
\end{equation}
Let $A_1, A_2, A_3$ denote these three summands, and note that $A_2$ and $A_3$ are isomorphic to each other, while $A_1$ is isomorphic to $\CFD$ of a solid torus.

Next, we obtain the immersed multicurve $\HFh(M) \subset T_M$ following the procedure from \cite[Section 2.4]{HRWImmersed}, as shown in Figure \ref{fig: HF(M)-torus}. Let $\gamma_1, \gamma_2, \gamma_3$ denote the components corresponding to the three summands in \eqref{eq: CFD(M)}. (For clarity, we have omitted $\gamma_3$, which simply runs parallel to $\gamma_2$.)
Then $\gamma_1$ is a loose embedded curve (in the terminology of \cite[Section 7.1]{HRWImmersed}), while each of the other two summands in \eqref{eq: CFD(M)} produces a figure-eight-shaped curve.

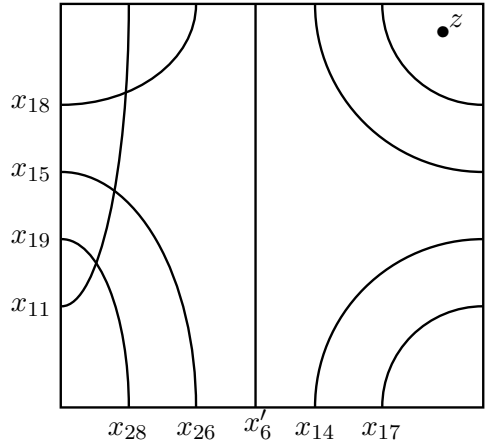
\begin{figure}
\labellist
\pinlabel $x_{28}$ [b] at 40 -5
\pinlabel $x_{26}$ [b] at 69 -5
\pinlabel $x_6'$ [b] at 95 -5
\pinlabel $x_{14}$ [b] at 119 -5
\pinlabel $x_{17}$ [b] at 147 -5
\pinlabel $x_{15}$ [r] at 11 111
\pinlabel $x_{19}$ [r] at 11 83
\pinlabel $x_{11}$ [r] at 11 54
\pinlabel $x_{18}$ [r] at 11 140
\pinlabel $\bullet$ at 173 170
\pinlabel $z$ [bl] at 173 170
\endlabellist
\includegraphics[width=2.5in]{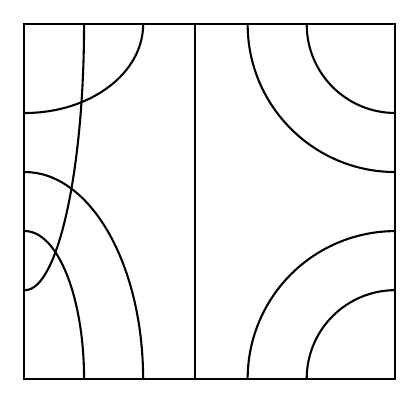}
\caption{Two components of the immersed curve $\HFh(M)$, drawn in the punctured torus $T_M$, which is viewed as a square with standard edge identifications.}
\label{fig: HF(M)-torus}
\end{figure}

We obtain the immersed curve for $\CFD(M,\mu,\lambda)$ by rotating the entire picture $90^\circ$ (in either direction). This corresponds to the conventions for $\HFh$ of knot complements used in \cite{HRWProperties}, where the homological longitude is drawn horizontally.

To lift $\HFh(M)$ to the cover $\overline T_M$, we must simply verify that the preferred lifts of $\gamma_2$ and $\gamma_3$ are at different heights. By \cite[Section 2]{HRWImmersed}, it suffices to show that for any pair of corresponding generators of $A_2$ and $A_3$ (e.g. $x_{14}$ and $x_{24}$), their spin$^c$ gradings differ by the Poincar\'e dual of $\mu$. We may recover the relative gradings between any two generators of $\CFD(M,\mu,\lambda)$ directly from the computations of the differential above (before the change of basis that produced \eqref{eq: CFD(M)}). As this is not directly required for the applications in this paper, we leave it as an exercise to the interested reader.
\end{proof}

\begin{figure}
\labellist
\pinlabel $h(\gamma_2)$ [l] at 76 114
\pinlabel $h(\gamma_1)$ [r] at 114 99
\pinlabel {{\color{red} $\HFh(X_K)$}} [l] at 20 30
\endlabellist
\includegraphics[width=2.5in]{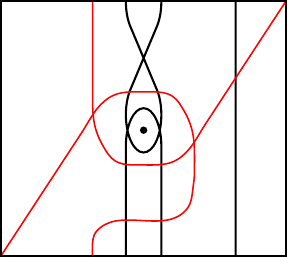}
\caption{Computation of $\HF(\HFh(X_K), h(\HFh(M)))$, where $K$ is the right-handed trefoil. The component $h(\gamma_3)$, which runs parallel to $h(\gamma_2)$, is not shown.} \label{fig:tensor}
\end{figure}

\begin{proof}[Proof of Theorem \ref{thm:knot-surgery}]
For any knot $K \subset S^3$, let $X_K$ denote the exterior of $K$. Let $h \co \partial M \to \partial X_K$ be the gluing homeomorphism that produces the manifold $Y_K$, as described above. As stated previously, $h$ takes the longitude of $M$ to the meridian of $K$, and vice versa. Let $\HFh(X_K)$ denote the immersed curve invariant of $X_K$, viewed as an immersed curve in $\partial X_K - \{z\}$. This curve can be obtained from $\CFKm(K)$ via the procedure in \cite[Section 4]{HRWProperties}.

Following the conventions of \cite{HRWImmersed}, we will draw $\partial X_K$ as a square with sides identified, such that $\lambda_K$ is a horizontal line and $\mu_K$ is a vertical line. See Figure \ref{fig:tensor} for an example, in which $K$ is the right-handed trefoil.

By \cite[Theorem 1.2]{HRWImmersed}, we have
\[
\HFh(Y_K) \cong \HF( \HFh(X_K), h(\HFh(M)) ),
\]
where $\HF(\cdot, \cdot)$ denotes the version of Lagrangian intersection Floer homology described in \cite{HRWImmersed}.
In particular, $h(\gamma_1)$ can be drawn as a vertical line segment away from the puncture, while and $h(\gamma_2)$ and $h(\gamma_3)$ are each figure-eight shaped curves contained within a small vertical strip $A$ containing the puncture.

Since $\gamma_1$ is the same as $\HF$ of a solid torus, $\HF(\HFh(X_K), h(\gamma_1))$ computes $\HFh$ of the $\infty$ filling of $K$, which is $S^3$. Thus,
\[
\dim \HF(\HFh(X_K), h(\gamma_1)) = 1.
\]
In other words, we may isotope $\HFh(X_K)$ to meet $h(\gamma_1)$ in a single point.

Note that $\HFh(X_K)$ may have components that are isotopic to $h(\gamma_2)$, which \emph{a priori} would necessitate consideration of local systems. Indeed, this is extremely common: any ``unit box'' in $\CFKm(K)$ gives rise to such a component, as in the example of the figure-eight knot \cite[Figure 10]{HRWImmersed}. Here, a unit box refers to a set of four elements $\{a,b,c,d\}$ in a reduced basis for $\CFKm(K)$ with the property that there are length-1 horizontal arrows from $a$ to $c$ and from $b$ to $d$, and length-1 vertical arrows from $a$ to $b$ and from $c$ to $d$. By \cite[Lemma 12.4]{HanselmanHFK}, any unit box actually splits off as a direct summand of $\CFKm(K)$, with no other horizontal or vertical arrows into or out of either generator after a change of basis. Thus, when we apply the $\CFKm$-to-$\CFD$-to-curves procedure of \cite[Section 4]{HRWProperties}, we obtain a figure eight curve with no crossover arrows to any component; in particular, its local system is trivial. (In Hanselman's interpretation of $\CFKm$ as immersed curves, this is the content of \cite[Lemma 12.5]{HanselmanHFK}.) Thus, local systems are not needed, and $\dim \HF(\HFh(X_K), h(\gamma_2))$ is equal to the minimal intersection number of $\HFh(X_K)$ with $h(\gamma_2)$.

By construction (specifically \cite[Proposition 48]{HRWProperties}), we may arrange that $\HFh(X_K)$ intersects the strip $A$ in $\dim \HFKh(K)$ transversal arcs (counted with multiplicity if nontrivial local systems are present), each of which meets $h(\gamma_2)$  twice. In this setup, $\HFh(X_K)$ intersects $h(\gamma_2)$ in minimal position. Thus,
\[
\dim \HF(\HFh(X_K), h(\gamma_2)) = \dim \HF(\HFh(X_K), h(\gamma_3)) = 2 \dim \HFKh(K).
\]
Adding up the contributions from the three components of $\HFh(M)$, we have
\[
\dim \HFh(Y_K) = 4 \dim \HFKh(K) + 1,
\]
as required.
\end{proof}

\begin{remark}
By \cite[Proposition 7.11]{HRWImmersed}, $M$ is a Heegaard Floer homology solid torus, which means that $\HFh(M)$ is unchanged under Dehn twists parallel to $\lambda$. That is, $\CFD(M, \mu, \lambda) \cong \CFD(M, \mu + n\lambda, \lambda)$ for all $n \in \Z$. By chasing through the proof of Lemma \ref{lemma: mazur}, one may verify that gluing $(X_K, \mu_K, \lambda_K)$ to $(M, \mu+n\lambda, \lambda)$ produces $S^3_1(P_{n-1}^K)$. This gives a confirmation that $\HFh(S^3_1(P_n^K))$ is independent of $n$, as stated above.

To the authors' knowledge, $M$ is the first known example of a Heegaard Floer homology solid torus that has no L-space fillings.
\end{remark}

\begin{remark} \label{rem:CFK(gamma)}
By recent work of Hanselman \cite{HanselmanHFK}, for a knot $J$ in a homology sphere $Y$, the knot Floer complex $\CFK_{UV=0}(Y,J)$ can also be encoded as an immersed multicurve in $\overline T_M$, denoted as $\HFm(Y,J)$, possibly with local systems. This is a two-way street; one can recover $\CFK_{UV=0}(Y,J)$ from $\HFm(Y,J)$ by taking the Floer homology of $\HFm(Y,J)$ with a vertical line passing through the punctures, taking a pair of basepoints on each side of each puncture. (See \cite[Section 5.2]{HanselmanHFK} for details). Furthermore, when $Y=S^3$, Hanselman shows that $\HFm(Y,J)$ is precisely $\HFh(M)$, where $M = Y - \nbd(J)$, thanks to the $\CFKm$-to-$\CFD$-to-curves procedure alluded to above. Forthcoming work of Hanselman and Levine will show that the same is true for an arbitrary 3-manifold $Y$, and not just for $Y = S^3$.

Applying this idea for our $(Y,\gamma)$ from above, we deduce that the multicurve on the left side of Figure \ref{fig: HF(M)} is precisely $\HFm(Y,\gamma)$. By the procedure just described, one can deduce that $\CFK_{UV=0}(Y,\gamma)$ has the form shown on the right side of Figure \ref{fig: HF(M)}.
\end{remark}

\bibliography{LLP-bibliography}
\bibliographystyle{amsalpha}

\end{document}